\documentclass[12pt,reqno]{article}
\usepackage{amsmath,amssymb,amsthm}
\usepackage{color}
\newtheorem{theorem}{Theorem}[section]
\newtheorem{proposition}[theorem]{Proposition}
\newtheorem{corollary}[theorem]{Corollary}
\newtheorem{definition}[theorem]{Definition}
\newtheorem{lemma}[theorem]{Lemma}
\newtheorem{example}[theorem]{Example}
\newtheorem{remark}[theorem]{Remark}
\newtheorem{problem}[theorem]{Problem}
\numberwithin{equation}{section}

\usepackage[dvipdfmx]{hyperref}

\def\bC{\mathbb{C}}
\def\bN{\mathbb{N}}
\def\bR{\mathbb{R}}
\def\cH{\mathcal{H}}
\def\cD{\mathcal{D}}
\def\cK{\mathcal{K}}
\def\cM{\mathcal{M}}
\def\cR{\mathcal{R}}
\def\fS{\mathfrak{S}}
\def\Inf{\mathop{\mathrm{Inf}}}
\def\eps{\varepsilon}
\def\<{\langle}
\def\>{\rangle}
\def\ffi{\varphi}
\def\Tr{\mathrm{Tr}}
\def\tr{\mathrm{tr}}
\def\Re{\mathrm{Re}\,}
\def\Ad{\mathrm{Ad}}
\def\Im{\mathrm{Im}\,}

\newcommand{\rmvpopsi}{\Delta_{\vpo\psi}}
\newcommand{\rmvppsi}{\Delta_{\varphi\psi}}
\newcommand{\rmvpvpo}{\Delta_{\varphi\vpo}}
\newcommand{\vpo}{\varphi_0}
\newcommand{\tp}{\tilde{p}}
\newcommand{\ellp}{L^p({\cM},\psi')}
\newcommand{\ellq}{L^q({\cM},\psi')}
\date{}
\begin{document}

\centerline{\Large Connections of unbounded operators and some related topics:}
\medskip
\centerline{\Large von Neumann algebra case}

\bigskip
\bigskip
\centerline{\large
Fumio Hiai\footnote{{\it E-mail:} hiai.fumio@gmail.com},
and Hideki Kosaki\footnote{{\it E-mail:} kosaki@math.kyushu-u.ac.jp}}

\medskip
\begin{center}
$^1$\,Graduate School of Information Sciences, Tohoku University, \\
Aoba-ku, Sendai 980-8579, Japan
\end{center}

\begin{center}
$^2$\,Graduate School of Mathematics, Kyushu University, \\
Nishi-ku, Fukuoka 819-0395, Japan
\end{center}

\bigskip

\begin{abstract}

The Kubo-Ando theory deals with connections for positive bounded operators.
On the other hand, in various analysis related to von Neumann algebras
it is impossible to avoid unbounded operators. 
In this article we try to extend a notion of connections to cover various classes 
of positive unbounded operators (or unbounded objects such as positive forms 
and weights) appearing naturally in the setting of von Neumann algebras, and
we must keep all the expected properties maintained.
This generalization is carried out for the following classes:
(i) positive $\tau$-measurable operators (affiliated with a semi-finite von Neumann algebra
equipped with a trace $\tau$),
(ii) positive elements in Haagerup's $L^p$-spaces,
(iii) semi-finite normal weights on a von Neumann algebra.
Investigation on these generalizations requires some analysis
(such as certain upper semi-continuity) on decreasing sequences 
in various classes.
Several results in this direction are proved, which may be of independent interest.
Ando studied Lebesgue decomposition for positive bounded operators
by making use of parallel sums. Here, such decomposition is obtained in the setting 
of non-commutative (Hilsum) $L^p$-spaces.


\medskip\noindent
{\it Key words and phrases:}

\noindent
closable operator,
connection,
Connes' spatial theory,
extended positive self-adjoint operator,
form sum,
$\gamma$-homogeneous operator,
generalized $s$-number,
geometric mean,
graph analysis,
Haagerup $L^p$-space,
Hilsum $L^p$-space,
Kubo-Ando theory,
Lebesgue decomposition,
measure topology,
modular automorphism group,
operator monotone function,
operator valued weight,
parallel sum,
positive form,
positive self-adjoint operator,
Radon-Nikodym cocycle,
relative modular operator,
spatial derivative, 
strong resolvent convergence,
$\tau$-measurable operator,
trace,
von Neumann algebra,
weight

\end{abstract}

{\baselineskip11pt
\tableofcontents
}

%%%%%%%%%%%%%%%%%%%%%%%%%%%%%%%%%%%%%%%%%%
%%%%%%%%%%%%%%% Introduction %%%%%%%%%%%%%%%%%%%
%%%%%%%%%%%%%%%%%%%%%%%%%%%%%%%%%%%%%%%%%%
\section{Introduction}\label{S-1}

Operator means in various settings have been studied in operator theory
and operator algebras, and they sometimes play unexpectedly important roles.
Studies of parallel sums in \cite{AD,AT} and geometric means in \cite{PW}
are such examples (and surely many others that the authors are 
unaware of). 
Also many properties of typical operator means such  as parallel sums 
(twice of harmonic means) and
geometric means were completely clarified in \cite{An3}. 
Motivated by these pioneering works, the so-called Kubo-Ando theory 
was established  in \cite{KA}.
In this theory a notion of operator means (or more generally connections) 
of positive operators was introduced in an axiomatic fashion.
They are in a one-to-one correspondence with the set of all non-negative operator monotone 
functions on $\bR_+$ (see the beginning of \S\ref{S-2.2}), and parallel sums
serve as building blocks in the sense that each connection 
is expressed as a certain integral of parallel sums
with particular parametrization (see \eqref{F-2.8}).
In fact, this integral is closely related to 
the well-known integral representation of an operator monotone function 
(\cite{Bh,Do}) corresponding to a connection in question.
This subject has attracted many specialists and there is a considerable amount 
of literature on this subject matter. It is simply impossible to list up 
important contributions (although some of them can be found in \cite{An4,Bh,Bh2,Hi2}
for instance). 

In the Kubo-Ando theory connections are defined for pairs of positive matrices
or bounded positive operators (on a Hilbert space). 
On the other hand, our main concern here is analysis 
in von Neumann algebras, where use of unbounded operators is unavoidable. 
For example unbounded operators have to enter into
study of non-commutative $L^p$-spaces. 
The main purpose of the article is to make a satisfactory theory of connections for
various classes of unbounded operators (or unbounded objects such as weights) 
appearing naturally in analysis related to von Neumann algebras. 
Generally unbounded operators are much more difficult to handle compared with bounded ones
since delicate problems such as domain questions and closability problems 
(under algebraic operations) have to be taken care of.
For a semi-finite von Neumann algebra  $\cM$ with a trace $\tau$ one can introduce the class
of $\tau$-measurable operators (see the beginning of \S\ref{S-3.1}). 
Operators in this  class
behave  in a perfectly reasonable way so that one can handle them without worrying about 
pathological phenomena.
This class turns out to be large enough to consider non-commutative $L^p$-spaces 
(see \cite{Di,Ku,Y} for early development of the theory).
The notion of $\tau$-measurability was originally given in \cite{Se} (see also \cite{Sti} 
for related topics) in a slightly different form, but the approach in \cite{Ne} is more convenient
to our purpose. Thus, the latter approach is employed, and we will 
use \cite{Ne, Te1} as our standard references on the subject.  
When a von Neumann algebra is finite, all unbounded closed operators 
(affiliated with $\cM$) is automatically $\tau$-measurable. 
This means that manipulations of unbounded operators in this case are 
completely under control, 
which was already shown by Murray-von Neumann since the beginning of the subject.
The recent articles \cite{DDSZ1,DDSZ2} have somewhat similar nature in the sense
that certain means for $\tau$-measurable operators are investigated. 
However, what are studied there are means in the sense of \cite{HK1,HK2} 
(which are different from Kubo-Ando means \cite{KA} considered here).
Hence there is no direct relation between \cite{DDSZ1,DDSZ2} and our present work, 
but some of standard techniques for $\tau$-measurable operators 
 (such as devices in \cite{FK}) are in common use.

We explain contents of later sections briefly, and more detailed description  
is given at the beginning of each section. 
The most general situation is treated in \S\ref{S-2}.
Namely, we  consider positive forms here, which are roughly positive self-adjoint operators
with possibly non-dense domains, or equivalently, elements in the extended positive part 
$\widehat{B(\cH)}_+$ (see \cite{Haa3}).
A general theory for parallel sums for positive forms was worked out in \cite{Ko6}.
Thus, by mimicking an integral expression of a connection (in terms of parallel sums) 
one can define connections of positive forms. 
Some properties of connections based on this naive definition are discussed in \S\ref{S-2.2}.

In \S\ref{S-3} we consider a semi-finite von Neumann algebra $\cM$ with a trace $\tau$, and
study connections of $\tau$-measurable operators. 
The definition of a $\tau$-measurable operator and some related topics are explained at the
beginning of \S\ref{S-3.1}.
Two definitions of connections are possible. We obviously have the naive  
definition (Definition \ref{D-3.15}) as a positive form (explained in \S\ref{S-2}), 
and another competing definition (Definition \ref{D-3.16}) motivated 
by the case of  bounded operators is also possible. We will show that the two definitions
actually coincide (in \S\ref{S-3.2}). 
In the Kubo-Ando theory 
decreasing sequences of positive operators appear quite often, for which convergence in the strong 
operator topology is usually used. 
Similarly, we encounter decreasing sequences of positive $\tau$-measurable 
operators in this section, and we need some general results on these matters.
They are collected in \S\ref{S-3.1}, and results such as Theorems \ref{T-3.8} 
and \ref{T-3.10} there might be of independent interest.

In \S\ref{S-4} we deal with a general von Neumann algebra and the associated Haagerup $L^p$-spaces 
$L^p(\cM)$ (\cite{Haa4}). We study connections in these $L^p$-spaces. 
(A similar study though specialized to connections in
$L^1(M)$ is also given in \cite[Appndix D]{Hi3}.)
The crossed product $\cM \rtimes_{\sigma} \bR$ relative to a modular action $\sigma$ is a semi-finite
algebra equipped with a so-called canonical trace $\tau$ 
(which is a standard fact in structure analysis of type III von Neumann algebras). 
The Haagerup $L^p$-spaces are constructed by making use of this semi-finite algebra. 
Since $\tau$-measurability and the measure topology there are (technically) important 
ingredients in Haagerup's theory (see Appendix \ref{S-A} for details), 
some of arguments in this section depend heavily upon those in \S\ref{S-3}.
Many interesting properties of connections were known in \cite{KA}. 
We show (besides some special topics in \S\ref{S-3.4} and \S\ref{S-3.5})
that these properties
remain to hold true in our more general settings (in \S\ref{S-3} and \S\ref{S-4}).

The Haagerup $L^1(\cM)$ can be identified with the predual $\cM_*$, meaning that a reasonable theory 
for normal positive linear functionals on $\cM$ is at our disposal.
More generally, in \S\ref{S-5} we investigate a notion of parallel sums for semi-finite normal weights on $\cM$.
To this end, we use Connes' spatial theory \cite{C3} and
it is achieved by considering parallel sums (in the sense of \cite{Ko6}) of spatial derivatives arising
from weights in question. In this study a notion of $(-1)$-homogeneity (see the beginning of \S\ref{S-5.1}
or Definition \ref{D-C.1}) plays an essential role.
Once parallel sums are defined, we can then define connections of weights in the usual way, 
which is briefly discussed in \S\ref{S-5.3}. 

Ando studied Lebesgue decomposition for (bounded) positive operators in \cite{An1}.
His work is based on ingenious use of parallel sums, and the (maximal) absolutely continuous part in
Lebesgue decomposition is captured as the limit of a certain sequence expressed in terms of parallel sums.
The final \S\ref{S-6} deals with Lebesgue decomposition in the setting of non-commutative $L^p$-spaces.
Since a reasonable notion of parallel sums has been prepared in \cite{Ko6} and \S\ref{S-3}, \S\ref{S-4},
we can play the same game by modifying arguments in \cite{An1}.
However, we take a different approach akin to that in \cite{Ko2,Ko3}.
Namely, based on graph analysis for relevant relative modular operators we express
Lebesgue decomposition in a more explicit manner. For this purpose  Hilsum's $L^p$-spaces 
(which are isometrically isomorphic to $L^p(\cM)$) are more convenient so that we use his $L^p$-spaces
in \S\ref{S-6}. In fact, for $\cM$ in its standard representation the positive parts of these $L^p$-spaces
consist of powers of relative modular operators. Results in this section generalize those in \cite{Ko3} 
(dealing with the predual $\cM_* \cong L^1(\cM)$).

We have three appendices for the reader's convenience.
In Appendix \ref{S-A} basic facts on the Haagerup $L^p$-spaces (used in \S\ref{S-4}) are reviewed.
Appendix \ref{S-B} explains several materials related to Connes' spatial theory \cite{C3}.
At first Connes' notion of spatial derivatives is explained in \S\ref{S-B.1} 
(via a slightly different approach presented in \cite{Te1}) and then \S\ref{S-B.2}
explains Hilsum's $L^p$-spaces.
Materials in these two subsections are needed in \S\ref{S-5} and \S\ref{S-6} respectively. 
In \cite{C3} Connes pointed out a canonical order-reversing correspondence 
between weights on $\cM$ (acting on a Hilbert space $\cH$) 
and operator-valued weights from $B(\cH)$ to the commutant $\cM'$. We review this correspondence
in \S\ref{S-B.3}. Parallel sums of weights in \S\ref{S-5} (and also the concept to be explained 
in Appendix \ref{S-C}) admit natural interpretation (see Theorem \ref{T-5.10} for instance) 
in terms of this correspondence, which may be of some interest. 
In this article we are forced to deal with decreasing sequences in many contexts 
such as positives forms, positive $\tau$-measurable operators
and so on. In the last appendix (Appendix \ref{S-C}) we study decreasing sequences
of (normal) weights.  Probably results here have never been recorded in literature.

%%%%%%%%%%%%%%%%%%%%%%%%%%%%%%%%%%%%%%%%%%
%%%%%%%%%%%%%% Positive forms %%%%%%%%%%%%%%%%%%%
%%%%%%%%%%%%%%%%%%%%%%%%%%%%%%%%%%%%%%%%%%
\section{Positive forms and their connections}\label{S-2}

We will study connections for various classes of unbounded operators
(such as $\tau$-measurable positive operators and so on) appearing naturally in analysis 
with von Neumann algebras.
Therefore, we begin with the most general situation, i.e.,  (unbounded) positive self-adjoint 
operators. 
Actually, the notion of positive forms in \cite{Ko6} is more convenient for our purpose.
In \S\ref{S-2.1} we will quickly review basic facts on positive forms, which are roughly
positive self-adjoint operators with non-dense (or dense) domains. 
The Kubo-Ando theory \cite{KA} deals with connections for bounded positive operators
in an axiomatic fashion and establishes a one-to-one correspondence 
between connections and non-negative operator monotone functions on $\bR_+$
(see the beginning of \S\ref{S-2.2}).
Since an operator monotone function admits an integral representation (\cite{Bh,Do}),
each connection can be expressed as a certain integral of parallel sums 
(see \eqref{F-2.7} and \eqref{F-2.8} for details).

A notion of parallel sums for positive forms was studied in \cite{Ko6}, which enables us to
introduce connections of positive forms 
similarly (see also \cite[Appendix B]{Ko7}).
In \cite{KA} the transpose and adjoint of a connection are important operations.
In \S\ref{S-2.2} we will explain connections for positive forms at first and then examine 
these two operations in the case of positive forms.

%%%%%%%%%%%%%%%%%%%%%%%%%%%%%%%%%%%%%%%%%%%%%%%%
%%%%%%%%%%% Positive forms %%%%%%%%%%%%%%%%%%%%%%%%%%%%%
%%%%%%%%%%%%%%%%%%%%%%%%%%%%%%%%%%%%%%%%%%%%%%%%
\subsection{Positive forms}\label{S-2.1}

This section is a brief summary of \cite{Ko6} to fix ideas and notations for later use.
Our standard references for basic facts on unbounded operators are \cite{Ka2,Sch, St}.
Let $\cH$ be a Hilbert space (assumed infinite-dimensional). By a \emph{positive form} $q$ on
$\cH$ we mean a function $q:\cH\to[0,+\infty]$ satisfying
\begin{itemize}
\item[(i)] $q(\lambda\xi)=|\lambda|^2q(\xi)$ for all $\xi\in\cH$ and $\lambda\in\bC$ (with the
convention $0\cdot\infty=0$),
\item[(ii)] $q(\xi_1+\xi_2)+q(\xi_1-\xi_2)=2q(\xi_1)+2q(\xi_2)$ for all $\xi_1,\xi_2\in\cH$,
\item[(iii)] $q$ is lower semi-continuous on $\cH$ (i.e.,
$q(\xi)\le\liminf_{n\to\infty}q(\xi_n)$ whenever $\xi_n\to\xi$).
\end{itemize}

The domain $\cD(q)$ of $q$ is given by $\cD(q):=\{\xi\in\cH;\,q(\xi)<\infty\}$, which is
obviously a linear subspace of $\cH$. We note that a positive form $q$ on $\cH$ bijectively
corresponds to an \emph{extended positive self-adjoint operator} $h$ in the sense of
Haagerup \cite{Haa3} or Kato \cite{Ka1}, i.e., a positive self-adjoint operator $h$ on some
(closed) subspace $\cK$ of $\cH$. We write $\widehat{B(\cH)}_+$, following \cite{Haa3}, for
the set of such extended positive self-adjoint operators on $\cH$. The correspondence
$q\leftrightarrow h$ is given in such a way that $\cD(q)=\cD(h^{1/2})$ (so
$\overline{\cD(q)}=\cK$) and $q(\xi)=\|h^{1/2}\xi\|^2$ for $\xi\in\cD(q)$. In terms of the
spectral decomposition $h=\int_0^\infty\lambda\,de_\lambda$ on $\cK$ (so
$e_\infty:=\lim_{\lambda\to\infty}e_\lambda$ is the projection onto $\cK$), we can also write
\begin{align}\label{F-2.1} %%%%%%%%%%%%%%%%%%%%%%%%%%%%%%% \label{F-2.1}
q(\xi)=\int_0^\infty\lambda\,d\|e_\lambda\xi\|^2+\infty\|e_\infty^\perp\xi\|^2,
\qquad\xi\in\cH
\end{align}
(with the convention $\infty\cdot0=0$). We write $q=q_h$ in this case. For each
$h\in\widehat{B(\cH)}_+$ we can define the resolvent $(1+h)^{-1}$, which is understood to be
$0$ on $\cK^\perp=\cD(q)^\perp$.

For positive forms $q_1,q_2$ corresponding to $h_1,h_2\in\widehat{B(\cH)}_+$ respectively, we
note that $q_1\le q_2$, i.e., $q_1(\xi)\le q_2(\xi)$ for all $\xi\in\cH$ if and only if
$h_1\le h_2$ in the form sense, i.e., $\cD(h_2^{1/2})\subseteq\cD(h_1^{1/2})$ and
$\|h_1^{1/2}\xi\|\le\|h_2^{1/2}\xi\|$ for all $\xi\in\cD(h_2^{1/2})$. Moreover, it is well-known
that $h_1\le h_2$ in the form sense if and only if $(1+h_1)^{-1}\ge(1+h_2)^{-1}$. The
(point-wise) sum $q_1+q_2$ is a positive form, which corresponds to the so-called
\emph{form sum} $h_1\,\dot+\,h_2$ \cite{Ka2}.
 Let $q$ be a positive form corresponding to
$h=\int_0^\infty\lambda\,de_\lambda$ (on $\cK$). Let $\cK_0$ be the null space of $h$ or
$\cK_0=\{\xi\in\cH;\,q(\xi)=0\}$. We can define $k\in\widehat{B(\cH)}_+$ by
$k:=\int_{(0,\infty)}\lambda^{-1}\,de_\lambda+0e_\infty^\perp$ (on $\cK_0^\perp$). Then the
\emph{inverse} $q^{-1}$ of $q$ is defined as the positive form corresponding to $k$, that is,
$$
q^{-1}(\xi)=\int_{[0,\infty)}\lambda^{-1}\,d\|e_\lambda\xi\|^2
\ \biggl(=\int_{(0,\infty)}\lambda^{-1}\,d\|e_\lambda\xi\|^2+\infty\|e_0\xi\|^2\biggr),
\qquad\xi\in\cH.
$$
The expression of $q^{-1}$ in terms of the pure positive form notation was shown in
\cite[Lemma 2]{Ko6} as
\begin{equation}\label{F-2.2} %%%%%%%%%%%%%%%%%%%%%%%%%%%%%\label{F-2.2}
q^{-1}(\xi)=\sup_{\zeta\in\cH}{|(\xi,\zeta)|^2\over q(\zeta)}
\end{equation}
with the convention $0/0=0$, $\alpha/0=\infty$ ($\alpha>0$).
Obviously, $q \mapsto q^{-1}$ is order-reversing and $(q^{-1})^{-1}=q$.

The parallel sum of two positive forms was introduced in \cite{Ko6}.

\begin{definition}\label{D-2.1}\rm
For positive forms $\phi,\psi$ define the \emph{parallel sum} $\phi:\psi$ by
$$
\phi:\psi:=(\phi^{-1}+\psi^{-1})^{-1}.
$$
\end{definition}

\begin{theorem}[\cite{Ko6}]\label{T-2.2}
The parallel sum $\phi:\psi$ is the maximum of all the positive forms $q$ satisfying
$$
q(\xi_1+\xi_2)\le\phi(\xi_1)+\psi(\xi_2),\qquad\xi_1,\xi_2\in\cH.
$$
\end{theorem}

In fact, we set
\begin{align}\label{F-2.3}
\rho_0(\xi):=\inf\{\phi(\xi_1)+\psi(\xi_2);\,\xi=\xi_1+\xi_2\}.
\end{align}
Then $\rho_0$ is \emph{quadratic}, that is, $\rho_0$ satisfies (i) and (ii) of the definition
of positive forms. The above theorem says that $\phi:\psi$ is the maximum of all the positive
forms $q$ satisfying $q(\xi)\le\rho_0(\xi)$ for all $\xi\in\cH$. Hence we can write
\cite{Si1,Si2}
\begin{align}\label{F-2.4}%%%%%%%%%%%%%%%%%%%%%%%%%%%%%%% \label{F-2.4}
(\phi:\psi)(\xi)=\inf_{\xi_n\to\xi}\liminf_{n\to\infty}\rho_0(\xi_n).
\end{align}

If $\{q_n\}$ is an increasing sequence of positive forms, 
then the point-wise supremum $q=\sup_nq_n$ is
a positive form. On the other hand, when $\{q_n\}$ is decreasing, the point-wise infimum
$\inf_nq_n$ is quadratic but not necessarily lower semi-continuous. We have the maximum
of all the positive forms $q$ satisfying $q(\xi)\le\inf_nq_n(\xi)$ for all $\xi\in\cH$ \cite{Si1,Si2},
which is denoted by $\Inf_nq_n$.
Note that $\{q_n^{-1}\}$ is increasing (in the decreasing case)
and consequently the point-wise supremum $\sup_n q_n^{-1}$ is a positive form as mentioned above.
We actually have 
\begin{equation}\label{F-2.5} %%%%%%%%%%%%%%%%%%%%%%%%%%%%%% \label{F-2.5}
\Inf_n q_n=\left(\sup_n q_n^{-1}\right)^{-1}
\end{equation}
(see \cite[Lemma 16]{Ko6}).
The following was given in \cite[Corollary 17]{Ko6}:

\begin{theorem}\label{T-2.3}
If $\{\phi_n\}$ and $\{\psi_n\}$ are decreasing sequences of positive forms, then
\begin{align}\label{F-2.6} %%%%%%%%%%%%%%%%%%%%%%%%%%%%%%%% \label{F-2.6}
\Inf_n(\phi_n:\psi_n)=\Bigl(\Inf_n\phi_n\Bigr):\Bigl(\Inf_n\psi_n\Bigr).
\end{align}
\end{theorem}

This result is obtained by repeated use of \eqref{F-2.5}. Indeed, we observe 
$$
\left(\Inf_n \phi_n\right)^{-1}+\left(\Inf_n \psi_n\right)^{-1}
=\sup_n \phi_n^{-1}+\sup_n\psi_n^{-1}=\sup_n\left(\phi_n^{-1}+\psi_n^{-1}\right)
=\sup_n \,(\phi_n:\psi_n)^{-1}
$$
with the decreasing sequence $\{\phi_n:\psi_n\}$ of parallel sums 
so that one can just take the inverses of the both sides.

The obvious estimate $\Inf_n \phi_n + \Inf_n \leq \inf_n \phi_n + \inf_n \psi_n= \inf_n\, (\phi_n+\psi_n)$
and the maximality of $\Inf_n\, (\phi_n+\psi_n)$ stated above yield
$$
\Inf_n \phi_n + \Inf_n \psi_n \leq \Inf_n\,(\phi_n+\psi_n),
$$
but the equality generally fails to hold here. Related results (for other classes of unbounded objects) 
will be mentioned in Theorem \ref{T-3.10}, Remark \ref{R-3.11} and Proposition \ref{P-3.27}
(see also the paragraph before that proposition).

Let $\{q_n\}$ be a sequence of positive forms corresponding to $\{h_n\}$ in
$\widehat{B(\cH)}_+$.
We say that a sequence $\{h_n\}$ in $\widehat{B(\cH)}_+$ converges in the
\emph{strong resolvent sense} to $h\in\widehat{B(\cH)}_+$ if $(1+h_n)^{-1}\to(1+h)^{-1}$
strongly \cite{Ka2,RS}. The following is from \cite[Lemma 18]{Ko6}:

\begin{lemma}\label{L-2.4}
Let $q_n$ $(n\in\bN)$ and $q$ be positive forms corresponding to $h_n$ and $h$ in
$\widehat{B(\cH)}_+$. Assume that $q_n$ is decreasing. Then $q=\Inf_nq_n$ if and only if
$h_n\to h$ in the strong resolvent sense, that is, $(1+h_n)^{-1}\nearrow(1+h)^{-1}$ strongly.
$($Similarly, when $q_n$ is increasing, $q=\sup_nq_n$ if and only if
$(1+h_n)^{-1}\searrow(1+h)^{-1}$ strongly.$)$
\end{lemma}

From this lemma we also write $h=\Inf_nh_n$ when $h_n$ is decreasing and $h_n\to h$ in
the strong resolvent sense.

%%%%%%%%%%%%%%%%%%%%%%%%%%%%%%%%%%%%%%%%%%%%%%%%%
%%%%%%%% Connections of positive forms %%%%%%%%%%%%%%%%%%%%%%%%%
%%%%%%%%%%%%%%%%%%%%%%%%%%%%%%%%%%%%%%%%%%%%%%%%%
\subsection{Connections of positive forms}\label{S-2.2}

At first recall the notion of connections for bounded positive operators on $\cH$ studied in
Kubo and Ando \cite{KA}. A \emph{connection} $\sigma$ is a binary operation
$\sigma:B(\cH)_+\times B(\cH)_+\to B(\cH)_+$ satisfying, for every $A,B,C,D\in B(\cH)_+$,
\begin{itemize}
\item[(I)] if $A\le C$ and $B\le D$, then $A\sigma B\le C\sigma D$,
\item[(II)] $C(A\sigma B)C\le(CAC)\sigma(CBC)$,
\item[(III)] if $A_n\searrow A$ and $B_n\searrow B$ strongly, then
$A_n\sigma B_n\searrow A\sigma B$ strongly.
\end{itemize}

A main result of \cite{KA} says that there is a bijective correspondence between the
connections $\sigma$ and the non-negative \emph{operator monotone functions} $f$ on $(0,\infty)$,
determined in such a way that
$$
A\sigma B=A^{1/2}f(A^{-1/2}BA^{-1/2})A^{1/2}
$$
for $A,B\in B(\cH)_+$ with $A$ invertible. (A connection $\sigma$ is called an
\emph{operator mean} if $I\sigma I=I$ or $f(1)=1$.) The function $f$, called the
\emph{representing function} of $\sigma$, has the integral expression
\begin{align}\label{F-2.7} %%%%%%%%%%%%%%%%%%%%%%%%%%%%%%% \label{F-2.7}
f(s)=\alpha+\beta s+\int_{(0,\infty)}{(1+t)s\over s+t}\,d\mu(t),\qquad s\in(0,\infty),
\end{align}
with unique $\alpha,\beta\ge0$ and a unique finite positive measure $\mu$ on $(0,\infty)$, see,
e.g., \cite{Bh}. Based on \eqref{F-2.7} it was shown in \cite{KA} that the connection $\sigma$
has the expression
\begin{align}\label{F-2.8} %%%%%%%%%%%%%%%%%%%%%%%%%%%%%%%% \label{F-2.8}
A\sigma B=\alpha A+\beta B+\int_{(0,\infty)}{1+t\over t}((tA):B)\,d\mu(t),
\qquad A,B\in B(\cH)_+.
\end{align}

For each connection $\sigma$ on $B(\cH)_+$ the connection $\sigma$ of positive forms was
introduced in \cite{Ko7} by extending the above expression as follows:

\begin{definition}\label{D-2.5}\rm
For a connection $\sigma$ given in \eqref{F-2.8} the \emph{connection} $\sigma$ of positive
forms $\phi,\psi$ is defined by
\begin{align}\label{F-2.9} %%%%%%%%%%%%%%%%%%%%%%%%%%%%%%%% \label{F-2.9}
(\phi\sigma\psi)(\xi):=\alpha\phi(\xi)+\beta\psi(\xi)
+\int_{(0,\infty)}{1+t\over t}((t\phi):\psi)(\xi)\,d\mu(t),\quad \xi\in\cH,
\end{align}
where the lower semi-continuity of $\xi\mapsto(\phi\sigma\psi)(\xi)$ on $\cH$ is easily seen from
that of $((t\phi):\psi)(\xi)$ and Fatou's lemma.
\end{definition}

By definition it is obvious that the connection $\phi\sigma\psi$ for positive forms is a
generalization of $A\sigma B$ for $A,B\in B(\cH)_+$, that is, $\phi\sigma\psi=q_{A\sigma B}$ if
$\phi(\xi)=q_A(\xi)=\|A^{1/2}\xi\|^2$ and $\psi(\xi)=q_B(\xi)=\|B^{1/2}\xi\|^2$ for $\xi\in\cH$.

The notions of transpose and adjoint play an important part in theory of connections (and
operator means) on $B(\cH)_+$. Corresponding to the transpose $\tilde f(t):=tf(t^{-1})$ and
the adjoint $f^*(t):=f(t^{-1})^{-1}$ for operator monotone functions $f>0$ on $(0,\infty)$ we
have the \emph{transpose} $\tilde\sigma$ and the \emph{adjoint} $\sigma^*$ of a connection
$\sigma$ on $B(\cH)_+$ \cite{KA}. In the rest of the section we examine $\tilde\sigma$ and
$\sigma^*$ for connections of positive forms.

\begin{proposition}\label{P-2.6}
For any connection $\sigma$ and positive forms $\phi,\psi$ we have
$$
\phi\tilde\sigma\psi=\psi\sigma\phi.
$$
\end{proposition}

\begin{proof}
For the representing function $f$ of $\sigma$ with 
the expression \eqref{F-2.7} we write
\begin{align*}
\tilde f(s)&=\alpha s+\beta+\int_{(0,\infty)}{(1+t)s\over1+ts}\,d\mu(t) \\
&=\beta+\alpha s+\int_{(0,\infty)}{(1+t)s\over s+t}\,d\tilde\mu(t),
\end{align*}
where $d\tilde\mu(t):=d\mu(t^{-1})$ for $t\in(0,\infty)$. Hence it follows that for every $\xi\in\cH$,
\begin{align*}
(\phi\tilde\sigma\psi)(\xi)&=\beta\phi(t)+\alpha\psi(\xi)
+\int_{(0,\infty)}{1+t\over t}((t\phi):\psi)(\xi)\,d\tilde\mu(t) \\
&=\alpha\psi(\xi)+\beta\phi(\xi)
+\int_{(0,\infty)}{1+t\over t}\,t((t^{-1}\phi):\psi)(\xi)\,d\mu(t).
\end{align*}
Since $t((t^{-1}\phi):\psi)(\xi)=(\phi:(t\psi))(\xi)=((t\psi):\phi)(\xi)$, we have
$(\phi\tilde\sigma\psi)(\xi)=(\psi\sigma\phi)(\xi)$.
\end{proof}

The parallel sum is the adjoint of the sum, and $\phi:\psi=(\phi^{-1}+\psi^{-1})^{-1}$
holds as the definition itself (Definition \ref{D-2.1}). But this is not true for general
connection $\sigma$ (see Remark \ref{R-2.10} below).

\begin{proposition}\label{P-2.7}
Let $\sigma$ be a connection. Either if $h,k\in B(\cH)_+$, or if $h,k\in\widehat{B(\cH)}_+$ are
bounded from below, i.e., $h\ge\eps1$, $k\ge\eps1$ for some $\eps>0$, then
$$
q_h\sigma^*q_k=(q_h^{-1}\sigma q_k^{-1})^{-1},\quad\mbox{i.e.,}\quad
(q_h\sigma^*q_k)^{-1}=q_h^{-1}\sigma q_k^{-1}.
$$
\end{proposition}

\begin{proof}
Assume $h,k\in B(\cH)_+$ first. From the decreasing convergence (in (III) above) of the
connection $\sigma^*$ on $B(\cH)_+$ we have
$$
q_h\sigma^*q_k=\inf_{\eps>0}\,(q_{h+\eps1}\sigma^*q_{k+\eps1})
=\Inf_{\eps>0}\,(q_{h+\eps1}\sigma^*q_{k+\eps1})
$$
so that
$$
(q_h\sigma^*q_k)^{-1}=\sup_{\eps>0}\,(q_{h+\eps1}\sigma^*q_{k+\eps1})^{-1}
$$
by \cite[Lemma 16,(ii)]{Ko6}. Since
$((h+\eps1)\sigma^*(k+\eps1))^{-1}=(h+\eps1)^{-1}\sigma(k+\eps1)^{-1}$ in the case of
invertible bounded positive operators, it follows that
$$
(q_{h+\eps1}\sigma^*q_{k+\eps1})^{-1}=q_{h+\eps1}^{-1}\sigma q_{k+\eps1}^{-1}.
$$
Hence it suffices to show that
\begin{align}\label{F-2.10} %%%%%%%%%%%%%%%%%%%%%%%%%%%%%% \label{F-2.10}
q_h^{-1}\sigma q_k^{-1}=\sup_{\eps>0}\,(q_{h+\eps1}^{-1}\sigma q_{k+\eps1}^{-1}).
\end{align}
For every $\xi\in\cH$, in terms of 
the expression \eqref{F-2.9} we write
\begin{align}
(q_{h+\eps1}^{-1}\sigma q_{k+\eps1}^{-1})(\xi)
&=\alpha((h+\eps1)^{-1}\xi,\xi)+\beta((k+\eps1)^{-1}\xi,\xi) \nonumber\\
&\quad+\int_{(0,\infty)}{1+t\over t}(\{(t(h+\eps1)^{-1}):(k+\eps1)^{-1}\}\xi,\xi)\,d\mu(t)
\nonumber\\
&=\alpha((h+\eps1)^{-1}\xi,\xi)+\beta((k+\eps1)^{-1}\xi,\xi) \nonumber\\
&\quad+\int_{(0,\infty)}{1+t\over t}(\{(t^{-1}h+k)+\eps(t^{-1}+1)\}^{-1}\xi,\xi)\,d\mu(t).
\label{F-2.11} %%%%%%%%%%%%%%%%%%%%%%%%%%%%%%%%%%%%% \label{F-2.11}
\end{align}
With the spectral decomposition $h=\int_0^{\|h\|}\lambda\,de_\lambda$ one can compute
\begin{align*}
((h+\eps1)^{-1}\xi,\xi)&=\int_0^{\|h\|}{1\over\lambda+\eps}\,d\|e_\lambda\xi\|^2 \\
&\nearrow\,\int_{(0,\|h\|]}\lambda^{-1}\,d\|e_\lambda\xi\|^2+\infty\|e_0\xi\|^2
=q_h^{-1}(\xi)
\end{align*}
as $\eps\searrow0$. Similarly, $((k+\eps1)^{-1}\xi,\xi)\nearrow q_k^{-1}(\xi)$ and
$(\{(t^{-1}h+k)+\eps(t^{-1}+1)\}^{-1}\xi,\xi)\,\nearrow\,q_{t^{-1}h+k}^{-1}(\xi)$ as
$\eps\searrow0$. Since $h,k$ are bounded, one furthermore has $q_{t^{-1}h+k}=t^{-1}q_h+q_k$ so
that $q_{t^{-1}h+k}^{-1}(\xi)=(t^{-1}q_h+q_k)^{-1}(\xi)=((tq_h^{-1}):q_k^{-1})(\xi)$. Hence
from the monotone convergence theorem applied to \eqref{F-2.11} it follows that
\begin{align*}
\sup_{\eps>0}\,(q_{h+\eps1}^{-1}\sigma q_{k+\eps1}^{-1})(\xi)
&=\alpha q_h^{-1}(\xi)+\beta q_k^{-1}(\xi)
+\int_{(0,\infty)}{1+t\over t}((tq_h^{-1}):q_k^{-1})(\xi)\,d\mu(t) \\
&=(q_h^{-1}\sigma q_k^{-1})(\xi),
\end{align*}
showing \eqref{F-2.10}.

Next, assume that $h,k$ are bounded from below. Then we can apply the above case to
$q_{h^{-1}}=q_h^{-1}$, $q_{k^{-1}}=q_k^{-1}$ and $\sigma^*$ in place of $q_h$, $q_k$ and
$\sigma$ to have $(q_h^{-1}\sigma q_k^{-1})^{-1}=q_h\sigma^* q_k$.
\end{proof}

\begin{proposition}\label{P-2.8}
For any connection $\sigma$ and positive forms $\phi,\psi$ we have
$$
\phi\sigma^*\psi\le(\phi^{-1}\sigma\psi^{-1})^{-1},\quad\mbox{i.e.,}\quad
(\phi\sigma^*\psi)^{-1}\ge\phi^{-1}\sigma\psi^{-1}.
$$
\end{proposition}

\begin{proof}
Write $\phi=q_h$ and $\psi=q_k$ with $h,k\in\widehat{B(\cH)}_+$ and take the spectral
decompositions
$$
h=\int_0^\infty\lambda\,de_\lambda+\infty e_\infty^\perp,\qquad
k=\int_0^\infty\lambda\,df_\lambda+\infty f_\infty^\perp.
$$
For each $n\in\bN$ set
\begin{align*}
h_n&:=(1/n)e_{1/n}+\int_{(1/n,\infty)}\lambda\,de_\lambda+\infty e_\infty^\perp, \\
k_n&:=(1/n)f_{1/n}+\int_{(1/n,\infty)}\lambda\,df_\lambda+\infty f_\infty^\perp.
\end{align*}
Since
\begin{align*}
q_{h_n}(\xi)&={1\over n}\,\|e_{1/n}\xi\|^2+\int_{(1/n,\infty)}\lambda\|e_\lambda\xi\|^2
+\infty\|e_\infty^\perp\xi\|^2 \\
&=\int_{[0,1/n]}(1/n-t)\,d\|e_\lambda\xi\|^2+q_h(\xi)\,\searrow\,q_h(\xi)
\end{align*}
as $n\to\infty$ for every $\xi\in\cH$, one has $q_h=\inf_nq_{h_n}=\Inf_nq_{h_n}$ so that
$q_h^{-1}=\sup_nq_{h_n}^{-1}$ by \cite[Lemma 16,(ii)]{Ko6}. The same holds for $q_k$
as well. Since $t^{-1}q_{h_n}(\xi)+q_{k_n}(\xi)\searrow t^{-1}q_h(\xi)+q_k(\xi)$, it follows
that $t^{-1}q_h+q_k=\Inf_n(t^{-1}q_{h_n}+q_{k_n})$ so that
$$
(tq_h^{-1}:q_k^{-1})=(t^{-1}q_h+q_k)^{-1}
=\sup_n\,(t^{-1}q_{h_n}+q_{k_n})^{-1}=\sup_n\,((tq_{h_n}^{-1}):q_{k_n}^{-1}),
\qquad t>0.
$$
From these and the monotone convergence theorem with 
the expression \eqref{F-2.9} we have
\begin{align*}
\sup_n\,(q_{h_n}^{-1}\sigma q_{k_n}^{-1})(\xi)
&=\sup_n\biggl[\alpha q_{h_n}^{-1}(\xi)+\beta q_{k_n}^{-1}(\xi)
+\int_{(0,\infty)}{1+t\over t}((tq_{h_n}^{-1}):q_{k_n}^{-1})(\xi)\,d\mu(t)\biggr] \\
&=\alpha q_h^{-1}(\xi)+\beta q_k^{-1}(\xi)
+\int_{(0,\infty)}{1+t\over t}((tq_h^{-1}):q_k^{-1})(\xi)\,d\mu(t) \\
&=(q_h^{-1}\sigma q_k^{-1})(\xi),\qquad\xi\in\cH.
\end{align*}
Moreover, since $h_n,k_n$ are bounded from below, Proposition \ref{P-2.7} implies that
$q_{h_n}^{-1}\sigma q_{k_n}^{-1}=(q_{h_n}\sigma^* q_{k_n})^{-1}\le(q_h\sigma^*q_k)^{-1}$ for
all $n$. Therefore, $q_h^{-1}\sigma q_k^{-1}\le(q_h\sigma^*q_k)^{-1}$ follows.
\end{proof}

\begin{example}\label{E-2.9}\rm
The most studied family of operator means (for bounded positive operators) is the
\emph{weighted geometric means} $\#_\alpha$ ($0\le\alpha\le1$). In particular, the
\emph{geometric mean} $\#=\#_{1/2}$ was first introduced by Pusz and Woronowicz \cite{PW} and
developed in \cite{An2,An3}. The representing function of $\#_\alpha$ ($0<\alpha<1$) is
$$
s^\alpha={\sin\alpha\pi\over\pi}\int_0^\infty{s\over s+t}\,{dt\over t^{1-\alpha}},
\qquad s>0.
$$
Hence for any positive forms $\phi,\psi$ we write
\begin{align}\label{F-2.12} %%%%%%%%%%%%%%%%%%%%%%%%%%%%%%% \label{F-2.12}
(\phi\#_\alpha\psi)(\xi)={\sin\alpha\pi\over\pi}\int_0^\infty
((t\phi):\psi)(\xi)\,{dt\over t^{2-\alpha}},\qquad0<\alpha<1,
\end{align}
and $\phi\#_0\psi=\phi$, $\phi\#_1\psi=\psi$. Since $(\#_\alpha)\,\tilde{}=\#_{1-\alpha}$ and
$(\#_\alpha)^*=\#_\alpha$ obviously, Propositions \ref{P-2.6} and \ref{P-2.8} imply that
$$
\phi\#_\alpha\psi=\psi\#_{1-\alpha}\phi,\qquad
(\phi\#_\alpha\psi)^{-1}\ge\phi^{-1}\#_\alpha\psi^{-1}
$$
for any positive forms $\phi,\psi$. 
We here give the homogeneity property of $\#_\alpha$:
\begin{align}\label{F-2.13} %%%%%%%%%%%%%%%%%%%%%%%%%%%%%%% \label{F-2.13}
(r_1\phi)\#_\alpha(r_2\psi)=r_1^{1-\alpha}r_2^\alpha(\phi\#_\alpha\psi),
\qquad r_1,r_2\ge0.
\end{align}
Indeed, this is clear when $r_1=0$ or $r_2=0$ or $\alpha=0,1$. So assume $r_1,r_2>0$ and
$0<\alpha<1$. Then one can easily compute
\begin{align*}
((r_1\phi)\#_\alpha(r_2\psi))(\xi)
&={\sin\alpha\pi\over\pi}\int_0^\infty
((r_1t\phi):(r_2\psi))(\xi)\,{dt\over t^{2-\alpha}} \\
&={\sin\alpha\pi\over\pi}\int_0^\infty
r_2((r_1r_2^{-1}t\phi):\psi)(\xi)\,{dt\over t^{2-\alpha}} \\
&={\sin\alpha\pi\over\pi}\int_0^\infty
r_2(t\phi:\psi)(\xi){r_1^{-1}r_2\over(r_1^{-1}r_2t)^{2-\alpha}}\,dt \\
&=r_1^{1-\alpha}r_2^\alpha(\phi\#_\alpha\psi)(\xi).
\end{align*}
\end{example}

\begin{remark}\label{R-2.10}\rm
It is known \cite{Ko5} that there exists a pair $(A,B)$ of
non-singular positive self-adjoint operators
such that all of $A$, $B$, $A^{-1}$ and $B^{-1}$ have dense domains and
$$
\cD(A)\cap\cD(B)=\cD(A^{-1})\cap\cD(B^{-1})=\{0\}.
$$
For such $A,B$ consider $h:=A^2$ and $k:=B^2$. Then for every $t>0$ we have
$(tq_h):q_k=(tq_h^{-1}):q_k^{-1}=0$ for all $t>0$. Consider the weighted geometric mean
$\#_\alpha$ with $0<\alpha<1$. Then by \eqref{F-2.12} we have
$q_h\#_\alpha q_k=q_h^{-1}\#_\alpha q_k^{-1}=0$, and hence
$(q_h\#_\alpha q_k)^{-1}=\infty>0=q_h^{-1}\#_\alpha q_k^{-1}$. This shows that
$(\phi\sigma^*\psi)^{-1}=\phi^{-1}\sigma\psi^{-1}$ does not hold for general positive forms
when $\sigma=\#_\alpha$ ($0<\alpha<1$).
\end{remark}

%%%%%%%%%%%%%%%%%%%%%%%%%%%%%%%%%%%%%%%%%%%%%%%%%%%
%%%%%%%%%%% Connections of $\tau$-measurable operator %%%%%%%%%%%%%%%%%%
%%%%%%%%%%%%%%%%%%%%%%%%%%%%%%%%%%%%%%%%%%%%%%%%%%%
\section{Connections of $\tau$-measurable operators}\label{S-3}

In this section we consider a semi-finite von Neumann algebra $\cM$ with a trace $\tau$,
and connections for $\tau$-measurable operators will be studied.
In \S\ref{S-3.1} we collect some general results on $\tau$-measurable operators needed 
in later subsections, and among other things various behaviors of decreasing sequences 
of such operators are studied (see Theorems \ref{T-3.8} and \ref{T-3.10} for instance). 
Then, in \S\ref{S-3.2} we consider two natural definitions for connections, 
which are shown to be equivalent (Theorem \ref{T-3.23}). 
Various properties are known for connections in the setting of bounded positive operators.
In \S\ref{S-3.3} we will show that these properties remain to hold true in our general setting. 
Some special topics will be covered in the last two subsections. 
The subspace $L^1+\cM$ (which is a natural ambient space in
non-commutative interpolation theory for instance) is considered in \S\ref{S-3.4}
while \S\ref{S-3.5} deals with analysis closely related to that in \cite{Ko4}.

%%%%%%%%%%%%%%%%%%%%%%%%%%%%%%%%%%%%%%%%%%%%%%%%%%%%
%%%%%%%%%%%%%%%%%%%%%%%%% 3.1  %%%%%%%%%%%%%%%%%%%%%%%%
%%%%%%%%%%%%%%%%%%%%%%%%%%%%%%%%%%%%%%%%%%%%%%%%%%%%
\subsection{$\tau$-measurability and decreasing sequences}\label{S-3.1}

We begin with brief explanation on $\tau$-measurable operators together with their
generalized $s$-numbers.  Throughout let $\cM$ be a semi-finite von Neumann algebra on
a Hilbert space $\cH$ with a faithful, semi-finite and normal trace $\tau$.
A densely defined closed operator $a$ affiliated with $\cM$ is said to be
\emph{$\tau$-measurable} if, for any $\delta>0$, there exists a projection $e\in\cM$
such that $e\cH\subseteq\cD(a)$ (so $\|ae\|<\infty$ 
by the closed graph theorem) and $\tau(e^\perp)\le\delta$.
It is known that $a$ is $\tau$-measurable if and only if 
$\tau(e_{(s,\infty)}(|a|)) <\infty$ for some $s>0$,
where $|a|$ is the positive part of the polar decomposition.
This notion was introduced in \cite{Ne}. Let $\overline\cM$ be the space of
$\tau$-measurable operators affiliated with $\cM$, and $\overline\cM_+$ be the positive
part of $\overline\cM$. As usual we simply write $a+b$ and $ab$ for the strong sum
and the strong product. See \cite{Ne,Te1} for more details on $\tau$-measurable operators.

For each $a\in\overline\cM$ we write $e_I(|a|)$ for the spectral projection of $|a|$
corresponding to an interval $I\subseteq[0,\infty)$. The \emph{generalized $s$-number}
$\mu_t(a)$ ($t>0$) was introduced in \cite{FK} as
\begin{align*}
\mu_t(a)&:=\inf\{\|ae\|;\,e\ \mbox{is a projection in $\cM$ with $\tau(1-e)\le t$}\} \\
&\ =\inf\{s\ge0;\,\tau(e_{(s,\infty)}(|a|))\le t\}.
\end{align*}
This notion is useful in analysis of $\tau$-measurable operators. Note that $\overline\cM$
is a complete metrizable topological *-algebra whose neighborhood basis of $0$ is
\begin{align*}
V(\eps,\delta):=\{a\in\overline\cM&;\,\mbox{there is a projection $e$ in $\cM$} \\
&\quad\mbox{such that $\|ae\|\le\eps,\,\tau(1-e)\le\delta$}\},\qquad
\eps,\delta>0,
\end{align*}
and furthermore (see the proof of \cite[Lemma 3.1]{FK}),
\begin{align}\label{F-3.1}
a\in V(\eps,\delta)\,\iff\,\mu_\delta(a)\le\eps.
\end{align}

We write $\overline\fS$ for the set of $a\in\overline\cM$ such that
$\tau(e_{(\eps,\infty)}(|a|))<\infty$ for all $\eps>0$ or equivalently
$\lim_{t\to\infty}\mu_t(a)=0$. The space $\overline\fS$ is a closed (in the measure topology)
two-sided ideal of $\overline\cM$ and contains $L^p(\cM,\tau)$ for all $p\in(0,\infty)$,
where $L^p(\cM,\tau)$ is the non-commutative $L^p$-space associated with $(\cM,\tau)$.
These $L^p$-spaces were studied in \cite{Di,Y}. 
The latter directly deals with unbounded operators whereas  in the former
the $L^p$-space is introduced as the Banach space completion of a certain space of bounded 
operators equipped with the $p$-norm.
We recall that the $L^p$-norm is given by
$$
\|a\|_p=\tau(|a|^p)^{1/p}=\left(\int_0^{\infty} \mu_t(a)^p\,dt\right)^{1/p}.
$$
Decreasingness of a generalized $s$-number yields the obvious estimate
$\|a\|_p^p \geq s\mu_s(a)^p$ for each $s>0$ so that we have
\begin{align}\label{F-3.2}
\mu_s(a) \leq \|a\|_ps^{-1/p},\qquad s>0.
\end{align}
Note that when $\cM=B(\cH)$ with the usual trace, $\overline\cM=B(\cH)$ and
$\overline\fS$ is the set of compact operators on $\cH$. Hence, operators in $\overline\fS$
are called $\tau$-compact operators in some literature. 

In this subsection we prepare a few basic facts on $\tau$-measurable operators, in particular,
on convergence of decreasing sequences in $\overline\cM_+$, which will be used later in this
section. For each $a\in\overline\cM_+$ the corresponding positive form $q_a$ has the domain
$\cD(q_a)=\cD(a^{1/2})$ dense in $\cH$ since $a$ is densely defined.

The first four lemmas might be known to experts while we give the proofs for completeness.

\begin{lemma}\label{L-3.1}
Let $a,b$ be positive self-adjoint operators affiliated with $\cM$. Assume that $a\le b$ in the
form sense, or equivalently $(1+a)^{-1}\ge(1+b)^{-1}$. Then $b\in\overline\cM_+$ implies
$a\in\overline\cM_+$.
\end{lemma}

\begin{proof}
Set $x:=(1+a)^{-1}$ and $y:=(1+b)^{-1}$; then $x,y\in\cM_+$ and $x\ge y$. From a standard
argument as in the proof of \cite[Proposition 2.2]{FK} it follows that
$\tau(e_{[0,\alpha)}(x))\le\tau(e_{[0,\alpha)}(y))$ for all $\alpha>0$. Since
$$
e_{(s,\infty)}(a)=e_{[0,(1+s)^{-1})}(x),\quad
e_{(s,\infty)}(b)=e_{[0,(1+s)^{-1})}(y),\qquad s\ge0,
$$
we find that $\tau(e_{(s,\infty)}(a))\le\tau(e_{(s,\infty)}(b))$ for all $s\ge0$, showing
the assertion.
\end{proof}

\begin{lemma}\label{L-3.2}
Let $a,b\in\overline\cM_+$. The order $a\le b$ in $\overline\cM_+$, i.e.,
$b-a\in\overline\cM_+$ is equivalent to $a\le b$ in the form sense.
\end{lemma}

\begin{proof}
The proof is standard as follows:
\begin{align*}
a\le b\ \mbox{in $\overline\cM_+$}&\,\iff\,1+a\le1+b\ \mbox{in $\overline\cM_+$} \\
&\,\iff\,(1+b)^{-1/2}(1+a)(1+b)^{-1/2}\le1 \\
&\,\iff\,(1+a)^{1/2}(1+b)^{-1}(1+a)^{1/2}\le1 \\
&\,\iff\,(1+b)^{-1}\le(1+a)^{-1},
\end{align*}
showing the assertion.
\end{proof}

The following factorization technique is the $\tau$-measurable operator version of
\cite{D} and will be repeatedly used throughout the section:

\begin{lemma}\label{L-3.3}
If $a,b\in\overline\cM_+$ and $a\le\lambda b$ for some $\lambda>0$, then there exists a unique
$x\in s(b)\cM s(b)$, where $s(b)$ is the support projection of $b$, such that
$a^{1/2}=xb^{1/2}$. Moreover, we have $\|x\|\le\lambda^{1/2}$.
\end{lemma}

\begin{proof}
Assume that $a\le\lambda b$ with $\lambda>0$. From Lemma \ref{L-3.2} it follows that
$\|a^{1/2}\xi\|^2\le\lambda\|b^{1/2}\xi\|^2$ for all $\xi\in\cD(b^{1/2})$. Since the range of
$b^{1/2}$ is dense in $s(b)\cH$, there is a unique operator $x$ on $\cH$ such that
$xs(b)^\perp=0$ and $a^{1/2}\xi=xb^{1/2}\xi$ for all $\xi\in\cD(b^{1/2})$. Moreover, we have
$\|x\|\le\lambda^{1/2}$. From uniqueness we easily see that $x\in\cM$ and hence
$x=s(a)xs(b)\in s(b)\cM s(b)$. Since $\cD(b^{1/2})$ is $\tau$-dense, we have $a^{1/2}=xb^{1/2}$
by \cite[Chap.~I, Proposition 12]{Te1}.
\end{proof}

\begin{lemma}\label{L-3.4}
For every $a,b\in\overline\cM_+$ the strong sum $a+b$ in $\overline\cM_+$ coincides with the
form sum $a\,\dot+\,b$, that is, $q_{a+b}=q_a+q_b$.
\end{lemma}

\begin{proof}
Let $h:=a+b$ in $\overline\cM_+$. By Lemma \ref{L-3.3} there exist $x,y\in s(h)\cM s(h)$ such
that $a^{1/2}=xh^{1/2}$ and $b^{1/2}=yh^{1/2}$, so $a^{1/2}=h^{1/2}x^*$ and $b^{1/2}=h^{1/2}y^*$
as well. We have
$$
h=h^{1/2}x^*xh^{1/2}+h^{1/2}y^*yh^{1/2}=h^{1/2}(x^*x+y^*y)h^{1/2},
$$
from which $x^*x+y^*y=s(h)$ follows immediately. Clearly,
$\cD(h^{1/2})\subseteq\cD(a^{1/2})\cap\cD(b^{1/2})$. Moreover, since
$h^{1/2}=(x^*x+y^*y)h^{1/2}=x^*a^{1/2}+y^*b^{1/2}$, we have
$\cD(a^{1/2})\cap\cD(b^{1/2})\subseteq\cD(h^{1/2})$. Therefore,
$\cD(h^{1/2})=\cD(a^{1/2})\cap\cD(b^{1/2})=\cD((a\,\dot+\,b)^{1/2})$. For every
$\xi\in\cD(h^{1/2})$,
\begin{align*}
\|a^{1/2}\xi\|^2+\|b^{1/2}\xi\|^2
&=\|xh^{1/2}\xi\|^2+\|yh^{1/2}\xi\|^2 \\
&=((x^*x+y^*y)h^{1/2}\xi,h^{1/2}\xi)=\|h^{1/2}\xi\|^2.
\end{align*}
Hence $q_a+q_b=q_h$, that is, $a\,\dot+\,b=h$ follows.
\end{proof}

\begin{definition}\label{D-3.5}\rm
Assume that $a,b\in\overline\cM_+$ satisfies $a\le\lambda b$ for some $\lambda>0$. By Lemma
\ref{L-3.3} we have a unique $x\in s(b)\cM s(b)$ such that $a^{1/2}=xb^{1/2}$. It is convenient
for later use to introduce the notation $T_{a/b}:=x^*x$, which is a unique $T\in(s(b)\cM s(b))_+$
such that $a=b^{1/2}Tb^{1/2}$. Moreover, $T_{a/b}\le\lambda s(b)$ holds. From the proof of
Lemma \ref{L-3.4} we also note that when $a,b\in\overline\cM_+$ and $h:=a+b$, we have
$T_{a/h}+T_{b/h}=s(h)$.
\end{definition}

\begin{lemma}\label{L-3.6}
Let $a,a_n\in\cM$ $(n\in\bN)$ and assume that $\sup_n\|a_n\|<\infty$. If $a_n\to a$ in the measure
topology, then $a_n\to a$ in the strong operator topology.
\end{lemma}

\begin{proof}
We may and do assume $a=0$. First, let us show the result when $\cM$ is standardly represented (by left
multiplication) on the Hilbert space $L^2(\cM,\tau)$. Let $\alpha:=\sup_n\|a_n\|<+\infty$. For every
$z\in L^2(\cM,\tau)$, by \cite[Lemma 2.5,(vii)]{FK} we have
$$
\|a_nz\|_2^2=\int_0^\infty\mu_t(a_nz)^2\,dt\le\int_0^\infty\mu_{t/2}(a_n)^2\mu_{t/2}(z)^2\,dt.
$$
From the assumption $a_n\to0$ in the measure topology, it follows that $\mu_{t/2}(a_n)\to0$ in
measure on $((0,\infty),dt)$. Moreover, note that
$\mu_{t/2}(a_n)^2\mu_{t/2}(z)^2\le\alpha^2\mu_{t/2}(z)^2\in L^1((0,\infty),dt)$. Hence Lebesgue's
dominated convergence theorem can be applied to see $\|a_nz\|_2\to0$.

It remains to show that the strong convergence $a_n\to0$ is space-free, i.e., independent of a
representing Hilbert space of $\cM$. To be more precise, we may show that if  $a_n\to0$ strongly
in $\cM$ with $\sup_n\|a_n\|<\infty$ and $\pi:\cM\to\cM_1$ is an isomorphism, then $\pi(a_n)\to0$
strongly in $\cM_1$. For this, we may consider the cases of $\pi$ being an amplification, an injective
induction, and a spatial isomorphism, separately (see \cite[Chap.~IV, Theorem 5.5]{Ta1}). The
assertion is obvious for the latter two cases. For the first case, let $\pi$ be the amplification of
$\cM$ onto $\cM\otimes1$ on $\cH\otimes\cK$. Then it is obvious that
$\pi(a_n)(\xi\otimes\eta)\to0$ for all $\xi\in\cH$ and $\eta\in\cK$. The assertion is immediate
since $\{\xi\otimes\eta:\xi\in\cH,\,\eta\in\cK\}$ is total in $\cH\otimes\cK$ (and $\{a_n\}$ is
uniformly bounded).
\end{proof}

The next proposition will be useful in our discussions below, which is probably known to
experts but we cannot find a suitable reference.

\begin{proposition}\label{P-3.7}
Let $a_n\in\overline\cM_+$ $(n\in\bN)$. If $a_n\to a$ in the measure topology for some
$a\in\overline\cM_+$, then $a_n\to a$ in the strong resolvent sense.
\end{proposition}

\begin{proof}
Since
$$
(1+a)^{-1}-(1+a_n)^{-1}=(1+a)^{-1}(a_n-a)(1+a_n)^{-1},
$$
we have for every $\eps>0$,
$$
\mu_\eps((1+a)^{-1}-(1+a_n)^{-1})\le\mu_\eps(a_n-a)\,\longrightarrow\,0
$$
as $n\to\infty$, which means that $(1+a_n)^{-1}\to(1+a)^{-1}$ in the measure topology.
Hence Lemma \ref{L-3.6} implies that $(1+a_n)^{-1}\to(1+a)^{-1}$ strongly, that is,
$a_n\to a$ in the strong resolvent sense.
\end{proof}

Let $a_n\in\overline\cM_+$ ($n\in\bN$) and assume that $a_1\ge a_2\ge\cdots$. Then there
is a positive self-adjoint operator $a$ on $\cH$ such that $a_n\to a$ in the strong resolvent
sense, i.e., $(1+a_n)^{-1}\nearrow(1+a)^{-1}$ (see Lemma \ref{L-2.4}). 
Note that Lemma \ref{L-3.1} guarantees $a\in\overline\cM_+$.
Let us agree to express this situation as $a_n\searrow a$ in the strong resolvent sense.

\begin{theorem}\label{T-3.8}
Let $a_n\in\overline\cM_+$ $(n\in\bN)$ and assume that $a_n\searrow a$ in the strong resolvent
sense. If $a_1\in\overline\fS_+$, then $a_n\to a$ in the measure topology. If $0<p<\infty$ and
$a_1\in L^p(\cM,\tau)_+$, then $a\in L^p(\cM,\tau)_+$ and $\|a_n-a\|_p\to0$.
\end{theorem}

\begin{proof}
First, assume $a_1\in L^1(\cM,\tau)_+$. Define $\ffi\in\cM_+^*$ by
$\ffi(x):=\lim_{n\to\infty}\tau(a_nx)$ for $x\in\cM$. Since $0\le\ffi\le\tau(a_1\cdot)$, $\ffi$
is normal, i.e., $\ffi\in\cM_*^+$. Hence $\ffi=\tau(b\,\cdot)$ for some $b\in L^1(\cM,\tau)_+$.
Since $a_n\ge b$ for all $n$, $\|a_n-b\|_1=\tau((a_n-b)1)=\tau(a_n1)-\ffi(1)\to0$. Proposition
\ref{P-3.7} shows that $a=b$. Hence $a\in L^1(\cM,\tau)_+$ and $\|a_n-a\|_1\to0$.

Next, assume $a_1\in\overline\fS_+$ and take the spectral decomposition
$a_1=\int_0^\infty\lambda\,de_\lambda$. Let us show that $\{a_n\}$ converges in $\overline\cM$.
In vew of \eqref{F-3.1} it suffices to prove that for any $\eps>0$ there is an $n_0$ such that
$\delta_\eps(a_n-a_m)\le\eps$ for all $n,m\ge n_0$. To do so, let $\eps>0$ be arbitrary and we
decompose $a_n$ into sum of several small pieces and an element of $L^1(\cM,\tau)_+$. Choose
$r>\delta>0$ such that $\delta\mu_\eps(a_1)<\eps^2$ and $\tau(e_r^\perp)<\eps$.
Let $f_1:=e_\delta$, $f_2:=e_r^\perp$ and $f:=1-(f_1+f_2)$. Then
$\tau(f)\le\tau(f_1^\perp)<\infty$ (thanks to $a_1\in\overline\fS$) and $\tau(f_2)<\eps$. Write
$$
a_n=a_nf_1+a_nf_2+a_nf=a_nf_1+a_nf_2+f_1a_nf+f_2a_nf+fa_nf.
$$
Since $\tau(f)<\infty$ and $\|fa_1f\|\le r$, one has $fa_1f\in L^1(\cM,\tau)_+$. From the
case proved above it follows that $\|fa_nf-b_0\|_1\to0$ for some $b_0\in L^1(\cM,\tau)_+$.
Hence there is an $n_0\in\bN$ such that
\begin{align}\label{F-3.3}
\mu_\eps(fa_nf-fa_mf)<\eps,\qquad n,m\ge n_0.
\end{align}
Since $\tau(f_2)<\eps$, one has
\begin{align}\label{F-3.4}
\mu_\eps(f_2a_nf)\le\mu_\eps(f_2a_n)=\mu_\eps(a_nf_2)=\mu_\eps^{1/2}(f_2a_n^2f_2)=0,
\qquad n\in\bN.
\end{align}
From \cite[Lemma 2.5,(vi)]{FK} one moreover finds that
\begin{align*}
\mu_\eps^2(f_1a_nf)&\le\mu_\eps^2(f_1a_n)=\mu_\eps^2(a_nf_1)
=\mu_\eps(f_1a_n^2f_1) \\
&\le\|f_1a_n^{1/2}\|\mu_\eps(a_n)\|a_n^{1/2}f_1\| \\
&=\|f_1a_nf_1\|^{1/2}\mu_\eps(a_n)\|f_1a_nf_1\|^{1/2} \\
&\le\|f_1a_1f_1\|\mu_\eps(a_1)\le\delta\mu_\eps(a_1)<\eps^2
\end{align*}
so that
\begin{align}\label{F-3.5}
\mu_\eps(f_1a_nf)\le\mu_\eps(a_nf_1)\le\eps,\qquad n\in\bN.
\end{align}
By \cite[Lemma 2.5,(v)]{FK} and \eqref{F-3.3}--\eqref{F-3.5} we find that
\begin{align*}
\mu_{9\eps}(a_n-a_m)
&\le\mu_\eps(a_nf_1)+\mu_\eps(a_mf_1)+\mu_\eps(f_1a_nf)+\mu_\eps(f_1a_mf) \\
&\quad+\mu_\eps(a_nf_2)+\mu_\eps(a_mf_2)+\mu_\eps(f_2a_nf)+\mu_\eps(f_2a_mf) \\
&\quad+\mu_\eps(fa_nf-fa_mf) \\
&\le5\eps
\end{align*}
for all $n,m\ge n_0$. Since $\eps>0$ is arbitrary, this shows that $\{a_n\}$ converges to 
some $b\in\overline\cM_+$  in the measure topology. Since $b=a$ by Proposition \ref{P-3.7},
the first assertion holds. The second follows from the first and \cite[Theorem 3.6]{FK}.
\end{proof}

The assumption of $a_1\in\overline\fS$ (or at least $a_n\in\overline\fS$ for some $n$) is
essential for the first assertion of Theorem \ref{T-3.8}. For this, we may consider a
sequence $a_n$ in $B(\cH)_+$ such that $a_n\searrow a$ but $\|a_n-a\|\not\to0$. Note that the
measure topology on $B(\cH)$ coincides with the operator norm topology.

The next lemma gives a convenient description of the limit in the strong resolvent sense of a
decreasing sequence in $\overline\cM_+$.

\begin{lemma}\label{L-3.9}
Let $a_n,b\in\overline\cM_+$ $(n\in\bN)$ be such that $a_1\ge a_2\ge\cdots$ and
$a_1\le\lambda b$ for some $\lambda>0$. Then $T_{a_n/b}$ $($see Definition \ref{D-3.5}$)$
is decreasing, and if $T_0:=\lim_nT_{a_n/b}$ $($the strong limit$)$ then
$a_n\searrow b^{1/2}T_0b^{1/2}$ in the strong resolvent sense.
\end{lemma}

\begin{proof}
Since $\|a_n^{1/2}\xi\|^2=(T_{a_n/b}b^{1/2}\xi,b^{1/2}\xi)$ is decreasing for every
$\xi\in\cD(b^{1/2})$, we see that $T_{a_n/b}$ is decreasing so that the strong limit
$T_0:=\lim_nT_{a_n/b}$ ($\in\cM_+$) exists. Now, let $a_0:=b^{1/2}T_0b^{1/2}$ and $a:=\lim_na_n$
in the strong resolvent sense, both of which are in $\overline\cM_+$ (by Lemma \ref{L-3.1}).
Since $(1+a_0)^{-1}\ge(1+a_n)^{-1}$ for all $n$ (by Lemma \ref{L-3.2}) and
$(1+a_n)^{-1}\nearrow(1+a)^{-1}$, one has $(1+a_0)^{-1}\ge(1+a)^{-1}$ so that $a_0\le a$ in
$\overline\cM_+$ by Lemma \ref{L-3.2} again. On the other hand, for every $\xi\in\cD(b^{1/2})$,
$$
(T_{a/b}b^{1/2}\xi,b^{1/2}\xi)=\|a^{1/2}\xi\|^2\le\|a_n^{1/2}\xi\|^2
=(T_{a_n/b}b^{1/2}\xi,b^{1/2}\xi),
$$
which implies $T_{a/b}\le T_{a_n/b}$ for all $n$ so that $T_{a/b}\le T_0$. Hence $a\le a_0$ in
$\overline\cM_+$ as well.
\end{proof}

\begin{theorem}\label{T-3.10}
Let $a_n,b_n\in\overline\cM_+$ and assume that $a_n\searrow a$ and $b_n\searrow b$ in the
strong resolvent sense. Then $a_n+b_n\searrow a+b$ in the strong resolvent sense. Moreover, if
$a_1,b_1\in\overline\fS_+$, then $a_n+b_n\to a+b$ in the measure topology. If
$a_1,b_1\in L^p(\cM,\tau)_+$ for some $p\in(0,\infty)$, then $a_n+b_n\to a+b$ in the
$($quasi-$)$norm $\|\cdot\|_p$.
\end{theorem}

\begin{proof}
Let $h:=a_1+b_1$. Set $T_n:=T_{a_n/h}$, $T:=T_{a/h}$, $S_n:=T_{b_n/h}$ and $S:=T_{b/h}$, which
are positive operators in $s(h)Ms(h)$. By Lemma \ref{L-3.9} we have $T_n\searrow T$ and
$S_n\searrow S$ strongly, and hence $T_n+S_n\searrow T+S$ strongly. It is immediate to see that
$T_{(a_n+b_n)/h}=T_n+S_n$ and $T_{(a+b)/h}=T+S$. Hence the conclusion of the first
assertion follows from Lemma \ref{L-3.9} again. The latter assertions are immediate from Theorem \ref{T-3.8}.
\end{proof}

\begin{remark}\label{R-3.11}\rm
Recall (Lemma \ref{L-2.4}) that $a_n\searrow a$ in the strong resolvent sense is equivalently
written as $q_a=\Inf_nq_{a_n}$ in terms of symbol $\Inf$. Theorem \ref{T-3.10} says that for
every decreasing sequences $a_n$ and $b_n$ in $\overline\cM_+$,
\begin{align}\label{F-3.6}
\Inf_n\,(q_{a_n}+q_{b_n})=\Bigl(\Inf_nq_{a_n}\Bigr)+\Bigl(\Inf_nq_{b_n}\Bigr).
\end{align}
Indeed, by Lemma \ref{L-3.4} and Theorem \ref{T-3.10},
$$
\Inf_n\,(q_{a_n}+q_{b_n})=\Inf_nq_{a_n+b_n}=q_{a+b}=q_a+q_b
=\Bigl(\Inf_nq_{a_n}\Bigr)+\Bigl(\Inf_nq_{b_n}\Bigr).
$$
Note that this is not true for general (not $\tau$-measurable) positive self-adjoint operators.
See \cite{Si1,Ko6} and \cite[\S7.4 and \S7.5]{Si3} for 
counterexamples for which \eqref{F-3.6} fails to hold.
\end{remark}

The following lemma might be well-known to specialists while we give a proof
for completeness:

\begin{lemma}\label{L-3.12}
Let $a\in\overline\fS$. Let $c,c_n\in\cM$ $(n\in\bN)$ be such that $c_n\to c$ in the strong* topology.
Then $ac_n\to ac$ and $c_na\to ca$ in the measure topology.
\end{lemma}

\begin{proof}
Note that $\alpha:=\sup_n\|c_n\|<+\infty$ by the uniform boundedness principle.
It suffices to show only the first convergence. Taking
the polar decomposition of $a$ we may assume $a\ge0$.  Take the spectral decomposition
$a=\int_0^\infty\lambda\,de_\lambda$. For any $\eps>0$ choose an $r>\eps$ such that
$\tau(e_r^\perp)<\eps$, and decompose $a$ as $a=a_1+a_2+a_3$ with
$$
a_1:=\int_{[0,\eps]}\lambda\,de_\lambda,\qquad
a_2:=\int_{(\eps,r]}\lambda\,de_\lambda,\qquad
a_3:=\int_{(r,\infty)}\lambda\,de_\lambda.
$$
Then it is clear that $\mu_\eps(a_1)\le\eps$ and $\mu_\eps(a_3)=0$. Since $a\in\overline\fS_+$,
note that $\tau(e_{(\eps,r]})<\infty$ so that $\tau(a_2^2)\le r^2\tau(e_{(\eps,r]})<\infty$.
We have
\begin{align*}
\mu_{3\eps}(a(c_n-c))
&\le\mu_\eps(a_1(c_n-c))+\mu_\eps(a_2(c_n-c))+\mu_\eps(a_3(c_n-c)) \\
&\le2 \alpha\mu_\eps(a_1)+\mu_\eps(a_2(c_n-c))+2\alpha\mu_\eps(a_3) \\
&\le2\alpha\eps+\mu_\eps(a_2(c_n-c)).
\end{align*}
Since $\tau(a_2^2)<\infty$ and $(c_n-c)^*(c_n-c)\to0$ 
weakly (and hence $\sigma$-weakly thanks to $\|c_n\| \leq \alpha$),
we have
$$
\|a_2(c_n-c)\|_2^2=\tau(a_2^2(c_n-c)^*(c_n-c))\,\longrightarrow\,0
$$
so that $a_2(c_n-c)\to0$ in the measure topology. Therefore,
$$
\limsup_{n\to\infty}\mu_{3\eps}(a-e_nae_n)\le2\alpha\eps.
$$
Since $\eps>0$ is arbitrary, this implies that $ac_n\to ac$ in the measure topology.
\end{proof}

%%%%%%%%%%%%%%%%%%%%%%%%%%%%%%%%%%%%%%%%%%%%%%%%%%%
%%%%%%%%%%%%%%%%%%%%%% 3.2 %%%%%%%%%%%%%%%%%%%%%%%%%%
%%%%%%%%%%%%%%%%%%%%%%%%%%%%%%%%%%%%%%%%%%%%%%%%%%%
\subsection{Definitions}\label{S-3.2}

As in the previous subsection let $\cM$ be a semi-finite von Neumann algebra with a faithful
semi-finite normal trace $\tau$. Throughout this subsection let $\sigma$ be any connection in
the sense of Kubo and Ando \cite{KA} having the expression in \eqref{F-2.8}. We define the
connection $\sigma$ for positive $\tau$-measurable operators in two different ways and prove
that they coincide.

The first definition is the restriction of Definition \ref{D-2.5} to positive forms
corresponding to $\tau$-measurable operators. For this we first give a lemma.

\begin{lemma}\label{L-3.13}
There exist a constant $\lambda>0$, depending on $\sigma$ only, such that
$$
\phi\sigma\psi\le\lambda(\phi+\psi)
$$
for all positive forms $\phi,\psi$.
\end{lemma}

\begin{proof}
Let $\phi,\psi$ be arbitrary positive forms. Theorem \ref{T-2.2} in particular implies that
$\phi:\psi\le\phi$ and $\phi:\psi\le\psi$. Therefore, from \eqref{F-2.9} we observe
\begin{align*}
\phi\sigma\psi&\le\alpha\phi+\beta\psi+\int_{(0,1]}{1+t\over t}\,t\phi\,d\mu(t)
+\int_{(1,\infty)}{1+t\over t}\,\psi\,d\mu(t) \\
&\le\alpha\phi+\beta\psi+2\mu((0,1])\phi+2\mu((1,\infty))\psi.
\end{align*}
Hence $\phi\sigma\psi\le\lambda(\phi+\psi)$ with
$\lambda:=\max\{\alpha+2\mu((0,1]),\beta+2\mu((1,\infty))\}$.
\end{proof}

\begin{proposition}\label{P-3.14}
For every $a,b\in\overline\cM_+$ there exist a unique $h\in\overline\cM_+$ such that
$$
q_a\sigma q_b=q_h.
$$
\end{proposition}

\begin{proof}
By Lemmas \ref{L-3.13} and \ref{L-3.4} we have
$$
q_a\sigma q_b\le\lambda(q_a+q_b)=\lambda q_{a+b}.
$$
This implies that the domain of $q_a\sigma q_b$ is dense in $\cH$. Hence there is a positive
self-adjoint operator $h$ such that $q_a\sigma q_b=q_h$. For every unitary $u'\in\cM'$ it
follows from the definition \eqref{F-2.9} and \cite[Corollary 7]{Ko6} that
$$
u'^*(q_a\sigma q_b)u'=\alpha u'^*q_au'+\beta u'^*q_bu'
+\int_{(0,\infty)}{1+t\over t}((tu'^*q_au'):(u'^*q_bu'))\,d\mu(t).
$$
Since $u'^*q_au'=q_{u'^*au'}=q_a$ and similarly for $q_b$, we have
$u'^*(q_a\sigma q_b)u'=a_a\sigma q_b$, i.e., $u'^*q_hu'=q_h$, which means that $u'^*hu'=h$.
Therefore, $h$ is affiliated with $\cM$. Moreover, since $q_h\le\lambda q_{a+b}$, we have
$h\le\lambda(a+b)$ in the form sense so that $h\in\overline\cM_+$ by Lemma \ref{L-3.1}.
Finally, the uniqueness of $h$ is obvious.
\end{proof}

\begin{definition}[First]\label{D-3.15}\rm
Proposition \ref{P-3.14} says that for each $a,b\in\overline\cM_+$ we can define
$a\sigma b\in\overline\cM_+$ by the condition $q_{a\sigma b}=q_a\sigma q_b$. Thus, we have the
\emph{connection} $\sigma:\overline\cM_+\times\overline\cM_+\to\overline\cM_+$. When
$\cM=B(\cH)=\overline\cM$, this connection $\sigma$ reduces to the connection $\sigma$ in
\cite{KA} on $B(\cH)_+\times B(\cH)_+$, see \cite[\S 3]{Ko6}.
\end{definition}

The second definition is given in a more conventional approach along the lines of
\cite{KA,PW}.

\begin{definition}[Second]\label{D-3.16}\rm
For every $a,b\in\overline\cM_+$ choose an $h\in\overline\cM_+$ such that $a+b\le\lambda h$
for some $\lambda>0$. With $T_{a/h}$ and $T_{b/h}$ (see Definition \ref{D-3.5}) we
define the \emph{connection} $a[\sigma]b\in\overline\cM_+$ by
\begin{align}\label{F-3.7}
a[\sigma]b:=h^{1/2}(T_{a/h}\sigma T_{b/h})h^{1/2},
\end{align}
where $T_{a/h}\sigma T_{b/h}$ is the connection \cite{KA} for bounded positive operators. The
definition is justified by Lemma \ref{L-3.18} below, so we have the connection
$[\sigma]:\overline\cM_+\times\overline\cM_+\to\overline\cM_+$. We use the symbol $[\sigma]$
to distinguish it from $\sigma$ in the previous definition while their coincidence will shortly
be proved.
\end{definition}

Recall \cite{KA} that the \emph{transformer inequality} for connections $\sigma$ on $B(\cH)_+$
holds true in the following slightly more general situation than (II) of \S\ref{S-2.2}:
$$
C^*(A\sigma B)C\le(C^*AC)\sigma(C^*BC),\qquad A,B\in B(\cH)_+,\ C\in B(\cH).
$$
Moreover, concerning the equality case of the transformer inequality, the following
result was shown in \cite[Theorem 3]{Fu}, which we state as a lemma to use in proving
Lemma \ref{L-3.18} and Theorem \ref{T-3.31}:

%%%%%%%%%%%%%%%%%%%%%%%% \label{L-3.17} %%%%%%%%%%%%%%%%%
\begin{lemma}\label{L-3.17}
Let $A,B\in B(\cH)_+$ and $C\in B(\cH)$ be such that $s(A+B)\le s(CC^*)$, where $s(\cdot)$
denotes  the support projection. Then for any connection $\sigma$,
$$
C^*(A\sigma B)C=(C^*AC)\sigma(C^*BC).
$$
\end{lemma}

\begin{lemma}\label{L-3.18}
For every $a,b\in\overline\cM_+$ the $a[\sigma]b$ given in \eqref{F-3.7} is an element of
$\overline\cM_+$ determined independently of the choice of $h$ as in Definition \ref{D-3.16}.
\end{lemma}

\begin{proof}
For any $h\in\overline\cM_+$ with $a+b\le\lambda h$ for some $\lambda>0$, since
$T_{a/h},T_{b/h}\in\cM_+$, we have for every unitary $u'\in\cM'$,
$$
u'^*(T_{a/h}\sigma T_{b/h})u'=(u'^*T_{a/h}u')\sigma(u'^*T_{b/h}u')
=T_{a/h}\sigma T_{b/h}.
$$
Hence $T_{a/h}\sigma T_{b/h}\in\cM_+$ so that the right hand side of \eqref{F-3.7} is
indeed in $\overline\cM_+$. To prove the independence of the choice of $h$, let
$k\in\overline\cM_+$ be another choice such as $h$. We may assume that $h\le k$. In fact,
we may consider the cases of $h,h+k$ and of $k,h+k$. Then by Lemma \ref{L-3.3} we have an
$x\in s(h)\cM s(k)$ such that $h^{1/2}=xk^{1/2}$. Since
$$
h^{1/2}T_{a/h}h^{1/2}=a=k^{1/2}T_{a/k}k^{1/2}=h^{1/2}x^*T_{a/k}xh^{1/2},
$$
one has $T_{a/h}=x^*T_{a/k}x$ and similarly $T_{b/h}=x^*T_{b/k}x$. Since
$s(a+b)\le s(h)=s(xx^*)$, it follows from Lemma \ref{L-3.17} that
\begin{align*}
k^{1/2}(T_{a/k}\sigma T_{b/k})k^{1/2}&=h^{1/2}x^*(T_{a/k}\sigma T_{b/k})xh^{1/2} \\
&=h^{1/2}((x^*T_{a/k}x)\sigma(x^*T_{b/k}x))h^{1/2} \\
&=h^{1/2}(T_{a/h}\sigma T_{b/h})h^{1/2}.
\end{align*}
\end{proof}

In particular, when $b\le\lambda a$ for some $\lambda>0$, taking $h=a$ in \eqref{F-3.7} we
write
\begin{align}\label{F-3.8}
a[\sigma]b=a^{1/2}f(T_{b/a})a^{1/2}.
\end{align}
Note that $T_{B/A}=A^{-1/2}BA^{-1/2}$ when $A,B\in B(\cH)_+$ with $A$ invertible. Therefore,
\eqref{F-3.8} is an extended version of the familiar formula (for Kubo-Ando connections)
$A\sigma B=A^{1/2}f(A^{-1/2}BA^{-1/2})A^{1/2}$ for $A$ invertible.

In the rest of the subsection we will prove that Definitions \ref{D-3.15} and \ref{D-3.16} are
equivalent. The next lemma is similar to the familiar convergence formula
$A\sigma B=\lim_{\eps\searrow0}(A+\eps I)\sigma(B+\eps I)$ for $A,B\in B(\cH)_+$.

\begin{lemma}\label{L-3.19}
Let $a,b,h\in \overline\cM_+$ be such that $a+b\le\lambda h$ for some $\lambda>0$.
Then we have
\begin{align}\label{F-3.9}
a[\sigma]b=\lim_{\eps\searrow0}(a+\eps h)[\sigma](b+\eps h),
\end{align}
a decreasing limit in the strong resolvent sense.
\end{lemma}

\begin{proof}
For any $\eps>0$ Lemma \ref{L-3.18} gives
\begin{align*}
(a+\eps h)[\sigma](b+\eps h)&=h^{1/2}(T_{(a+\eps h)/h}\sigma T_{(b+\eps h)/h})h^{1/2} \\
&=h^{1/2}((T_{a/h}+\eps s(h))\sigma(T_{b/h}+\eps s(h)))h^{1/2}.
\end{align*}
Since $T_{a/h}+\eps s(h)\searrow T_{a/h}$ and $T_{b/h}+\eps s(h)\searrow T_{b/h}$ strongly
(even in the operator norm) as $\eps\searrow0$, we have
$(T_{a/h}+\eps s(h))\sigma(T_{b/h}+\eps s(h))\searrow T_{a/h}\sigma T_{b/h}$ strongly as
$\eps\searrow0$ (see (III) of \S\ref{S-2.2}). Hence the assertion follows from Lemma \ref{L-3.9}.
\end{proof}

\begin{remark}\label{R-3.20}\rm
A typical choice of $h$ in the above lemma is $h=a+b$. In this case, when both of $a,b$ are
in $\overline\fS_+$, it follows from Theorem \ref{T-3.8} that \eqref{F-3.9} becomes the
limit in the measure topology. Moreover, when $a,b\in L^p(M,\tau)_+$ with $p\in(0,\infty)$ and
$h=a+b$, \eqref{F-3.9} becomes the limit in the (quasi-)norm $\|\cdot\|_p$.
\end{remark}

\begin{lemma}\label{L-3.21}
Let $a,b\in\overline\cM_+$. For every $\xi\in\cD(a^{1/2})\cap\cD(b^{1/2})$ we have
\begin{align}\label{F-3.10}
q_{a[:]b}(\xi)=\inf\{q_a(\eta)+q_b(\zeta);\,\eta,\zeta\in\cH,\,\eta+\zeta=\xi\},
\end{align}
where $a[:]b$ is given in Definition \ref{D-3.16} for parallel sum.
\end{lemma}

\begin{proof}
Firstly we recall the well-known variational formula of $A:B$ for $A,B\in B(\cH)_+$, that is,
for every $\xi\in\cH$,
\begin{align}\label{F-3.11}
((A:B)\xi,\xi)=\inf\{(A\eta,\eta)+(B\zeta,\zeta);\,\eta,\zeta\in\cH,\,\eta+\zeta=\xi\}
\end{align}
(see \cite[Theorem 9]{AT} and also \cite[Lemma 3.1.5]{Hi2}).
For any $\xi\in\cH$ write $q(\xi)$ for the right hand side of \eqref{F-3.10}. Let $h:=a+b$ so
that $\cD(h^{1/2})=\cD(a^{1/2})\cap\cD(b^{1/2})$ (see Lemma \ref{L-3.4}).
Assume $\xi\in\cD(a^{1/2})\cap\cD(b^{1/2})$.
From \eqref{F-3.7} for parallel sum we have
\begin{align*}
q_{a[:]b}(\xi)&=\|(a[:]b)^{1/2}\xi\|^2=\|(T_{a/h}:T_{b/h})^{1/2}h^{1/2}\xi\|^2 \\
&=\inf\{\|T_{a/h}^{1/2}\eta\|^2+\|T_{b/h}^{1/2}\zeta\|^2;\,
\eta,\zeta\in\cH,\,\eta+\zeta=h^{1/2}\xi\}
\end{align*}
thanks to \eqref{F-3.11}.
For every $\eta,\zeta\in\cH$ with $\eta+\zeta=h^{1/2}\xi$ choose a sequence
$\eta_n\in\cD(h^{1/2})$ such that $h^{1/2}\eta_n\to s(h)\eta$. Let
$\zeta_n:=\xi-\eta_n\in\cD(h^{1/2})$; then $\eta_n+\zeta_n=\xi$ and
$$
h^{1/2}\zeta_n=h^{1/2}\xi-h^{1/2}\eta_n\,\longrightarrow\,
h^{1/2}\xi-s(h)\eta=s(h)(h^{1/2}\xi-\eta)=s(h)\zeta.
$$
We hence find that
\begin{align*}
\|T_{a/h}^{1/2}\eta\|^2+\|T_{b/h}^{1/2}\zeta\|^2
&=\|T_{a/h}^{1/2}s(h)\eta\|^2+\|T_{b/h}^{1/2}s(h)\zeta\|^2 \\
&=\lim_{n\to\infty}\bigl(\|T_{a/h}^{1/2}h^{1/2}\eta_n\|^2
+\|T_{b/h}^{1/2}h^{1/2}\zeta_n\|^2\bigr) \\
&=\lim_{n\to\infty}(q_a(\eta_n)+q_b(\zeta_n))
\ge q(\xi),
\end{align*}
which implies $q_{a[:]b}(\xi)\ge q(\xi)$. Conversely, let $\eta,\zeta\in\cH$ with
$\eta+\zeta=\xi$. If $\eta\not\in\cD(a^{1/2})$ or $\zeta\not\in\cD(b^{1/2})$, then
$q_a(\eta)+q_b(\zeta)=\infty$. So we may assume that $\eta\in\cD(a^{1/2})$ and
$\zeta\in\cD(b^{1/2})$. Then, since $\xi\in\cD(a^{1/2})\cap\cD(b^{1/2})$, we must have
$\eta,\zeta\in\cD(a^{1/2})\cap\cD(b^{1/2})=\cD(h^{1/2})$ and therefore
\begin{align*}
q_a(\eta)+q_b(\zeta)&=\|T_{a/h}^{1/2}h^{1/2}\eta\|^2+\|T_{b/h}^{1/2}h^{1/2}\zeta\|^2 \\
&\ge\|(T_{a/h}:T_{b/h})^{1/2}h^{1/2}\xi\|^2=\|(a[:]b)^{1/2}\xi\|^2.
\end{align*}
Hence $q(\xi)\ge q_{a[:]b}(\xi)$ follows and the assertion has been shown.
\end{proof}

\begin{lemma}\label{L-3.22}
Let $a,b\in\overline\cM_+$. Then $q_{a[:]b}$ coincides with $q_a:q_b$ given in Definition
\ref{D-2.1}. Hence we have $a[:]b=a:b$.
\end{lemma}

\begin{proof}
First, assume that $\lambda^{-1}b\le a\le\lambda b$ for some $\lambda>0$. Then, from
\eqref{F-3.8} it is immediately verified that $(1:\lambda^{-1})a\le a[:]b\le(1:\lambda)a$, so
we have $\cD(a^{1/2})=\cD(b^{1/2})=\cD((a[:]b)^{1/2})$. Hence both sides of \eqref{F-3.10} are
$\infty$ when $\xi\in\cH\setminus\cD(a^{1/2})$, so the equality \eqref{F-3.10} holds for all
$\xi\in\cH$. This means that when $\lambda^{-1}b\le a\le\lambda b$, the quadratic form
defined as in \eqref{F-2.3} (for $q_a,q_b$), i.e., the right hand side of \eqref{F-3.10}
for all $\xi\in\cH$ is equal to $q_{a[:]b}$. Since the latter is lower semi-continuous on $\cH$,  the
procedure of $\displaystyle\inf_{\xi_n\to\xi}\liminf_{n\to\infty}$ in \eqref{F-2.4} is not
needed, and by Theorem \ref{T-2.2} we have $q_{a[:]b}=q_a:q_b$ in this case.

For general $a,b\in\overline\cM_+$ let $h:=a+b$. The convergence in \eqref{F-3.9} for parallel sum
is rewritten as
$$
q_{a[:]b}=\Inf_{\eps>0}q_{(a+\eps h)[:](b+\eps h)},
$$
as well as $q_a=\Inf_{\eps>0}q_{a+\eps h}$ and $q_b=\Inf_{\eps>0}q_{b+\eps h}$.
For each $\eps>0$, since $\lambda^{-1}(b+\eps h)\le a+\eps h\le\lambda(b+\eps h)$ for some
$\lambda>0$, it follows from the above case that
$$
q_{(a+\eps h)[:](b+\eps h)}=q_{a+\eps h}:q_{b+\eps h}.
$$
Hence by \eqref{F-2.6} we obtain
$$
q_{a[:]b}=\Inf_{\eps>0}\,(q_{a+\eps h}:q_{b+\eps h})
=\Bigl(\Inf_{\eps>0}q_{a+\eps h}\Bigr):\Bigl(\Inf_{\eps>0}q_{b+\eps h}\Bigr)
=q_a:q_b,
$$
and the assertion has been shown.
\end{proof}

Our main result of the subsection is the following:

\begin{theorem}\label{T-3.23}
For every $a,b\in\overline\cM_+$ we have $q_{a[\sigma]b}=q_a\sigma q_b$ and hence
$$
a[\sigma]b=a\sigma b.
$$
\end{theorem}

\begin{proof}
Consider the integral expressions in \eqref{F-2.8} and \eqref{F-2.9} for $\sigma$. First assume
$\alpha,\beta>0$. Let $h:=a+b$. For every $\xi\in\cD(h^{1/2})=\cD(a^{1/2})\cap\cD(b^{1/2})$ we
have
\begin{align*}
q_{a[\sigma]b}(\xi)&=\|(a[\sigma]b)^{1/2}\xi\|^2
=\bigl((T_{a/h}\sigma T_{b/h})h^{1/2}\xi,h^{1/2}\xi\bigr) \\
&=\alpha\|T_{a/h}^{1/2}h^{1/2}\xi\|^2+\beta\|T_{b/h}^{1/2}h^{1/2}\xi\|^2 \\
&\qquad+\int_{(0,\infty)}{1+t\over t}\,\|((tT_{a/h}):T_{b/h})^{1/2}h^{1/2}\xi\|^2
\,d\mu(t) \\
&=\alpha q_a(\xi)+\beta q_b(\xi)+\int_{(0,\infty)}{1+t\over t}\,q_{(ta)[:]b}(\xi)\,d\mu(t) \\
&=\alpha q_a(\xi)+\beta q_b(\xi)+\int_{(0,\infty)}{1+t\over t}((tq_a):q_b)(\xi)\,d\mu(t) \\
&=(q_a\sigma q_b)(\xi).
\end{align*}
In the above we have used $q_{(ta)[:]b}=(tq_a):q_b$ due to Lemma \ref{L-3.22}. Since
$a[\sigma]b\ge h^{1/2}(\alpha T_{a/h}+\beta T_{b/h})h^{1/2}=\alpha a+\beta b$, note that
$q_{a[\sigma]b}(\xi)=\infty$ for all $\xi\in\cH\setminus\cD(h^{1/2})$. Since
$(q_a\sigma q_b)(\xi)\ge\alpha q_a(\xi)+\beta q_b(\xi)$, note also that
$(q_a\sigma q_b)(\xi)=\infty$ for all $\xi\in\cH\setminus\cD(h^{1/2})$. Therefore,
$q_{a[\sigma]b}(\xi)=(q_a\sigma q_b)(\xi)$ for all $\xi\in\cH$.

Next, for general $\sigma$ with the representing operator monotone function $f$ on $(0,\infty)$
we consider the connection $\sigma_1$ with representing function $f_1(t):=f(t)+1+t$. The above
proved case gives
\begin{align}\label{F-3.12}
q_{a[\sigma_1]b}=q_a\sigma_1q_b.
\end{align}
We moreover have
\begin{align}\label{F-3.13}
q_a\sigma_1q_b=q_a\sigma q_b+q_a+q_b=q_{a\sigma b}+q_a+q_b=q_{a\sigma b+a+b},
\end{align}
where the second equality is Definition \ref{D-3.15} and the last equality is due to Lemma
\ref{L-3.4}. On the other hand, letting $h:=a+b$, by the definition \eqref{F-3.7} we have
\begin{align*}
(a+\eps h)[\sigma_1](b+\eps h)
&=h^{1/2}(T_{(a+\eps h)/h}\sigma_1T_{(b+\eps h)/h})h^{1/2} \\
&=h^{1/2}(T_{(a+\eps h)/h}\sigma T_{(b+\eps h)/h}+T_{(a+\eps h)/h}+T_{(b+\eps h)/h})h^{1/2} \\
&=(a+\eps h)[\sigma](b+\eps h)+(a+\eps h)+(b+\eps h)
\end{align*}
for every $\eps>0$. Hence by Lemma \ref{L-3.19} and Theorem \ref{T-3.10}, letting
$\eps\searrow0$ we have $a[\sigma_1]b=a[\sigma]b+a+b$, that is,
\begin{align}\label{F-3.14}
q_{a[\sigma_1]b}=q_{a[\sigma]b+a+b}.
\end{align}
Combining \eqref{F-3.12}--\eqref{F-3.14} implies that $a[\sigma]b+a+b=a\sigma b+a+b$ (in
$\overline\cM_+$), from which $a[\sigma]b=a\sigma b$ follows.
\end{proof}

From now on we will use the same symbol $\sigma$ for both definitions of Definitions
\ref{D-3.15} and \ref{D-3.16}. From the definition \eqref{F-3.7} the following is clear:

\begin{proposition}\label{P-3.24}
Let $\sigma_1,\sigma_2$ be connections with the representing operator monotone functions
$f_1,f_2$ on $(0,\infty)$  respectively. If $f_1(t)\le f_2(t)$ for all $t\in(0,\infty)$, then
$a\sigma_1b\le a\sigma_2b$ for all $a,b\in\overline\cM_+$.
\end{proposition}

In particular, for every symmetric (i.e., $\tilde\sigma=\sigma$) operator mean $\sigma$ we have
\begin{align}\label{F-3.15}
2(a:b)\le a\sigma b\le{a+b\over2},\qquad a,b\in\overline\cM_+.
\end{align}
Furthermore, we note that \eqref{F-3.15} holds for general positive forms $\phi,\psi$, that is,
$2(\phi:\psi)\le\phi\sigma\psi\le(\phi+\psi)/2$. (The proof of this is not difficult from the
definition \eqref{F-2.9} and left to the reader.) But we do not know whether Proposition
\ref{P-3.24} holds for positive forms in general.

%%%%%%%%%%%%%%%%%%%%%%%%%%%%%%%%%%%%%%%%%%%%%%%%%%
%%%%%%%%%%%%%%%%%%%%% 3.3 %%%%%%%%%%%%%%%%%%%%%%%%%%%
%%%%%%%%%%%%%%%%%%%%%%%%%%%%%%%%%%%%%%%%%%%%%%%%%%
\subsection{Properties}\label{S-3.3}

In this subsection we extend properties of connections for bounded positive operators to those
for $\tau$-measurable operators.

Let $\sigma$ be any connection corresponding to an operator monotone function $f>0$ on
$[0,\infty)$, originally defined on $B(\cH)_+$ in \cite{KA} and extended to $\overline\cM_+$
in the previous subsection.

\begin{theorem}\label{T-3.25}
Let $a,b,a_i,b_i\in\overline\cM_+$.
\begin{itemize}
\item[\rm(1)] \emph{(Monotonicity).}\enspace
If $a_1\le a_2$ and $b_1\le b_2$, then $a_1\sigma b_1\le a_2\sigma b_2$.
\item[\rm(2)] \emph{(Decreasing convergence).}\enspace
If $a_n,b_n\in\overline\cM_+$ $(n\in\bN)$, $a_n\searrow a$ and $b_n\searrow b$ in the strong
resolvent sense, then $a_n\sigma b_n\searrow a\sigma b$ in the strong resolvent sense and hence
\begin{align}\label{F-3.16}
\Inf_n\,(a_n\sigma b_n)=\Bigl(\Inf_na_n\Bigr)\sigma\Bigl(\Inf_nb_n\Bigr).
\end{align}
\item[\rm(3)] \emph{(Concavity).}\enspace
We have $(a_1+a_2)\sigma(b_1+b_2)\ge a_1\sigma b_1+a_2\sigma b_2$.
\item[\rm(4)] \emph{(Transpose).}\enspace
We have $a\tilde\sigma b=b\sigma a$.
\end{itemize}
\end{theorem}

\begin{proof}
All of (1)--(4) are immediately seen by applying the same properties for bounded positive
operators to $T_{a/h}\sigma T_{b/h}$ in the definition \eqref{F-3.7}. The properties
(1), (3) and (4) are also obvious by Definition \ref{D-3.15} since connections of positive
forms satisfy those (see \cite{Ko6} and Proposition \ref{P-2.6}).  We here prove (2) for
instance. Let $h:=a_1+b_1$. Since $T_{a_n/h}\searrow T_{a/h}$ and $T_{b_n/h}\searrow T_{b/h}$
strongly by Lemma \ref{L-3.9}, one has $T_{a_n/h}\sigma T_{b_n/h}\searrow T_{a/h}\sigma T_{b/h}$
strongly. Therefore, by Lemma \ref{L-3.9} again one has
$$
a_n\sigma b_n=h^{1/2}(T_{a_n/h}\sigma T_{b_n/h})h^{1/2}
\,\searrow\, h^{1/2}(T_{a/h}\sigma T_{b/h})h^{1/2}=a\sigma b
$$
in the strong resolvent sense, which is also written as \eqref{F-3.16} due to Lemma \ref{L-2.4}.
\end{proof}

Let $\overline\cM_{++}:=\{a\in\overline\cM_+;\,a^{-1}\in\overline\cM_+\}$. The next
proposition extends the formula $A\sigma^*B=(A^{-1}\sigma B^{-1})^{-1}$ for $A,B\in B(\cH)_{++}$
to $\tau$-measurable operators.

\begin{proposition}[Adjoint]\label{P-3.26}
Let $\sigma^*$ be the adjoint of $\sigma$. For every $a,b\in\overline\cM_+$ we have
$(q_a\sigma^*q_b)^{-1}=q_a^{-1}\sigma q_b^{-1}$. Moreover, if $a,b\in\overline\cM_{++}$, then
$a\sigma^*b,a^{-1}\sigma b^{-1}\in\overline\cM_{++}$ and
$$
a\sigma^*b=(a^{-1}\sigma b^{-1})^{-1}.
$$
\end{proposition}

\begin{proof}
Let $a,b\in\overline\cM_+$. With the spectral decomposition
$a=\int_0^\infty\lambda\,de_\lambda$ set $a_n:=(1/n)e_n+\int_{(1/n,\infty)}\lambda\,de_\lambda$
and similarly $b_n$ for $n\in\bN$. Since $a_n\searrow a$ and $b_n\searrow b$ in the strong
resolvent sense (even in the measure topology), it follows from Theorem \ref{T-3.25},(2)
that $a_n\sigma^*b_n\searrow a\sigma^*b$ in the strong resolvent sense, that is,
$q_a\sigma^*q_b=\Inf_n(q_{a_n}\sigma^*q_{b_n})$. Therefore,
$$
(q_a\sigma^*q_b)^{-1}=\sup_n(q_{a_n}\sigma^*q_{b_n})^{-1}
=\sup_n(q_{a_n}^{-1}\sigma q_{b_n}^{-1})\le q_a^{-1}\sigma q_b^{-1},
$$
where we have used \cite[Lemma 16,(ii)]{Ko6} for the first equality and Proposition \ref{P-2.7}
for the second. Since Proposition \ref{P-2.8} gives the reverse inequality, the first assertion
follows.

Next, assume $a,b\in\overline\cM_{++}$. Then
$a\sigma^*b,a^{-1}\sigma b^{-1}\in\overline\cM_+$. In this case, the first assertion means that
$q_{(a\sigma^*b)^{-1}}=q_{a^{-1}\sigma b^{-1}}$ and so $(a\sigma^*b)^{-1}=a^{-1}\sigma b^{-1}$.
This implies the latter assertion.
\end{proof}

We note that the convergence $a_n\sigma b_n\searrow a\sigma b$ in the measure topology does not
hold in general even when $a_n\searrow a$ and $b_n\searrow b$ in the measure topology. In fact,
it is known \cite[Example 4.5]{Hi1} that there are $A,B\in B(\cH)_+$ such that
$\lim_{\eps\searrow0}\|(A+\eps I)\#B\|>\|A\#B\|$. The situation being so, the next result
seems rather best possible for the convergence in the measure topology.

\begin{proposition}\label{P-3.27}
Let $a_n,b_n\in\overline\cM_+$ be such that $a_n\searrow a$ and $b_n\searrow b$ in the strong
resolvent sense. Assume that one of the following conditions is satisfied:
\begin{itemize}
\item[\rm(1)] $a_1,b_1\in\overline\fS$,
\item[\rm(2)] $a_1\in\overline\fS$ and $\lim_{t\to\infty}f(t)/t=0$,
\item[\rm(3)] $b_1\in\overline\fS$ and $f(0^+)=0$.
\end{itemize}
Then $a_n\sigma b_n\in\overline\fS$ for all $n$ and $a_n\sigma b_n\searrow a\sigma b$ in the
measure topology.
\end{proposition}

\begin{proof}
By Theorem \ref{T-3.25},(2) 
we have $a_n\sigma b_n\searrow a\sigma b$ in the strong
resolvent sense. Once $a_1\sigma b_1\in\overline\fS$ is shown, by Theorem \ref{T-3.8} we
have $a_n\sigma b_n\searrow a\sigma b$ in the measure topology. The case (1) is obvious since
$a_1\sigma b_1\le\lambda(a_1+b_1)$ for some $\lambda>0$ by Lemmas \ref{L-3.13} and \ref{L-3.2}.
Assume (3). Choose an $h\in\overline\cM_+$ such that $h\ge1$ and $a_1+b_1\le h$. We write
$a_1\sigma b_1=h^{1/2}(T_{a_1/h}\sigma T_{b_1/h})h^{1/2}$. Since $b_1=h^{1/2}T_{b_1/h}h^{1/2}$
is in $\overline\fS$ and $h\ge1$, it follows that $T_{b_1/h}\in\cM_+$ is in $\overline\fS$ as
well. Note that
\begin{align}\label{F-3.17}
T_{a_1/h}\sigma T_{b_1/h}\le1\sigma T_{b_1/h}=f(T_{b_1/h}).
\end{align}
Since $f(0^+)=0$ and $f$ is strictly increasing on $(0,\infty)$, by \cite[Lemma 2.5,(iv)]{FK}
we have $\mu_t(f(T_{b_1/h}))=f(\mu_t(T_{b_1/h}))$ for all $t>0$, which implies that
$f(T_{b_1/h})\in\overline\fS$ and hence $T_{a_1/h}\sigma T_{b_1/h}\in\overline\fS$ thanks to
\eqref{F-3.17}. Therefore, $a_1\sigma b_1\in\overline\fS$ follows as well. The case (2) is
shown immediately from (3) since $a_1\sigma b_1=b_1\tilde\sigma a_1$ and
$\tilde f(0^+)=\lim_{t\to\infty}f(t)/t=0$, where $\tilde f$ is the transpose of $f$.
\end{proof}

\begin{theorem}[Transformer inequality]\label{T-3.28}
For every $a,b\in\overline\cM_+$ and $c\in\overline\cM$ we have
\begin{align}\label{F-3.18}
c^*(a\sigma b)c\le(c^*ac)\sigma(c^*bc).
\end{align}
\end{theorem}

We first prove the case of parallel sum as a lemma.

\begin{lemma}\label{L-3.29}
For every $a,b\in\overline\cM_+$ and $c\in\overline\cM$ we have $c^*(a:b)c\le(c^*ac):(c^*bc)$.
\end{lemma}

\begin{proof}
By Theorem \ref{T-2.2} it suffices to prove that, for every $\eta,\zeta\in\cH$,
$$
q_{c^*(a:b)c}(\eta+\zeta)\le q_{c^*ac}(\eta)+q_{c^*bc}(\zeta).
$$
Let the above right hand side be finite, that is, $\eta\in\cD((c^*ac)^{1/2})$ and
$\zeta\in\cD((c^*bc)^{1/2})$. Take the polar decomposition $a^{1/2}c=v|a^{1/2}c|$. Since
$(c^*ac)^{1/2}=|a^{1/2}c|=v^*a^{1/2}c$ (the strong product), note that
$\{\xi\in\cH;\,\xi\in\cD(c),\,c\xi\in\cD(a^{1/2})\}$ is a core of $(c^*ac)^{1/2}$. Hence one
can choose a sequence $\eta_n$ such that $\eta_n\to\eta$, $\eta_n\in\cD(c)$,
$c\eta_n\in\cD(a^{1/2})$, and $q_{c^*ac}(\eta_n)\to q_{c^*ac}(\eta)$. For each $\eta_n$ one has
also
\begin{align}\label{F-3.19}
q_{c^*ac}(\eta_n)=\|v^*a^{1/2}c\eta_n\|^2=\|a^{1/2}c\eta_n\|^2=q_a(c\eta_n).
\end{align}
Similarly, one can choose a sequence $\zeta_n$ such that $\zeta_n\to\zeta$, $\zeta_n\in\cD(c)$,
$c\zeta_n\in\cD(b^{1/2})$, $q_{c^*bc}(\zeta_n)\to q_{c^*bc}(\zeta)$, and
$q_{c^*bc}(\zeta_n)=q_b(c\zeta_n)$. Since $a:b\le a,b$, note that
$c\eta_n,c\zeta_n\in\cD((a:b)^{1/2})$ and a similar argument as in \eqref{F-3.19} gives
$q_{c^*(a:b)c}(\eta_n+\zeta_n)=q_{a:b}(c\eta_n+c\zeta_n)$. From the lower semi-continuity of
$q_{c^*(a:b)c}$ it follows that
\begin{align*}
q_{c^*(a:b)c}(\eta+\zeta)
&\le\liminf_{n\to\infty}q_{c^*(a:b)c}(\eta_n+\zeta_n)
=\liminf_{n\to\infty}q_{a:b}(c\eta_n+c\zeta_n) \\
&\le\liminf_{n\to\infty}(q_a(c\eta_n)+q_b(c\zeta_n))
=q_{c^*ac}(\eta)+q_{c^*bc}(\zeta)
\end{align*}
due to Theorem \ref{T-2.2} for the latter inequality.
\end{proof}

\begin{proof}[Proof of Theorem \ref{T-3.28}]
To prove this, we use the definition \eqref{F-2.9} and Definition \ref{D-3.15}. Assume first that
$\alpha,\beta>0$ in \eqref{F-2.9}. We need to prove that
\begin{align}\label{F-3.20}
q_{c^*(a\sigma b)c}(\xi)\le q_{(c^*ac)\sigma(c^*bc)}(\xi),\qquad\xi\in\cH.
\end{align}
Let the right hand side be finite. Then by the assumption $\alpha,\beta>0$ we have
$q_{c^*ac}(\xi)<\infty$ and $q_{c^*bc}(\xi)<\infty$ so that $q_{c^*(a+b)c}(\xi)<\infty$ by
Lemma \ref{L-3.4}. This implies that $\xi\in\cD((a+b)^{1/2}c)$. Hence one can choose a sequence
$\xi_n$ such that $\xi_n\in\cD(c)$, $c\xi_n\in\cD((a+b)^{1/2})$ and
$\|(a+b)^{1/2}c(\xi_n-\xi)\|\to0$. By Lemmas \ref{L-3.13} and \ref{L-3.4} one has
\begin{align*}
q_{(c^*ac)\sigma(c^*bc)}(\xi_n-\xi)&\le\lambda(q_{c^*ac}+q_{c^*bc})(\xi_n-\xi)
=\lambda q_{c^*(a+b)c}(\xi_n-\xi) \\
&=\lambda\|(a+b)^{1/2}c(\xi_n-\xi)\|^2\,\longrightarrow\,0.
\end{align*}
Moreover, since $q_{a\sigma b}\le\lambda q_{a+b}$ means that $a\sigma b\le\lambda(a+b)$ (in the
form sense), one has
$$
q_{c^*(a\sigma b)c}(\xi_n-\xi)=\|(a\sigma b)^{1/2}c(\xi_n-\xi)\|^2
\le\lambda\|(a+b)^{1/2}c(\xi_n-\xi)\|^2\,\longrightarrow\,0.
$$
These imply that
$$
q_{(c^*ac)\sigma(c^*bc)}(\xi)=\lim_{n\to\infty}q_{(c^*ac)\sigma(c^*bc)}(\xi_n),\qquad
q_{c^*(a\sigma b)c}(\xi)=\lim_{n\to\infty}q_{c^*(a\sigma b)c}(\xi_n).
$$
Hence we may prove \eqref{F-3.20} for $\xi$ such that $\xi\in\cD(c)$ and
$c\xi\in\cD((a+b)^{1/2})$. For each such $\xi$ we have by \eqref{F-2.9}
\begin{align*}
q_{c^*(a\sigma b)c}(\xi)&=\|(a\sigma b)^{1/2}c\xi\|^2=q_{a\sigma b}(c\xi) \\
&=\alpha q_a(c\xi)+\beta q_b(c\xi)
+\int_{(0,\infty)}{1+t\over t}\,q_{(ta):b}(c\xi)\,d\mu(t) \\
&=\alpha q_{c^*ac}(\xi)+\beta q_{c^*bc}(\xi)
+\int_{(0,\infty)}{1+t\over t}\,q_{c^*((ta):b)c}(\xi)\,d\mu(t) \\
&\le\alpha q_{c^*ac}(\xi)+\beta q_{c^*bc}(\xi)
+\int_{(0,\infty)}{1+t\over t}\,q_{(c^*(ta)c):(c^*bc)}(\xi)\,d\mu(t) \\
&=q_{(c^*ac)\sigma(c^*bc)}(\xi),
\end{align*}
where Lemma \ref{L-3.29} has been used for the above inequality. Therefore, \eqref{F-3.20} has
been shown so that \eqref{F-3.18} follows.

For general $\sigma$ we use the same trick as in the proof of Theorem \ref{T-3.23}. Let
$\sigma_1$ be the connection corresponding to $f(t)+1+t$. The above proved case gives
$c^*(a\sigma_1b)c\le(c^*ac)\sigma_1(c^*bc)$. Since $c^*(a\sigma_1b)c=c^*(a\sigma b+a+b)c$ and
$(c^*ac)\sigma_1(c^*bc)=(c^*ac)\sigma(c^*bc)+c^*ac+c^*bc$ as in \eqref{F-3.13}, we obtain
\eqref{F-3.18} for general $\sigma$.
\end{proof}

\begin{remark}\label{R-3.30}\rm
When $\phi,\psi$ are general positive forms and $c$ is a bounded operator, the transformer
inequality
$$
c^*(\phi:\psi)c\le(c^*\phi c):(c^*\psi c)
$$
was shown in \cite[Corollary 7]{Ko6} 
with the definition $(c^*\phi c)(\xi):=\phi(c\xi)$, $\xi\in\cH$.
This transformer inequality immediately extends
to $\phi\sigma\psi$ from the definition \eqref{F-2.9}. Note that if $a\in\overline\cM_+$ and
$c\in\cM$, then $q_a(c\xi)=q_{c^*ac}(\xi)$ for all $\xi\in\cH$ since
$\cD(a^{1/2}c)=\{\xi\in\cH;\,c\xi\in\cD(a^{1/2})\}$ (i.e., $a^{1/2}c$ is closed without taking
closure) in this case. Therefore, Theorem \ref{T-3.28} is included in \cite{Ko6} when $c\in\cM$.
\end{remark}

The next theorem shows the equality case in the transformer inequality \eqref{F-3.18}.

\begin{theorem}\label{T-3.31}
Let $a, b\in\overline\cM_+$ and $c\in\overline\cM$ be such that $s(a+b)\le s(cc^*)$. Then
$$
c^*(a\sigma b)c=(c^*ac)\sigma(c^*bc)
$$
for any connection $\sigma$.
\end{theorem}

\begin{proof}
Let $h:=a+b$. For any connection $\sigma$, by Definition \ref{D-3.16} we write
$$
c^*(a\sigma b)c=c^*h^{1/2}(T_{a/h}\sigma T_{b/h})h^{1/2}c.
$$
Let $c^*h^{1/2}=vk^{1/2}$ be the polar decomposition with $k^{1/2}:=|c^*h^{1/2}|$ so that
$k=h^{1/2}cc^*h^{1/2}$ and $s(k)=s(h)$ since $s(h)\le s(cc^*)$. Furthermore, set
$\tilde a:=k^{1/2}T_{a/h}k^{1/2}$ and $\tilde b:=k^{1/2}T_{b/h}k^{1/2}$. We then have
$\tilde a+\tilde b=k^{1/2}(T_{a/h}+T_{b/h})k^{1/2}=k^{1/2}s(h)k^{1/2}=k$,
$T_{\tilde a/k}=T_{a/h}$ and $T_{\tilde b/k}=T_{b/h}$ since
$s(T_{a/h}),s(T_{b/h})\le s(h)=s(k)$ (see Definition \ref{D-3.5}). Therefore,
$$
\tilde a\sigma\tilde b=k^{1/2}(T_{\tilde a/k}\sigma T_{\tilde b/k})k^{1/2}
=v^*c^*h^{1/2}(T_{a/h}\sigma T_{b/h})h^{1/2}cv=v^*c^*(a\sigma b)cv
$$
so that
$$
c^*(a\sigma b)c=v(\tilde a\sigma\tilde b)v^*.
$$
Since
$$
v\tilde av^*=vk^{1/2}T_{a/h}k^{1/2}v^*=c^*h^{1/2}T_{a/h}h^{1/2}c=c^*ac
$$
and similarly $v\tilde bv^*=c^*bc$, it suffices to show that
\begin{align}\label{F-3.21}
v(\tilde a\sigma\tilde b)v^*=(v\tilde av^*)\sigma(v\tilde bv^*).
\end{align}
Note that $v\tilde av^*+v\tilde bv^*=c^*hc=vkv^*$ and $(vkv^*)^{1/2}=vk^{1/2}v^*$, and
it is immediate to see that $T_{v\tilde av^*/vkv^*}=vT_{\tilde a/k}v^*$ and
$T_{v\tilde bv^*/vkv^*}=vT_{\tilde b/k}v^*$. Hence \eqref{F-3.21} follows since
\begin{align*}
(v\tilde av^*)\sigma(v\tilde bv^*)
&=vk^{1/2}v^*((vT_{\tilde a/k}v^*)\sigma(vT_{\tilde b/k}v^*))vk^{1/2}v^* \\
&=vk^{1/2}v^*v(T_{\tilde a/k}\sigma T_{\tilde b/k})v^*vk^{1/2}v^* \\
&=vk^{1/2}(T_{\tilde a/k}\sigma T_{\tilde b/k})k^{1/2}v^*=v(\tilde a\sigma\tilde b)v^*,
\end{align*}
where the second equality in the above is due to Lemma \ref{L-3.17} since
$s(T_{\tilde a/k}+T_{\tilde b/k})=s(\tilde a+\tilde b)=s(k)=s(v^*v)$.
\end{proof}

In particular, when $a\in\overline\cM_{++}$, Theorem \ref{T-3.31} shows
\begin{align}\label{F-3.22}
a\sigma b=a^{1/2}f(a^{-1/2}ba^{-1/2})a^{1/2},
\end{align}
which is the generalization of the familiar expression of $A\sigma B$ for $A,B\in B(\cH)_+$
with $A$ invertible.

Now, we consider $\cM\otimes M_2$, the tensor product of $\cM$ and the $2\times2$ matrix
algebra $M_2=M_2(\bC)$, equipped with the faithful semi-finite normal trace $\tau\otimes\Tr$,
where $\Tr$ is the usual trace on $M_2$. The space $\overline{\cM\otimes M_2}$ of
$\tau\otimes\Tr$-measurable operators is 
identified with $\overline\cM\otimes M_2$
in the natural way. 

The following result is the $\tau$-measurable operator version of
\cite[Theorem I.1]{An2}:

\begin{lemma}\label{L-3.32}
For every $a,b,c\in\overline\cM_+$, $\begin{bmatrix}a&c\\c^*&b\end{bmatrix}\ge0$ in
$\overline{\cM\otimes M_2}$ if and only if $a,b\in\overline\cM_+$ and $c=a^{1/2}zb^{1/2}$
for some contraction $z\in\cM$.
\end{lemma}

\begin{proof}
Assume that $a,b\ge0$ and $c=a^{1/2}zb^{1/2}$ with $z\in\cM$, $\|z\|\le1$. Then
$$
\begin{bmatrix}a&c\\c^*&b\end{bmatrix}
=\begin{bmatrix}a&a^{1/2}zb^{1/2}\\b^{1/2}z^*a^{1/2}&b\end{bmatrix}
=\begin{bmatrix}a^{1/2}&0\\0&b^{1/2}\end{bmatrix}\begin{bmatrix}1&z\\z^*&1\end{bmatrix}
\begin{bmatrix}a^{1/2}&0\\0&b^{1/2}\end{bmatrix}\ge0.
$$
Conversely, assume that $\begin{bmatrix}a&c\\c^*&b\end{bmatrix}\ge0$. Let $e:=s(a)$ and
$f:=s(b)$. Setting $\begin{bmatrix}h&m\\m^*&k\end{bmatrix}=
\begin{bmatrix}a&c\\c^*&b\end{bmatrix}^{1/2}$ we write
$$
\begin{bmatrix}a&c\\c^*&b\end{bmatrix}=\begin{bmatrix}h&m\\m^*&k\end{bmatrix}^2
=\begin{bmatrix}h^2+mm^*&hm+mk\\m^*h+km^*&m^*m+k^2\end{bmatrix},
$$
so that
$$
a=h^2+mm^*,\quad b=m^*m+k^2,\quad c=hm+mk.
$$
Since $h^2,mm^*\le a$, there are $x,x_1\in e\cM e$ such that $h=xa^{1/2}$ and $|m^*|=x_1a^{1/2}$.
Since $k^2,m^*m\le b$, there are $y,y_1\in f\cM f$ such that $k=yb^{1/2}$ and $|m|=y_1b^{1/2}$.
With the polar decompositions $m=v|m|$ and $m^*=w|m^*|$ we write
$$
c=hv|m|+|m^*|w^*k=a^{1/2}x^*vy_1b^{1/2}+a^{1/2}x_1^*w^*yb^{1/2}=a^{1/2}zb^{1/2}
$$
with $z:=x^*vy_1+x_1^*w^*y\in e\cM f$. Therefore,
$$
0\le\begin{bmatrix}a&c\\c^*&b\end{bmatrix}
=\begin{bmatrix}a&a^{1/2}zb^{1/2}\\b^{1/2}z^*a^{1/2}&b\end{bmatrix}
=\begin{bmatrix}a^{1/2}&0\\0&b^{1/2}\end{bmatrix}\begin{bmatrix}e&z\\z^*&f\end{bmatrix}
\begin{bmatrix}a^{1/2}&0\\0&b^{1/2}\end{bmatrix},
$$
and it remains to show $\|z\| \leq 1$.
For every $\eps>0$, multiplying $\begin{bmatrix} \! (\eps+a^{1/2})^{-1}&\!\!\! 0\\ \! 0& \!\!\! (\eps+b^{1/2})^{-1}
\end{bmatrix}$ to the above both sides, we have
$$
0\le\begin{bmatrix}{a^{1/2}\over\eps+a^{1/2}}&0\\0&{b^{1/2}\over\eps+b^{1/2}}\end{bmatrix}
\begin{bmatrix}e&z\\z^*&f\end{bmatrix}
\begin{bmatrix}{a^{1/2}\over\eps+a^{1/2}}&0\\0&{b^{1/2}\over\eps+b^{1/2}}\end{bmatrix}.
$$
Letting $\eps\searrow0$ gives
$$
0\le\begin{bmatrix}e&0\\0&f\end{bmatrix}\begin{bmatrix}e&z\\z^*&f\end{bmatrix}
\begin{bmatrix}e&0\\0&f\end{bmatrix}=\begin{bmatrix}e&z\\z^*&f\end{bmatrix}
$$
thanks to $z\in e\cM f$. This implies $\begin{bmatrix}1&z\\z^*&1\end{bmatrix}\ge0$, which is
equivalent to $\|z\|\le1$.
\end{proof}

The following proposition extends the well-known characterization of the
geometric mean, proved by Pusz and Woronowicz \cite{PW} (see also \cite{An2}),
in terms of $2\times2$ operator matrices to the present setting of $\tau$-measurable operators:

\begin{proposition}\label{P-3.33}
For every $a,b\in\overline\cM_+$ we have
$$
a\#b=\max\biggl\{c\in\overline\cM_+;\,
\begin{bmatrix}a&c\\c&b\end{bmatrix}\ge0\ \mbox{in $\overline{\cM\otimes M_2}$}\biggr\}.
$$
\end{proposition}

\begin{proof}
Let $h:=a+b$ and $a^{1/2}=xh^{1/2}$, $b^{1/2}=yh^{1/2}$ with $x,y\in e\cM e$, where $e:=s(h)$.
Since $a\#b=h^{1/2}(T_{a/h}\#T_{b/h})h^{1/2}$, one has
\begin{align*}
\begin{bmatrix}a&a\#b\\a\#b&b\end{bmatrix}
&=\begin{bmatrix}h^{1/2}T_{a/h}h^{1/2}&
h^{1/2}(T_{a/h}\#T_{b/h})h^{1/2}\\h^{1/2}(T_{a/h}\#T_{b/h})h^{1/2}
&h^{1/2}T_{b/h}h^{1/2}\end{bmatrix} \\
&=\begin{bmatrix}h^{1/2}&0\\0&h^{1/2}\end{bmatrix}
\begin{bmatrix}T_{a/h}&T_{a/h}\#T_{b/h}\\T_{a/h}\#T_{b/h}&T_{b/h}\end{bmatrix}
\begin{bmatrix}h^{1/2}&0\\0&h^{1/2}\end{bmatrix}\ge0.
\end{align*}

Let $c\in\overline\cM_+$ be such that $\begin{bmatrix}a&c\\c&b\end{bmatrix}\ge0$. By Lemma
\ref{L-3.32} there is a contraction $z\in\cM$ such that $c=a^{1/2}zb^{1/2}$ and so
$c=h^{1/2}x^*zyh^{1/2}=h^{1/2}\hat zh^{1/2}$, where $\hat z:=x^*zy\in e\cM e$.
Since
$$
{h^{1/2}\over\eps+h^{1/2}}\,\hat z\,{h^{1/2}\over\eps+h^{1/2}}
=(\eps+h^{1/2})^{-1}c(\eps+h^{1/2})^{-1}\ge0,\qquad\eps>0,
$$
one has $\hat z\ge0$ by letting $\eps\searrow0$. Furthermore, since
\begin{align*}
0\le\begin{bmatrix}a&c\\c&b\end{bmatrix}
&=\begin{bmatrix}h^{1/2}T_{a/h}h^{1/2}&h^{1/2}\hat zh^{1/2}\\
h^{1/2}\hat zh^{1/2}&h^{1/2}T_{b/h}h^{1/2}\end{bmatrix} \\
&=\begin{bmatrix}h^{1/2}&0\\0&h^{1/2}\end{bmatrix}
\begin{bmatrix}T_{a/h}&\hat z\\\hat z&T_{b/h}\end{bmatrix}
\begin{bmatrix}h^{1/2}&0\\0&h^{1/2}\end{bmatrix}
\end{align*}
and $\begin{bmatrix}T_{a/h}&\hat z\\\hat z&T_{b/h}\end{bmatrix}\in
\begin{bmatrix}e&0\\0&e\end{bmatrix}(\cM\otimes M_2)
\begin{bmatrix}e&0\\0&e\end{bmatrix}$, it follows as above that
$\begin{bmatrix}T_{a/h}&\hat z\\\hat z&T_{b/h}\end{bmatrix}\ge0$, which implies that
$\hat z\le T_{a/h}\#T_{b/h}$. This gives $c\le h^{1/2}(T_{a/h}\#T_{b/h})h^{1/2}=a\#b$.
\end{proof}

The following is the extension of a similar characterization of the parallel sum
(see \cite{An2}) to the setting of $\tau$-measurable operators:

\begin{proposition}\label{P-3.34}
For every $a,b\in\overline\cM_+$,
$$
a:b=\max\biggl\{c\in\overline\cM_+;\,
\begin{bmatrix}a&0\\0&b\end{bmatrix}\ge\begin{bmatrix}c&c\\c&c\end{bmatrix}
\ \mbox{in $\overline{\cM\otimes M_2}$}\biggr\}.
$$
\end{proposition}

\begin{proof}
Let $h:=\begin{bmatrix}a&0\\0&b\end{bmatrix}$ and
$k:=\begin{bmatrix}a:b&a:b\\a:b&a:b\end{bmatrix}$. Then $h,k\ge0$ in
$\overline{\cM\otimes M_2}$ and
$$
h^{1/2}=\begin{bmatrix}a^{1/2}&0\\0&b^{1/2}\end{bmatrix},\qquad
k^{1/2}={1\over\sqrt2}\begin{bmatrix}(a:b)^{1/2}&(a:b)^{1/2}\\
(a:b)^{1/2}&(a:b)^{1/2}\end{bmatrix}.
$$
Hence $\cD(h^{1/2})=\cD(a^{1/2})\oplus\cD(b^{1/2})$ and
$\cD(k^{1/2})=\cD((a:b)^{1/2})\oplus\cD((a:b)^{1/2})$. Note that
$\cD(a^{1/2})\subseteq\cD((a:b)^{1/2})$ and $\cD(b^{1/2})\subseteq\cD((a:b)^{1/2})$ so that
$\cD(h^{1/2})\subseteq\cD(k^{1/2})$. For every $\eta\in\cD(a^{1/2})$ and $\zeta\in\cD(b^{1/2})$
we have
\begin{align*}
\|h^{1/2}(\eta\oplus\zeta)\|^2&=\|a^{1/2}\eta\|^2+\|b^{1/2}\zeta\|^2
=q_a(\eta)+q_b(\zeta) \\
&\ge q_{a:b}(\eta+\zeta)=\|k^{1/2}(\eta\oplus\zeta)\|^2.
\end{align*}
In the above we have used Theorem \ref{T-2.2} together with Definition \ref{D-3.15}.
Therefore, $h\ge k$ follows (see Lemma \ref{L-3.2}).

Let $c\in\overline\cM_+$ be such that $\begin{bmatrix}a&0\\0&b\end{bmatrix}\ge
\begin{bmatrix}c&c\\c&c\end{bmatrix}$. Let $z:=\begin{bmatrix}c&c\\c&c\end{bmatrix}$. Note that
$z^{1/2}={1\over\sqrt2}\begin{bmatrix}c^{1/2}&c^{1/2}\\c^{1/2}&c^{1/2}\end{bmatrix}$ and
$\cD(z^{1/2})=\cD(c^{1/2})\oplus\cD(c^{1/2})$. Hence $h\ge z$ means that
$\cD(a^{1/2})\subseteq\cD(c^{1/2})$, $\cD(b^{1/2})\subseteq\cD(c^{1/2})$ and for every
$\eta\in\cD(a^{1/2})$ and $\zeta\in\cD(b^{1/2})$,
$$
\|a^{1/2}\eta\|^2+\|b^{1/2}\zeta\|^2=\|h^{1/2}(\eta\oplus\zeta)\|^2
\ge\|z^{1/2}(\eta\oplus\zeta)\|^2=\|c^{1/2}(\eta+\zeta)\|^2.
$$
This implies that $q_c(\eta+\zeta)\le q_a(\eta)+q_b(\zeta)$ for all $\eta,\zeta\in\cH$. Hence
Theorem \ref{T-2.2} gives $q_c\le q_a:q_b$, that is, $c\le a:b$.
\end{proof}

%%%%%%%%%%%%%%%%%%%%%%%%%%%%%%%%%%%%%%%%%%%%%%%%%%%
%%%%%%%%%%%%%%%%%%% 3.4 %%%%%%%%%%%%%%%%%%%%%%%%%%%%%%
%%%%%%%%%%%%%%%%%%%%%%%%%%%%%%%%%%%%%%%%%%%%%%%%%%%
\subsection{Connections on $L^1(\cM,\tau)+\cM$}\label{S-3.4}

The subspace $L^1(\cM,\tau)+\cM$ of $\overline\cM$ appears in several situations related
to $(\cM,\tau)$, for instance, in interpolation theory (e.g., \cite{Ko0,Te2}), majorization
theory (e.g., \cite{Hi0}, \cite[Chap.~3]{LSZ}), symmetric operator ideals (e.g.,
\cite[Chap.~2]{LSZ}) and so on. Indeed, the integral $\int_0^s\mu_t(a)\,dt$ in (ii) below
is of fundamental importance  in a twofold sense. Firstly it is a natural continuous
analogue of the Ky Fan norm (see \cite{Bh}) for compact operators (i.e., the sum of the first
$s$ largest eigenvalues), and secondly  it is nothing but the $K$-functional in real
interpolation theory (see \cite{BL} and \cite[p.~289]{FK}):
$$
\inf\{ \|b\|_1+t\|c\|;\, a=b+c\} \ \left(=K(t,a;L^1(\cM,\tau),\cM) \right).
$$
In this subsection we present further results on connections $a\sigma b$ when
$a,b\in\overline\cM_+$ are in $L^1(\cM)+\cM$.

It is well-known \cite[Proposition 1.2]{Hi0} that the following conditions for
$a\in\overline\cM$ are equivalent:
\begin{itemize}
\item[(i)] $a\in L^1(\cM,\tau)+\cM$;
\item[(ii)] $\int_0^t\mu_s(a)\,ds<+\infty$ for some (equivalently, for all) $t>0$;
\item[(iii)] $|a|e_{(r,\infty)}(|a|)\in L^1(\cM,\tau)$ for some $r\ge0$.
\end{itemize}

In the rest of this subsection, for notational brevity we write $L^p$ for $L^p(\cM,\tau)$
and $(L^1+\cM)_+$ for $(L^1+\cM)\cap\overline\cM_+$. Note that
$L^1\cap\cM\subseteq L^p \subseteq L^1+\cM$ for all $p\in[1,\infty]$. Since $ax,xa\in L^1$
for any $a\in L^1+\cM$ and $x\in L^1\cap\cM$, one can define a duality pairing between
$L^1+\cM$ and $L^1\cap\cM$ by
$$
\<a,x\>_\tau:=\tau(ax)\ (=\tau(xa)),\qquad a\in L^1+\cM,\ x\in L^1\cap\cM.
$$
By the condition (iii) it is clear that $(L^1+\cM)_+=L^1_++\cM_+$ and so $\<a,x\>_\tau\ge0$ for
all $x\in(L^1+\cM)_+$ and $x\in L^1\cap\cM_+$. In particular, when $\tau(1)<+\infty$, of
course $L^1+\cM=L^1$ and $L^1\cap\cM=\cM$ so that the above duality pairing is the usual
duality between $L^1\cong\cM_*$ and $\cM$.

The first lemma might be well-known to experts while we give a proof for completeness.

\begin{lemma}\label{L-3.35}
Let $a,b\in(L^1+\cM)_+$. If $\<a,x\>_\tau\le\<b,x\>_\tau$ for all $x\in L^1\cap\cM_+$, then
$a\le b$. Hence, if $\<a,x\>_\tau=\<b,x\>_\tau$ for all $x\in L^1\cap\cM_+$, then $a=b$.
\end{lemma}

\begin{proof}
Take the Jordan decomposition $b-a=(b-a)_+-(b-a)_-$. Let $h:=(b-a)_-$ and $e:=s(h)$. For every
$x\in L^1\cap\cM_+$ we note that
$$
0\le\<h,x\>_\tau=\<e(a-b)e,x\>_\tau=\<a-b,exe\>_\tau\le0,
$$
where the inequality follows from the assumption since
$e(L^1\cap\cM_+)e\subseteq L^1\cap\cM_+$ obviously. Hence $\<h,x\>_\tau=0$ for all
$x\in L^1\cap\cM_+$. Now let us show $h=0$ (which means $a\le b$).

By the condition (iii) above we can choose an $r\ge0$ for which $f:=e_{(r,\infty)}(h)$ satisfies
$\tau(f)<+\infty$, $hf\in L^1_+$ and $hf^\perp\in\cM_+$. Since
$f\cM_+f\subset L^1\cap\cM_+$, one has $\tau(hfx)=\tau(hfxf)=0$ for all $x\in\cM_+$ so that
$hf=0$. One also has $\tau(hf^\perp x)=\tau(hf^\perp xf^\perp)=0$ for all $x\in L^1\cap\cM_+$.
Since $L^1\cap\cM_+$ is dense in $L^1_+$, it follows that $\tau(hf^\perp x)=0$ for all
$x\in L^1_+$ and hence $hf^\perp=0$. Therefore, $h=hf+hf^\perp=0$.
\end{proof}

The following two lemmas will be needed in proving subsequent theorems:

\begin{lemma}\label{L-3.36}
Let $a\in(L^1+\cM)_+$, $x_1,x_2\in L^2\cap\cM$ and $k_1,k_2\in\cM$. Then
$x_2^*x_1a^{1/2}k_1k_2^*a^{1/2}\in L^1$, $x_1a^{1/2}k_1,x_2a^{1/2}k_2\in L^2$ and
$$
\tau(x_2^*x_1a^{1/2}k_1k_2^*a^{1/2})=\tau(k_2^*a^{1/2}x_2^*x_1a^{1/2}k_1)
=(x_1a^{1/2}k_1,x_2a^{1/2}k_2)
$$
with the inner product $(\cdot,\cdot)$ in $L^2$.
\end{lemma}

\begin{proof}
Choose an $r\ge0$ such that $af\in L^1_+$ and $af^\perp\in\cM_+$ with $f:=e_{(r,\infty)}(a)$.
Then $a^{1/2}=a^{1/2}f+a^{1/2}f^\perp\in L^2+\cM$ so that $a^{1/2}k_1, k_2^*a^{1/2}\in L^2+\cM$.
Since $x_2^*x_1\in L^1\cap\cM\subseteq L^2$ obviously, one has $x_2^*x_1a^{1/2}k_1\in L^1\cap L^2$
(noting $L^2\cdot L^2\subseteq L^1$, $L^1\cdot\cM\subseteq L^1$, $\cM\cdot L^2\subseteq L^2$
and $L^2\cdot\cM\subseteq L^2$). Therefore, $x_2^*x_1a^{1/2}k_1k_2^*a^{1/2}\in L^1$ and
$$
\tau(x_2^*x_1a^{1/2}k_1k_2^*a^{1/2})=\tau(k_2^*a^{1/2}x_2^*x_1a^{1/2}k_1).
$$
Moreover, one has $x_1a^{1/2}k_1,x_2a^{1/2}k_2\in L^2$ so that the above quantity is written as
$(x_1a^{1/2}k_1,x_2a^{1/2}k_2)$.
\end{proof}

\begin{lemma}\label{L-3.37}
Let $a,a_n\in(L^1+\cM)_+$ for $n\in\bN$. If $a_n\searrow a$ in the strong resolvent sense, then
$\<a_n,x\>_\tau\searrow\<a,x\>_\tau$ for every $x\in L^1\cap\cM_+$.
\end{lemma}

\begin{proof}
By the assumption it is clear that $\<a_n,x\>_\tau$ decreases and $\<a_n,x\>_\tau\ge\<a,x\>_\tau$.
So it suffices to show that $\inf_n\<a_n,x\>_\tau=\<a,x\>_\tau$. For each $\delta>0$ let
$a_\delta:=(1+\delta a)^{-1}a$ ($=\delta^{-1} (1-(1+\delta a)^{-1}))$ and $a_{n,\delta}:=(1+\delta a_n)^{-1}a_n$.
Note that
\begin{align}
\<a_n,x\>_\tau&=\<a_n-a_{n,\delta},x\>_\tau+\<a_{n,\delta},x\>_\tau \nonumber\\
&=\<\delta(1+\delta a_n)^{-1}a_n^2,x\>_\tau+\<a_{n,\delta},x\>_\tau. \label{F-3.23}
\end{align}
Moreover, by the assumption note that for each $\delta>0$ we have $a_{n,\delta}\searrow a_\delta$
strongly in $\cM_+$ so that
\begin{equation}\label{F-3.24}
\<a_{n,\delta},x\>_\tau=\tau(a_{n,\delta}x)
\,\searrow\,\tau(a_\delta x)=\<a_\delta,x\>_\tau
\end{equation}
as $n\to\infty$. Since
$s\ge0\mapsto\delta(1+\delta s)^{-1}s^2$ is a continuous increasing function, by
\cite[Lemma 2.5,(iv), (iii)]{FK} we have
$$
\mu_t(\delta(1+\delta a_n)^{-1}a_n^2)={\delta\mu_t(a_n)^2\over1+\delta\mu_t(a_n)}
\le{\delta\mu_t(a_1)^2\over1+\delta\mu_t(a_1)},\qquad t>0.
$$
Therefore,
\begin{align}
&\<\delta(1+\delta a_n)^{-1}a_n^2,x\>_\tau
\le\tau\bigl(\big|\delta(1+\delta a_n)^{-1}a_n^2x\big|\bigr) \nonumber\\
&\quad=\int_0^\infty\mu_t\bigl(\delta(1+\delta a_n)^{-1}a_n^2x\bigr)\,dt
\quad\mbox{(by \cite[Proposition 2.7]{FK})} \nonumber\\
&\quad\le\int_0^\infty\mu_{t/2}\bigl(\delta(1+\delta a_n)^{-1}a_n^2\bigr)\mu_{t/2}(x)\,dt
\quad\mbox{(by \cite[Lemma 2.5,(vii)]{FK})} \nonumber\\
&\quad=2\int_0^\infty\mu_t\bigl(\delta(1+\delta a_n)^{-1}a_n^2\bigr)\mu_t(x)\,dt \nonumber\\
&\quad\le2\int_0^\infty{\delta\mu_t(a_1)^2\over1+\delta\mu_t(a_1)}\,\mu_t(x)\,dt \nonumber\\
&\quad\le2\|x\|_\infty\int_0^r\mu_t(a_1)\,dt+2\delta\mu_r(a_1)^2\int_r^\infty\mu_t(x)\,dt
\nonumber\\
&\quad\le2\|x\|_\infty\int_0^r\mu_t(a_1)\,dt+2\delta\mu_r(a_1)^2\|x\|_1 \label{F-3.25}
\end{align}
for any $r>0$. Since $\lim_{r\searrow0}\int_0^r\mu_t(a_1)\,dt=0$
(due to the condition (ii) at the beginning of this subsection), for each $\eps>0$ there is
an $r>0$ such that $2\|x\|_\infty\int_0^r\mu_t(a_1)\,dt\le\eps/2$. Then there is a $\delta>0$
such that $2\delta\mu_r(a_1)^2\|x\|_1\le\eps/2$. For such a $\delta>0$ it follows from
\eqref{F-3.23} and \eqref{F-3.25} that
$$
\<a_n,x\>_\tau\le\eps+\<a_{n,\delta},x\>_\tau
$$
so that
$$
\inf_n\<a_n,x\>_\tau\le\eps+\<a_\delta,x\>_\tau\le\eps+\<a,x\>_\tau
$$
thanks to \eqref{F-3.24}, which implies $\inf_n\<a_n,x\>_\tau=\<a,x\>_\tau$ as desired.
\end{proof}

Note that if $a,b\in(L^1+\cM)_+$ then $a\sigma b\in(L^1+\cM)_+$ for any connection $\sigma$
(see Lemma \ref{L-3.13}). The next theorem is an extension of the variational formula
\eqref{F-3.11} (for bounded positive operators) to $a:b$ for $a,b\in(L^1+\cM)_+$.

\begin{theorem}\label{T-3.38}
For every $a,b\in(L^1+\cM)_+$ and $x\in L^2\cap\cM$ we have
\begin{align}\label{F-3.26}
\<a:b,x^*x\>_\tau=\inf\{\<a,y^*y\>_\tau+\<b,z^*z\>_\tau;\,y,z\in L^2\cap\cM,\,y+z=x\}.
\end{align}
\end{theorem}

\begin{proof}
Let $h:=a+b$. By Theorem \ref{T-3.25},(2) and Lemma \ref{L-3.37} note that, with
$x,y,z\in L^2\cap\cM$,
$$
\<a:b,x^*x\>_\tau=\inf_{\eps>0}\<(a+\eps h):(b+\eps h),x^*x\>_\tau
$$
and
\begin{align*}
\inf_{y+z=x}\{\<a,y^*y\>_\tau+\<b,z^*z\>_\tau\}
&=\inf_{y+z=x}\inf_{\eps>0}\{\<a+\eps h,y^*y\>_\tau+\<b+\eps h,z^*z\>_\tau\} \\
&=\inf_{\eps>0}\inf_{y+z=x}\{\<a+\eps h,y^*y\>_\tau+\<b+\eps h,z^*z\>_\tau\}.
\end{align*}
So it suffices to prove \eqref{F-3.26} when $\lambda^{-1}b\le a\le\lambda b$ for some
$\lambda>0$. In this case, let $e:=s(a)=s(b)\in\cM$ and take a $k\in e\cM e$ with
$a^{1/2}=kh^{1/2}$ and $T_{a/h}=k^*k$. As immediately seen, note that $k$ is invertible in
$e\cM e$. Since $T_{a/h}+T_{b/h}=e$ (see Definition \ref{D-3.5}), we have
\begin{align*}
a:b&=h^{1/2}(T_{a/h}:T_{b/h})h^{1/2}
=h^{1/2}T_{a/h}(T_{a/h}+T_{b/h})^{-1}T_{b/h}h^{1/2} \\
&=h^{1/2}k^*k(e-k^*k)h^{1/2}
=h^{1/2}k^*(e-kk^*)kh^{1/2} \\
&=a^{1/2}(e-kk^*)a^{1/2}.
\end{align*}
Now let $x,y,z\in L^2\cap\cM$ with $y+z=x$, and using Lemma \ref{L-3.36} we compute
\begin{align}
&\<a,y^*y\>_\tau+\<b,z^*z\>_\tau-\<a:b,x^*x\>_\tau \nonumber\\
&\quad=\tau((x-z)^*(x-z)a)+\tau(z^*zb)-\tau(x^*x(a-a^{1/2}kk^*a^{1/2})) \nonumber\\
&\quad=\tau(z^*za)+\tau(z^*zb)-\tau(x^*za)-\tau(z^*xa)+\tau(x^*xa^{1/2}kk^*a^{1/2}) \nonumber\\
&\quad=\tau(z^*zh)+\tau(k^*a^{1/2}x^*xa^{1/2}k)-2\Re\tau(x^*za^{1/2}k^{*-1}k^*a^{1/2}) \nonumber\\
&\quad=\|zh^{1/2}\|_2^2+\|xa^{1/2}k\|_2^2-2\Re(za^{1/2}k^{*-1},xa^{1/2}k) \nonumber\\
&\quad=\|zh^{1/2}\|_2^2+\|xa^{1/2}k\|_2^2-2\Re(zh^{1/2},xa^{1/2}k) \label{F-3.27}
\end{align}
thanks to $h^{1/2}=k^{-1}a^{1/2}=a^{1/2}k^{*-1}$ with inverses $k^{-1},k^{*-1}$ in $e\cM e$.
The Schwarz inequality gives
$$
\<a,y^*y\>_\tau+\<b,z^*z\>_\tau-\<a:b,x^*x\>_\tau\ge0.
$$

We note that $\{zh^{1/2}:z\in L^2\cap\cM\} \subseteq  L^2e$. Let us show that the left hand side
here is dense in $L^2e$ in the norm $\|\cdot\|_2$. For this it suffices to show that if
$w\in L^2$ annihilates the above left hand side (i.e., $\tau(wzh^{1/2})=0$ for all
$z\in L^2\cap\cM$), then $w$ annihilates $L^2e$, equivalently, $ew=0$. So assume that
$w\in L^2$ satisfies $\tau(wzh^{1/2})=0$ for all $z\in L^2\cap\cM$. Take the spectral
decomposition $h=\int_0^\infty\lambda\,de_\lambda$. For each $n\in\bN$ let $f_n:=e_n-e_{1/n}$.
Then, for every $z\in L^2\cap\cM$, from $zf_n\in L^2\cap\cM$ we have $\tau(wzf_nh^{1/2})=0$.
Note that $f_nh^{1/2}$ is bounded and invertible in $f_n\cM f_n$. Hence $f_nh^{1/2}w\in L^2$
and $\tau(f_nh^{1/2}wz)=0$ for all $z\in L^2\cap\cM$. Since $L^2\cap\cM$ is dense in $L^2$, it
follows that $f_nh^{1/2}w=0$ and hence $f_nw=0$, so $\tau(f_nww^*)=0$. Since $f_n\nearrow e$,
we have $\tau(eww^*)=0$, i.e., $ew=0$.

Since $xa^{1/2}k\in L^2e$ by Lemma \ref{L-3.36}, from what we have just shown one can
choose a sequence $z_n\in L^2\cap\cM$ such that $\|z_nh^{1/2}-xa^{1/2}k\|_2\to0$. Letting
$y_n:=x-z_n$ and putting $y=y_n$, $z=z_n$ in \eqref{F-3.27} one has
\begin{align*}
&\<a,y_n^*y_n\>_\tau+\<b,z_n^*z_n\>_\tau-\<a:b,x^*x\>_\tau \\
&\quad=\|z_nh^{1/2}\|_2^2+\|xa^{1/2}k\|_2^2-2\Re(z_nh^{1/2},xa^{1/2}k)\,\longrightarrow0.
\end{align*}
Thus, the expression \eqref{F-3.26} has been obtained.
\end{proof}

Note by Lemma \ref{L-3.35} that $a:b$ for $a,b\in(L^1+\cM)_+$ can be determined by the
variational expression \eqref{F-3.26}.

%%%%%%%%%%%%%%%%%%%%%%%%%%%%%%%%%%%%%%%%%%%%%%%%%%%
%%%%%%%%%%%%%%%%%%% 3.5 %%%%%%%%%%%%%%%%%%%%%%%%%%%%%%
%%%%%%%%%%%%%%%%%%%%%%%%%%%%%%%%%%%%%%%%%%%%%%%%%%%
\subsection{Comparison between $a\#_\alpha b$ and $|a^{1-\alpha}b^\alpha|$}\label{S-3.5}

Consider the following conditions for $a\in\overline\cM$:
\begin{itemize}
\item[(i)] $\int_0^\delta\log^+\mu_s(a)\,ds<\infty$ for some (equivalently, any)
$\delta>0$, where $\log^+t$ stands for $\max\{\log t,0\}$ for $t\ge0$. 
\item[(ii)] $\int_0^\delta\mu_s(a)^q\,ds<\infty$ for some (equivalently, any) $\delta>0$
and some $q>0$.
\end{itemize}

The implication (ii)$\implies$(i) is immediately seen. If
$\mu_s(a)\le Cs^{-\beta}$ ($0<s<\delta$) for some (equivalently, any) $\delta>0$
and some $C,\beta>0$ (in particular, if $a\in L^p(\cM,\tau)$ for some
$p\in(0,\infty]$, see \eqref{F-3.2}), then (ii) is satisfied. Whenever
$a\in\overline\cM$ satisfies (i), we define $\Lambda_t(a)\in[0,\infty)$ for $t>0$ by
$$
\Lambda_t(a):=\exp\int_0^t\log\mu_s(a)\,ds.
$$
This is a generalization of the \emph{Fuglede-Kadison determinant} $\Delta(a)$
\cite{FuKa} for $a\in\cM$. 
In fact, we have $\Delta_{\tau(1)}(a)=\Delta(a)$ when $\cM$ is finite and $\tau(1) < \infty$.
The majorization properties in terms of $\Lambda_t(a)$
($t>0$) are useful \cite{Fa,FK} to derive trace and norm inequalities for $\tau$-measurable
operators. We remark that 
the Fuglede-Kadison determinant in such a generalized setting 
and certain related estimates (in logarithmic submajorization) 
are subjects in the recent article \cite{DDSZ3}.

\begin{lemma}\label{L-3.39}
If $a,b\in\overline\cM$ satisfy $(2^\circ)$, then $a+b$, $ab$ and $a^\alpha$ $(\alpha>0)$
satisfy the same.
\end{lemma}

\begin{proof}
Assume that $\int_0^\delta\mu_s(a)^p\,ds<\infty$ and $\int_0^\delta\mu_s(b)^q\,ds<\infty$ for
some $\delta,p,q>0$. By \cite[Lemma 2.5]{FK} note that $\mu_s(a+b)\le\mu_{s/2}(a)+\mu_{s/2}(b)$
and $\mu_s(ab)\le\mu_{s/2}(a)\mu_{s/2}(b)$ for all $s>0$. With $1/r=1/p+1/q$ H\"older's
inequality gives
\begin{align*}
\int_0^\delta\mu_s(ab)^r\,ds
&\le\biggl(\int_0^\delta\mu_{s/2}(a)^p\,ds\biggr)^{r/p}
\biggl(\int_0^\delta\mu_{s/2}(b)^q\,ds\biggr)^{r/q} \\
&=\biggl(2\int_0^{\delta/2}\mu_s(a)^p\,ds\biggr)^{r/p}
\biggl(2\int_0^{\delta/2}\mu_s(b)^q\,ds\biggr)^{r/q}<\infty.
\end{align*}
With $r:=\min\{p,q\}$ one has
\begin{align*}
\int_0^\delta\mu_s(a+b)^r\,ds
&\le2^r\int_0^\delta(\mu_{s/2}(a)^r+\mu_{s/2}(b)^r)\,ds \\
&\le2^r\biggl(\int_0^\delta(1+\mu_{s/2}(a)^p)\,ds
+\int_0^\delta(1+\mu_{s/2}(b)^q)\,ds\biggr)<\infty.
\end{align*}
The assertion for $a^\alpha$ is clear since $\mu_s(a^\alpha)=\mu_s(a)^\alpha$.
\end{proof}

The next theorem is the main result of \cite{Ko4} while it was shown in the case
$a,b\in\overline\fS_+$, but the whole proof is valid in the present setting. The proof in
\cite{Ko4} starts with the inequality
$$
\Lambda_t(ab)\le\Lambda_t(a)\Lambda_t(b),\qquad t>0,
$$
for $a,b\in\fS$ ($:=\cM\cap\overline\fS$) in \cite{Fa}, which indeed holds true for
$a,b\in\overline\cM$ satisfying (ii) as shown in \cite[pp.\ 287--288]{FK}. In the
proofs in \cite{FK,Ko4} the equality
\begin{align*}
&\lim_{p\searrow0}\bigg(\int_0^t\phi(s)^p\,{ds\over t}\biggr)^{1/p}
=\exp\int_0^t\log\phi(s)\,ds \\
&\qquad\mbox{if}\ \ \int_0^t\phi(s)^q\,ds<\infty\ \ \mbox{for some $q>0$}
\end{align*}
is applied to $\phi(s)=\mu_s(ab)$, $\mu_s(|ab|^r)$, etc., where the condition (ii) is
essential.

\begin{theorem}[\cite{Ko4}]\label{T-3.40}
If $a,b\in\overline\cM_+$ satisfy $(ii)$, then
\begin{align}\label{F-3.28}
\Lambda_t(|ab|^r)\le\Lambda_t(a^rb^r),\qquad t>0,\ r\ge1,
\end{align}
or equivalently,
\begin{align}\label{F-3.29}
\Lambda_t(a^rb^r)\le\Lambda_t(|ab|^r),\qquad t>0,\ 0<r\le1.
\end{align}
\end{theorem}

The aim of this subsection is to compare $a\#_\alpha b$ (see Example \ref{E-2.9}) and
$|a^{1-\alpha}b^\alpha|$ for $0\le\alpha\le1$ as follows:

\begin{theorem}\label{T-3.41}
If $a,b\in\overline\cM_+$ satisfy $(2^\circ)$, then
\begin{align}\label{F-3.30}
\Lambda_t(a\#_\alpha b)\le\Lambda_t(a^{1-\alpha}b^\alpha),\qquad t>0,\ 0\le\alpha\le1.
\end{align}
\end{theorem}

\begin{proof}
The case $\alpha=0$ or $1$ is trivial, so assume $0<\alpha<1$. Since
$\Lambda_t(a\#_\alpha b)=\Lambda_t(b\#_{1-\alpha}a)$ and
$\Lambda_t(a^{1-\alpha}b^\alpha)=\Lambda_t(b^\alpha a^{1-\alpha})$, we may assume
$1/2\le\alpha<1$. For each $n\in\bN$ let $a_n:=(1/n)+a$. By Lemma \ref{L-3.39} note that $a_n$
and $a_n^{-1/2}ba_n^{-1/2}$ satisfy (ii). We have
\begin{align*}
\Lambda_t(a\#_\alpha b)&\le\Lambda_t(a_n\#_\alpha b) \\
&=\Lambda_t(a_n^{1/2}(a_n^{-1/2}ba_n^{1/2})^\alpha a_n^{1/2})
\quad\mbox{(by \eqref{F-3.22})} \\
&=\Lambda_t((a_n^{-1/2}ba_n^{-1/2})^{\alpha/2}a_n^{1/2})^2 \\
&\le\Lambda_t(|(a_n^{-1/2}ba_n^{-1/2})^{1/2}a_n^{1/2\alpha}|^\alpha)^2
\quad\mbox{(by \eqref{F-3.29})} \\
&=\Lambda_t(a_n^{1/2\alpha}(a_n^{-1/2}ba_n^{-1/2})a_n^{1/2\alpha})^\alpha \\
&=\Lambda_t(|b^{1/2}a_n^{1-\alpha\over2\alpha}|^{2\alpha}) \\
&\le\Lambda_t(b^\alpha a_n^{1-\alpha})\quad\mbox{(by \eqref{F-3.28} since $2\alpha\ge1$)} \\
&=\Lambda_t(a_n^{1-\alpha}b^\alpha).
\end{align*}
Now, to prove \eqref{F-3.30}, it suffices to show that
\begin{align}\label{F-3.31}
\Lambda_t(a^{1-\alpha}b^\alpha)=\lim_{n\to\infty}\Lambda_t(a_n^{1-\alpha}b^\alpha).
\end{align}
Recalling continuity of an operator function in the measure topology \cite{Ti} (as long as a
function in question is continuous), we have
$$
|a_n^{1-\alpha}b^\alpha|^2=b^\alpha a_n^{2(1-\alpha)}b^\alpha
\,\searrow\,b^\alpha a^{2(1-\alpha)}b^\alpha=|a^{1-\alpha}b|^2
$$
in the measure topology as $n\to\infty$. Hence by \cite[Lemma 3.4,(ii)]{FK},
$\mu_s(a_n^{1-\alpha}b^\alpha)\searrow\mu_s(a^{1-\alpha}b^\alpha)$ as $n\to\infty$ for
a.e.\ $s\in(0,t)$. Since $\int_0^t\log^+\mu_s(a_1^{1-\alpha}b^\alpha)\,ds<\infty$ from (ii),
one can apply the monotone convergence theorem to
$\log^+\mu_s(a_1^{1-\alpha}b^\alpha)-\log\mu_s(a_n^{1-\alpha}b^\alpha)\ge0$, and then
\eqref{F-3.31} follows.
\end{proof}

\begin{remark}\label{R-3.42}\rm
Let $a,b$ be as stated in Theorems \ref{T-3.40} and \ref{T-3.41}. Theorem \ref{T-3.40} says
that $r>0\mapsto\Lambda_t(|a^rb^r|^{1/r})$ is monotone increasing for each $t>0$. Hence Theorem
\ref{T-3.41} furthermore implies that for $0<r\le q$,
\begin{align}\label{F-3.32}
\Lambda_t((a^r\#_\alpha b^r)^{1/r})\le\Lambda_t(|a^{q(1-\alpha)}b^{q\alpha}|^{1/q}),
\qquad t>0,\ 0\le\alpha\le1.
\end{align}
\end{remark}

\begin{problem}\label{Q-3.43}\rm
For any $a,b$ as above we conjecture that $r>0\mapsto\Lambda_t((a^r\#_\alpha b^r)^{1/r})$ is
monotone decreasing for each $t>0$. This is well-known for the $B(\cH)$ case (in particular,
for matrices) as log-majorization, whose proof is based on the anti-symmetric tensor power
technique, see, e.g., \cite{Hi1,Hi2}. The technique is not at our disposal in the von Neumann
algebra setting. 
If the conjecture is true, then \eqref{F-3.32}  holds for all $r,q>0$ without restriction.
\end{problem}

%%%%%%%%%%%%%%%%%%%%%%%%%%%%%%%%%%%%%%%%%%%%%%%%%%
%%%%%%%% Connections on Haagerup's $L^p$-space %%%%%%%%%%%%%%%%%%%%%
%%%%%%%%%%%%%%%%%%%%%%%%%%%%%%%%%%%%%%%%%%%%%%%%%%
\section{Connections on Haagerup's $L^p$-spaces}\label{S-4}

In this section let $\cM$ be a general von Neumann algebra on a Hilbert space $\cH$, and we
will discuss connections on Haagerup's $L^p$-spaces. A brief description of Haagerup's
$L^p$-spaces $L^p(\cM)$ for $0<p\le\infty$ is given in Appendix \ref{S-A} for the reader's
convenience. The basis of Haagerup's $L^p(\cM)$ is the crossed product von Neumann algebra
$\cR:=\cM\rtimes_\sigma\bR$ with respect to the modular automorphism group $\sigma_t$ for
a faithful normal semi-finite weight $\ffi_0$ on $\cM$. Note that $\cR$ is semi-finite with the canonical
trace $\tau$, and consider the space $\overline\cR$ of $\tau$-measurable operators affiliated
with $\cR$. Then $L^p(\cM)$'s are constructed inside $\overline\cR$ by \eqref{F-A.4} in terms
of the dual action $\theta_s$ (see Appendix \ref{S-A}).

In particular, when $\cM$ is semi-finite with a faithful semi-finite normal trace $\tau_0$, we
note that all the results in this section hold true in the setting of the conventional
$L^p$-spaces $L^p(\cM,\tau_0)$ (\cite{Di,Y}) with respect to $\tau_0$.
Note that Haagerup's $L^p(\cM)$ in this situation is identified (up to an isometric isomorphism)
with $L^p(\cM,\tau_0)$, and all the results below reduce to those for $L^p(\cM,\tau_0)$.
In fact, some of them are included in \S\ref{S-3}, otherwise proofs can easily be modified
to the $L^p(\cM,\tau_0)$ setting.

Now, let $\sigma$ be any connection. As defined in the previous section, for every
$a,b\in\overline\cR_+$ we have $a\sigma b\in\overline\cR_+$. The next lemma says that we have
the connection $\sigma:L^p(\cM)_+\times L^p(\cM)_+\to L^p(\cM)_+$ by restricting $\sigma$ to
$L^p(\cM)_+$ for any $p\in(0,\infty]$.

%%%%%%%%%%%%%%%%%%%%%% \label{L-4.1} %%%%%%%%%%%%%%%%%%%
\begin{lemma}\label{L-4.1}
Let $0<p\le\infty$. If $a,b\in L^p(\cM)_+$, then $a\sigma b\in L^p(\cM)_+$.
\end{lemma}

\begin{proof}
Let $h:=a+b\in L^p(\cM)_+$. Let $x,y\in s(h)\cR s(h)$ such that $a^{1/2}=xh^{1/2}$ and
$b^{1/2}=yh^{/12}$ as in Definition \ref{D-3.5}. 
Here note that the support projection $s(h)$ is
in $\cM$. Applying $\theta_s$ to $a^{1/2}=xh^{1/2}$ gives
$\theta_s(a)^{1/2}=\theta_s(x)\theta_s(h)^{1/2}$. Since $\theta_s(a)=e^{-s/p}a$ and
$\theta_s(h)=e^{-s/p}h$ for all $s\in\bR$, one has $a^{1/2}=\theta_s(x)h^{1/2}$. Hence from
the uniqueness of $x$ it follows that $\theta_s(x)=x$ for all $s\in\bR$ so that
$T_{a/h}=x^*x\in \cM_+$, and similarly $T_{b/h}=y^*y\in \cM_+$. Therefore, one has
$a\sigma b=h^{1/2}(T_{a/h}\sigma T_{b/h})h^{1/2}\in L^p(\cM)_+$ immediately.
\end{proof}

The connection $a\sigma b$ for $a,b\in L^p(\cM)_+$ does not depend upon the choice of a
faithful semi-finite normal weight $\ffi_0$ on $\cM$. In fact, for another faithful semi-finite
normal weight $\ffi_1$ on $\cM$ we have the (canonical) isomorphism
$\kappa:\cR\to\cR_1:=\cM\rtimes_{\sigma^{\ffi_1}}\bR$. The $\kappa$ induces an isometric
isomorphism from $L^p(\cM)$ in $\overline\cR$ and that in $\overline\cR_1$ (see the remark at
the end of Appendix \ref{S-A}), for which we have $\kappa(a\sigma b)=\kappa(a)\sigma\kappa(b)$
for any $a,b\in L^p(\cM)_+$.

By restricting the properties of Theorem \ref{T-3.25} (for $\overline\cR_+$) to $L^p(\cM)_+$ for
any $p\in(0,\infty]$, it is clear that all of them hold for $a,b,a_i,b_i\in L^p(\cM)_+$. But the
decreasing convergence holds more strongly in the norm $\|\cdot\|_p$ as follows.
For this it is worthwhile to first give the variant of Theorem \ref{T-3.8} for a
decreasing sequence in $L^p(\cM)_+$.

%%%%%%%%%%%%%%%%%%%%%% \label{P-4.2} %%%%%%%%%%%%%%%%%%
\begin{proposition}\label{P-4.2}
Let $0<p<\infty$. If $a_n\in L^p(\cM)_+$ $(n\in\bN)$ and $a_n\searrow a$ in the strong resolvent
sense, then $a\in L^p(\cM)_+$ and $\|a_n-a\|_p\to0$.
\end{proposition}

\begin{proof}
Since $L^p(\cM)\subseteq\overline\fS$ (defined inside $\overline\cR$) by \cite[Lemma B]{Ko8}, 
Theorem \ref{T-3.8} implies that $a_n\searrow a$ in the measure topology. By
\cite[Chap.~II, Proposition 26]{Te1} or \cite[Lemma B]{Ko8} note that the $\|\cdot\|_p$-norm
topology on $L^p(\cM)$ coincides with the relative topology induced from  the measure topology
on $\overline\cR$. (More precisely we have $\mu_t(x)=t^{-1/p}\|x\|_p$ for $x \in L^p(\cM)$.)
Hence we have the conclusion.
\end{proof}

Now, the next theorem follows from Propositions \ref{P-3.27} and \ref{P-4.2}.

%%%%%%%%%%%%%%%%%%%%% \label{T-4.3} %%%%%%%%%%%%%%%%%%%
\begin{theorem}\label{T-4.3}
Let $0<p<\infty$. If $a_n,b_n\in L^p(\cM)_+$ $(n\in\bN)$, $a_n\searrow a$ and $b_n\searrow b$
in the strong resolvent sense, then $a,b\in L^p(\cM)_+$ and $\|a_n\sigma b_n-a\sigma b\|_p\to0$.
\end{theorem}

For $a,b\in L^p(\cM)_+$ and $c\in L^q(\cM)$ with $0<p,q\le\infty$ we have the transformer
inequality $c^*(a\sigma b)c\le(c^*ac)\sigma(c^*bc)$ in $L^r(\cM)$ ($1/p+2/q=1/r$) by
Theorem \ref{T-3.28} (and H\"older's inequality for Haagerup's $L^p$-spaces). Furthermore,
Theorem \ref{T-3.31} shows that $c^*(a\sigma b)c=(c^*ac)\sigma(c^*bc)$ in $L^r(\cM)$ whenever
$s(a+b)\le s(cc^*)$. In particular, when $p=1$ and $r=\infty$, this says that
\begin{align}\label{F-4.1}
c^*(\phi\sigma\psi)c=(c^*\phi c)\sigma(c^*\psi c)
\end{align}
for every $\phi,\psi\in\cM_*^+$ and $c\in\cM$ with $s(\phi+\psi)\le s(cc^*)$.

The following proposition is the variational expression of the parallel sum of
$a,b\in L^p(\cM)$, $1\le p<\infty$, in terms of the $L^p$-$L^q$-duality. When $p=1$, this
reduces to the expression of $\phi:\psi$ for $\phi,\psi\in\cM_*^+$ as
\begin{equation}\label{F-4.2} %%%%%%%%%%%%%%%%%%%%%%%%%%%% \label{F-4.2}
(\phi:\psi)(x^*x)=\inf\{\phi(y^*y)+\psi(z^*z);\,
y,z\in\cM,\,y+z=x\},\qquad x\in\cM.
\end{equation}

%%%%%%%%%%%%%%%%%%% \label{P-4.4} %%%%%%%%%%%%%%%%%%
\begin{proposition}\label{P-4.4}
Let $1\le p\le\infty$ and $1/p+1/q=1$. For every $a,b\in L^p(\cM)_+$ we have
$$
\<a:b,x^*x\>_{p,q}=\inf\{\<a,y^*y\>_{p,q}+\<b,z^*z\>_{p,q};\,
y,z\in L^{2q}(\cM),\,y+z=x\}
$$
for every $x\in L^{2q}(\cM)$, where $\<a,c\>_{p,q}=\tr(ac)$ for $a\in L^p(\cM)$ and
$c\in L^q(\cM)$ $($see \eqref{F-A.5}$)$.
\end{proposition}

\begin{proof}
The case $p=\infty$ (hence $q=1$) is the well-known variational expression (see \eqref{F-3.11})
since $\cM=L^\infty(\cM)$ standardly acts on $L^2(\cM)$ (by left mulitplication) and
$\<a,x^*x\>_{\infty,1}=\tr(ax^*x)=(ax^*,x^*)$ for $x\in L^2(\cM)$.

Assume $1\le p<\infty$. Let $h:=a+b$ and $e:=s(h)\in\cM$. The proof below is similar to that of
Theorem \ref{T-3.38} (with $\<\cdot,\cdot\>_{p,q}$ in place of $\<\cdot,\cdot\>_\tau$),
so we only sketch it. As before (using Theorem \ref{T-4.3} in the present case) we may assume that
$\lambda^{-1}b\le a\le\lambda b$ for some $\lambda>0$, and take a $k\in e\cM e$, invertible in $e\cM e$,
such that $a^{1/2}=kh^{1/2}$ and $T_{a/h}=k^*k$ (see the proof of Lemma \ref{L-4.1}).  As in the
previous proof we have $a:b=a^{1/2}(e-kk^*)a^{1/2}$ and for $y,z\in L^{2q}(\cM)$ with $y+z=x$,
\begin{align*}
&\<a,y^*y\>_{p,q}+\<b,z^*z\>_{p,q}-\<a:b,x^*x\>_{p,q} \\
&\qquad=\|zh^{1/2}\|_2^2+\|xa^{1/2}k\|_2^2-2\Re(zh^{1/2},xa^{1/2}k).
\end{align*}
The remaining proof is the same as in the proof of Theorem \ref{T-3.38} since
$L^{2q}(\cM)h^{1/2}$ is dense in $L^2(\cM)e$ and $xa^{1/2}k\in L^2(\cM)e$.
\end{proof}

The integral in the definition \eqref{F-2.9} of $\phi\sigma\psi$ is given in the weak sense (i.e.,
in the evaluation at each $\xi\in\cH$). In the following we give the integral expression of
$a\sigma b$ for $a,b\in L^p(\cM)$, $1\le p<\infty$, in the strong sense of Bochner integral
(see, e.g., \cite{DU}):

%%%%%%%%%%%%%%%%%%%% \label{P-4.5} %%%%%%%%%%%%%%%%%%%%%
\begin{proposition}\label{P-4.5}
Let $1\le p<\infty$ and $a,b\in L^p(\cM)$. Then the $L^p(\cM)$-valued function
$t\mapsto{1+t\over t}((ta):b)$ is Bochner integrable on $(0,\infty)$ with respect to
the measure $\mu$ and
$$
a\sigma b=\alpha a+\beta b+\int_{(0,\infty)}{1+t\over t}((ta):b)\,d\mu(t).
$$
\end{proposition}

\begin{proof}
Let $h:=a+b$. From  the expression \eqref{F-2.8} applied to $A=T_{a/h}$ and $B=T_{b/h}$
we have the integral expression of $a\sigma b$ in the weak sense in terms of the $L^p$-$L^q$-duality.
That is, for every $c\in L^q(\cM)$, $1/p+1/q=1$, since $h^{1/2}ch^{1/2}\in L^1(\cM)$ ($\cong\cM_*$),
we have
\begin{align*}
\tr\,c(a:b)&=\tr\,(h^{1/2}ch^{1/2})(T_{a/h}\sigma T_{b/h}) \\
&=\alpha\tr\,(h^{1/2}ch^{1/2})T_{a/h}+\beta\tr\,(h^{1/2}ch^{1/2})T_{b/h} \\
&\qquad+\int_{(0,\infty)}{1+t\over t}\,\tr\,(h^{1/2}ch^{1/2})((tT_{a/h}):T_{b/h})\,d\mu(t) \\
&=\alpha\tr\,ca+\beta\tr\,cb+\int_{(0,\infty)}{1+t\over t}\,\tr\,c((ta):b)\,d\mu(t).
\end{align*}
Let us show that the integral $\int_{(0,\infty)}{1+t\over t}((ta):b)\,d\mu(t)$ exists as
a Bochner integral. When $\lambda^{-1}\le a\le\lambda b$ for some $\lambda>0$, since
$(ta):b=b^{1/2}((tT_{a/b}):e)b^{1/2}$ (where $e:=s(b)$) and $t\mapsto(tT_{a/b}):e$ is continuous
in the operator norm, it follows that $t\mapsto(ta):b$ is continuous in the norm $\|\cdot\|_p$.
For general $a,b\in L^p(\cM)_+$, since $(ta+\eps h):(b+\eps h)\to(ta):b$ in the norm
$\|\cdot\|_p$ as $\eps\searrow0$ by Theorem \ref{T-4.3}, the function $t\mapsto(ta):b$ is
strongly measurable on $(0,\infty)$ in the norm $\|\cdot\|_p$. Furthermore, since
$$
\|(ta):b\|_p\le\|(t(a+b)):(a+b)\|_p={t\over1+t}\,\|a+b\|_p,
$$
we have $\int_{(0,\infty)}{1+t\over t}\,\|(ta):b\|_p\,d\mu(t)<\infty$. Hence the result follows.
\end{proof}

Consider the tensor products $\cM\otimes M_2$ with $\ffi_0\otimes\Tr$, where $\ffi_0$ is a
faithful semi-finite normal weight on $\cM$ and $\Tr$ is the usual trace on $M_2=M_2(\bC)$.
Since $\sigma_t^{\ffi_0\otimes\Tr}=\sigma_t^{\ffi_0}\otimes\mathrm{id}_2$ with the identity map
$\mathrm{id}_2$ on $M_2$, we have
$$
(\cM\otimes M_2)\rtimes_{\sigma^{\ffi_0}\otimes\mathrm{id}_2}\bR=\cR\otimes M_2,
$$
and the dual action on $\cR\otimes M_2$ is $\theta_s\otimes\mathrm{id}_2$, where $\theta_s$ is
the dual action on $\cR$. Hence $L^p(\cM\otimes M_2)$ is identified with $L^p(\cM)\otimes M_2$
and its positive part is $(L^p(\cM)\otimes M_2)\cap(\overline\cR\otimes M_2)_+$. In this
setting we can write the expressions of Propositions \ref{P-3.33} and \ref{P-3.34} restricted
to $a,b\in L^p(\cM)_+$, $0<p\le\infty$, as follows:
\begin{align*}
a\#b&=\max\biggl\{c\in L^p(\cM)_+;\,
\begin{bmatrix}a&c\\c&b\end{bmatrix}\ge0
\ \mbox{in $L^p(\cM)\otimes M_2$}\biggr\}, \\
a:b&=\max\biggl\{c\in L^p(\cM)_+;\,
\begin{bmatrix}a&0\\0&b\end{bmatrix}\ge\begin{bmatrix}c&c\\c&c\end{bmatrix}
\ \mbox{in $L^p(\cM)\otimes M_2$}\biggr\}.
\end{align*}

The next proposition gives a Haagerup's $L^p$-norm inequality for the weighted
geometric mean, which is a consequence of Theorem \ref{T-3.41}. Similarly to
\cite[Theorem 4]{Ko4} a key idea of the proof is to use the following formulas
(see \cite{Ko8,FK}) for every $a\in L^p(\cM)$ where $0<p\le\infty$:
\begin{align}\label{F-4.3} %%%%%%%%%%%%%%%%%%%%%%%%%%%%% \label{F-4.3}
\mu_t(a)=t^{-1/p}\|a\|_p,\qquad\Lambda_t(a)=((et^{-1})^{1/p}\|a\|_p)^t,\qquad t>0,
\end{align}
as already mentioned in the proof of Proposition \ref{P-4.2}. Note that they
hold for the case $p=\infty$ as well. Indeed, if $a\in L^\infty(\cM)=\cM$, then any
spectral projection of $|a|$ is $\theta$-invariant, so for any $s<\|a\|_\infty$ we have
$\tau(e_{(s,\|a\|_\infty]}(|a|)=\infty$. Hence $\mu_t(a)=\|a\|_\infty$ for all $t>0$, showing
the formulas in \eqref{F-4.3} for $p=\infty$.

Another fact we need in the proof below is that for any connection $\sigma$,
\begin{align}\label{F-4.4} %%%%%%%%%%%%%%%%%%%%%%%%%%%% \label{F-4.4}
\theta_s(x\sigma y)=\theta_s(x)\sigma\theta_s(y),\qquad
s\in\bR,\ x,y\in\overline\cR_+,
\end{align}
which is easily verified as follows: With $h:=x+y$ we have
$$
\theta_s(T_{x/h}\sigma T_{y/h})=\theta_s(T_{x/h})\sigma\theta_s(T_{y/h})
=T_{\theta_s(x)/\theta_s(h)}\sigma T_{\theta_s(y)/\theta_s(h)}
$$
so that
\begin{align*}
\theta_s(x\sigma y)&=\theta_s(h^{1/2}(T_{x/h}\sigma T_{y/h})h^{1/2}) \\
&=\theta_s(h)^{1/2}(T_{\theta_s(x)/\theta_s(h)}\sigma T_{\theta_s(y)/\theta_s(h)})
\theta_s(h)^{1/2}=\theta_s(x)\sigma\theta_s(y).
\end{align*}
(Note that Lemma \ref{L-4.1} is also immediate from the fact \eqref{F-4.4}.)

%%%%%%%%%%%%%%%%%%%%% \label{T-4.6} %%%%%%%%%%%%%%%%%%%
\begin{theorem}\label{T-4.6}
Let $p_1,p_2,p\in(0,\infty]$ and $0\le\alpha\le1$ be such that $1/p=(1-\alpha)/p_1+\alpha/p_2$.
Assume $a\in L^{p_1}(\cM)_+$ and $b\in L^{p_2}(\cM)_+$. Then we have $a\#_\alpha b$,
$|a^{1-\alpha}b^\alpha|\in L^p(\cM)_+$ and
\begin{align}\label{F-4.5} %%%%%%%%%%%%%%%%%%%%%%%%%%%%%% \label{F-4.5}
\|a\#_\alpha b\|_p\le\|a^{1-\alpha}b^\alpha\|_p\le\|a\|_{p_1}^{1-\alpha}\|b\|_{p_2}^\alpha.
\end{align}
\end{theorem}

\begin{proof}
Since the case $\alpha=0$ or $1$ is trivial, we may assume $0<\alpha<1$. From \eqref{F-4.3}
note that $x\in\overline\cR$ satisfies the condition (ii) of \S\ref{S-3.5} whenever $x\in L^p(\cM)$
for some $p\in(0,\infty]$. From \eqref{F-4.4} and \eqref{F-2.13} we have
\begin{align*}
\theta_s(a\#_\alpha b)&=(e^{-s/p_1}a)\#_\alpha(e^{-s/p_2}b) \\
&=e^{-s((1-\alpha)/p_1+\alpha/p_2)}(a\#_\alpha b)=e^{-s/p}(a\#_\alpha b),
\qquad s\in\bR,
\end{align*}
showing $a\#_\alpha b\in L^p(M)$. On the other hand, from H\"older's inequality for Haagerup's
$L^p$-spaces, we find that $a^{1-\alpha}b^\alpha\in L^p(\cM)$ and
\begin{align*}
\|a^{1-\alpha}b^\alpha\|_p
\le\|a^{1-\alpha}\|_{p_1/(1-\alpha)}\|b^\alpha\|_{p_2/\alpha}
=\|a\|_{p_1}^{1-\alpha}\|b\|_{p_2}^\alpha,
\end{align*}
which is the second inequality of \eqref{F-4.5}. By \eqref{F-3.30} (for $a,b\in\overline\cR_+$)
and \eqref{F-4.3} we further have
$$
e^{1/p}\|a\#_\alpha b\|_p=\Lambda_1(a\#_\alpha b)
\le\Lambda_1(a^{1-\alpha}b^\alpha)=e^{1/p}\|a^{1-\alpha}b^\alpha\|_p,
$$
which gives the first inequality of \eqref{F-4.5}.
\end{proof}

%%%%%%%%%%%%%%%%%%%% \label{R-4.7} %%%%%%%%%%%%%%%%%%%%
\begin{remark}\label{R-4.7}\rm
In the situation of Theorem \ref{T-4.6}, for any $r>0$ we have
$$
\|(a^r\#_\alpha b^r)^{1/r}\|_p\le\|\,|a^{r(1-\alpha)}b^{r\alpha}|^{1/r}\|_p
\le\|a\|_{p_1}^{1-\alpha}\|b\|_{p_2}^\alpha
$$
by applying \eqref{F-4.5} to $a^r,b^r$ in view of $1/(p/r)=(1-\alpha)/(p_1/r)+\alpha/(p_2/r)$.
Since $r>0\mapsto\|\,|a^{r(1-\alpha)}b^{r\alpha}|^{1/r}\|_p$ is monotone increasing
\cite[Theorem 4]{Ko4}, we have
$$
\|(a^r\#_\alpha b^r)^{1/r}\|_p\le\|\,|a^{q(1-\alpha)}b^{q\alpha}|^{1/q}\|_p,
\qquad0<r\le q.
$$
If the conjecture of Problem \ref{Q-3.43} is true, then  it follows that
$r>0\mapsto\|(a^r\#_\alpha b^r)^{1/r}\|_p$ is monotone decreasing and the above inequality
holds for all independent $r,q>0$.
\end{remark}

%%%%%%%%%%%%%%%%%%%%%%%%%%%%%%%%%%%%%%%%%%%
%%%%%%%%%%% Parallel sums of weights %%%%%%%%%%%%%%%%%%
%%%%%%%%%%%%%%%%%%%%%%%%%%%%%%%%%%%%%%%%%%
\section{Parallel sums of weights}\label{S-5}

We have studied connections for various classes of unbounded objects.
In every case parallel sums are building blocks for connections
(see  Definition \ref{D-2.5}, \S\ref{S-3.2} and Proposition \ref{P-4.5}).
In this section we will study a notion of parallel sums for semi-finite normal weights
on a von Neumann algebra.
A general theory for parallel sums for weights will be presented in \S\ref{S-5.1} and \S\ref{S-5.2}.
In the final subsection \S\ref{S-5.3}  connections for weights will be considered 
based on usual integral expressions.

Throughout the  section let $\cM$ be a von Neumann algebra on a Hilbert space $\cH$.
We fix a  faithful  semi-finite normal weight $\chi$ on the commutant $\cM'$ and use
Connes' spatial derivatives (see Appendix \S\ref{S-B.1}).

%%%%%%%%%%%%%%%%%%%%%%%%%%%%%%%%%%%%%%%%%%%%%%%%%%
%%% \subsection{Parallel sums of spatial derivatives and weights} %%%%%%%%%%%%%%%%
%%%%%%%%%%%%%%%%%%%%%%%%%%%%%%%%%%%%%%%%%%%%%%%%%%
\subsection{Parallel sums of spatial derivatives and weights}\label{S-5.1}

For notational convenience we set
$$
P_0(\cM, \bC):=\mbox{the set of all semi-finite normal weights on $\cM$}.
$$
A (densely defined) positive self-adjoint operator  $T$ acting on $\cH$ is said to be
\emph{$(-1)$-homogeneous} (relative to $\chi$) if
$$
T^{it} y'= \sigma'_{-t}(y')T^{it} 
$$
for each $t \in \bR$ and $y' \in \cM'$.
Here, $\{\sigma'_t\}_{t \in\bR}$ is the modular automorphism group on $\cM'$ induced by $\chi$,
and $T^{it}$ is understood to be defined on the support of $T$.
We take a weight $\phi \in P_0(\cM, \bC)$ and the spatial derivative $d\phi/d\chi$ will be 
considered (see Appendix \S\ref{S-B.1}).
When $\phi$ is faithful, we have
$$
\sigma'_t(y')=(d\phi/d\chi)^{-it}y'(d\phi/d\chi)^{it} \quad \mbox{for $y' \in \cM'$}
$$
(\cite[Theorem 9]{C3}, and also see (i), (ii) at the beginning of \S\ref{S-B.3}) and hence 
$$
(d\phi/d\chi)^{it} y'=(d\phi/d\chi)^{it} y'(d\phi/d\chi)^{-it } (d\phi/d\chi)^{it}=\sigma'_{-t}(y')(d\phi/d\chi)^{it},
$$
that is, $d\phi/d\chi$ is $(-1)$-homogeneous.
The $(-1)$-homogeneity condition is known to characterize positive self-adjoint operators $T$ of the form
$T=d\phi/d\chi$ for some $\phi \in P_0(\cM, \bC)$ (see \cite[Theorem 13]{C3}).

We set
$$
\overline{\cM^+_{-1}}:=
\mbox{the set of all $(-1)$-homogeneous positive self-adjoint operators on $\cH$}.
$$
From the explanation so far 
\begin{equation}\label{F-5.1} %%%%%%%%%%%%%%%%%%%%%%%%%%%%%% \label{F-5.1}
\phi \in P_0(\cM, \bC)\ \longleftrightarrow\ d\phi/d\chi \in \overline{\cM^+_{-1}}
\end{equation}
is an order preserving one-to-one correspondence.

Here is our strategy (for investigating a notion of  parallel sums of weights): 
A reasonable theory on parallel sums of positive forms was worked out in \cite{Ko6}.
A positive form means a lower semi-continuous positive quadratic form defined 
everywhere in $\cH$ with the value $+\infty$ allowed, i.e., an element in the
extended positive part $\widehat{B(\cH)}_+$ (see \S\ref{S-2.1}).
Positive self-adjoint operators are positive forms so that their parallel sums can be defined. 
Thus, one hopes that the above one-to-one correspondence \eqref{F-5.1} can be used 
to define and investigate ``parallel sums of weights".

We start from two semi-finite normal weights $\phi, \psi \in P_0(\cM, \bC)$
so that we have (densely defined) positive self-adjoint operators $d\phi/d\chi$ and $d\psi/d\chi$.
Regard them as positive forms, which means that $d\phi/d\chi$ for instance is identified with
$$
q_{d\phi/d\chi}(\xi)=
\left\{
\begin{array}{cc}
\Vert (d\phi/d\chi)^{1/2}\xi \Vert^2& \mbox{when $\xi \in {\mathcal D}((d\phi/d\chi)^{1/2})$},
\\[2mm]
+\infty & \mbox{when $\xi \not\in {\mathcal D}((d\phi/d\chi)^{1/2})$}.
\end{array}
\right.
$$
As positive forms we define their parallel sum, i.e., 
$$
(d\phi/d\chi):(d\psi/d\chi)=\bigl( (d\phi/d\chi)^{-1}+(d\psi/d\chi)^{-1} \bigr)^{-1}.
$$
Since $(d\phi/d\chi):(d\psi/d\chi)$ is majorized by $d\phi/d\chi$ and $d\psi/d\chi$,
this parallel sum (defined as a positive form) corresponds to a (densely defined) positive self-adjoint operator. 
Recall that spatial derivatives $d\phi/d\chi$ and $d\psi/d\chi$ are $(-1)$-homogeneous.
If the parallel sum $(d\phi/d\chi):(d\psi/d\chi)$ is  proved to be $(-1)$-homogeneous,
then it corresponds to some semi-finite normal weight 
by the correspondence \eqref{F-5.1}. 
The difficulty here is that computation of its $it$-power $\bigr((d\phi/d\chi):(d\psi/d\chi)\bigr)^{it}$
seems impossible so that checking $(-1)$-homogeneity looks hopeless in this approach
(at least to the authors).
We will overcome this difficulty by making use of Haagerup's method of identifying weights 
with certain positive self-adjoint operators (which plays a crucial role in his theory 
on non-commutative $L^p$-spaces). 

Here, $\varphi_0$ is a faithful normal semi-finite weight on $\cM$, and let $\cR:=\cM\rtimes_\sigma\bR$
be the crossed product acting on $L^2(\bR,\cH)=\cH\otimes L^2(\bR,dt)$ with respect to the modular
automorphism group $\sigma_t=\sigma_t^{\varphi_0}$. Recall that $\cR$ is semi-finite with the
canonical trace $\tau$ and the dual action $\theta_s$ is given on $\cR$. These materials are
essential to define Haagerup's $L^p$-spaces as explained in Appendix \ref{S-A}. We set
\begin{eqnarray*}
&&{\bold H}={\bold H}(\cR,\theta)
:=\mbox{the set of all positive self-adjoint operators $h$ affiliated with $\cR$}\\
&& \qquad \qquad \qquad\qquad \mbox{satisfying $\theta_s(h)=e^{-s}h$ for each $s \in \bR$}.
\end{eqnarray*}
From the discussions so far and in Appendix \ref{S-A}, by combining \eqref{F-5.1} with \eqref{F-A.3}
we have the following order preserving one-to-one bijective correspondence:
\begin{equation}\label{F-5.2} %%%%%%%%%%%%%%%%%%%%%%%%%%%%% \label{F-5.2}
h_{\phi} \in {\bold H}  \ \longleftrightarrow \ d\phi/d\chi \in \overline{\cM^+_{-1}}.
\end{equation}
This correspondence admits the following explicit description (see \cite[Theorem 2]{H} 
and also \cite[Chap.~IV, Proposition 4]{Te1}):
%%%%%%%%%%%%%%%%%%%%% \label{L-5.1} %%%%%%%%%%%%%%%%%%%%
\begin{lemma}\label{L-5.1}
With the unitary operator $U$ on $L^2(\bR, \cH)$ defined by
$$
(U\xi)(t)=(d\varphi_0/d\chi)^{it}\xi(t) \quad (\xi \in L^2(\bR, \cH))
$$
we have
$$
Uh_{\phi}U^*=(d\phi/d\chi) \otimes H \quad \mbox{for each $\phi \in P_0(\cM, \bC)$},
$$
where $H$ is the generator of the translations $\lambda(t)\xi=\xi(\cdot-t)$,
$\xi\in L^2(\bR,\cH)$, i.e., $\lambda(t)=H^{it}$ $($see Appendix \ref{S-A}$)$.
\end{lemma}

The images (under this transformation $U\cdot U^*$) of the standard
generators in the crossed product $\cR=\cM \rtimes_{\sigma} \bR$ 
can be also explicitly written down (see \cite[Chap.~IV, Proposition 3]{Te1}) although
they are not needed here.

So far (also in Appendix \ref{S-A}) we have used the identification
$L^2(\bR, \cH) \cong \cH \otimes L^2(\bR,dt)$. 
In the proof of the next lemma we will use 
$L^2(\bR, \cH) \cong L^2(\bR,dt) \otimes \cH$ instead,
which seems more fitting to present arguments.
We take the Fourier transform on $L^2(\bR,dt)$ and arrive at
\begin{equation}\label{F-5.3} %%%%%%%%%%%%%%%%%%%%%%%%%%%%%% \label{F-5.3}
Uh_{\phi}U^*=e^t \otimes d\phi/d\chi,
\end{equation}
where $e^t$ means the multiplication operator $m_{e^{\cdot}}$
on $L^2(\bR,dt)$ induced by the exponential
function $t \mapsto e^t$. We note that the action of  $e^t \otimes d\phi/d\chi$ to a vector
$\xi \in  L^2(\bR, \cH)$ looks like
\begin{equation}\label{F-5.4} %%%%%%%%%%%%%%%%%%%%%%%%%%%%%% \label{F-5.4}
\left (\left( e^t \otimes d\phi/d\chi \right) \xi\right)(t)=e^t\,(d\phi/d\chi)\xi(t).
\end{equation}
When one identifies $L^2(\bR,dt) \otimes \cH$ with the disintegration 
${\displaystyle \int_{\bR}^{\oplus} \cH\, dt}$, this means 
$$
Uh_{\phi}U^*=\int_{\bR}^{\, \oplus} e^t(d\phi/d\chi)\, dt
$$
(see \cite{L, Nu} for instance). 
However, for computations of parallel sums
we will have to regard the positive self-adjoint 
operator $e^t \otimes d\phi/d\chi$ as a positive form,
and probably reduction theory for positive forms has 
not been properly formulated. 
Thus, reduction picture for positive forms will not be used in the poof of the next lemma.
However, we will use this picture to express spectral projections of relevant positive forms,
which is perfectly legitimate.

%%%%%%%%%%%%%%%%%%%%%%% \label{L-5.2} %%%%%%%%%%%%%%%%%%%%%
\begin{lemma}\label{L-2.3}
For each $\phi$ and $\psi$ we have
\begin{equation}\label{F-5.5} %%%%%%%%%%%%%%%%%%%%%%%%%%%%%% \label{F-5.5}
\left( e^t \otimes (d\phi/d\chi) \right)
:
\left( e^t \otimes (d\psi/d\chi)\right)
=
e^t \otimes \left((d\phi/d\chi):(d\psi/d\chi)\right).
\end{equation}
\end{lemma}
\begin{proof}
Our strategy is to compare spectral projections $E_I(e^t \otimes (d\phi/d\chi))$
and $e_I(d\phi/d\chi)$ (with $I \subseteq [0,+\infty]$).  
Since $e^t \otimes (d\phi/d\chi)$ is actually
 a (densely defined) positive self-adjoint operator on $L^2(\bR,dt) \otimes \cH$, we have
$$
E_{\{+\infty\}} (e^t \otimes (d\phi/d\chi))=0.
$$
Also, since the support of $d\phi/d\chi$ is $s(\phi)$ (i.e., the support as a weight), we have
$$
E_{\{0\}}(e^t \otimes (d\phi/d\chi))=1_{L^2(\bR,dt)} \otimes (1-s(\phi)).
$$
For each $\lambda>0$, the expression \eqref{F-5.4} shows
\begin{equation}\label{F-5.6} %%%%%%%%%%%%%%%%%%%%%%%%%%%%%%% \label{F-5.6}
E_{[\lambda, \infty)}
(e^t \otimes (d\phi/d\chi)) 
=\int^{\oplus}_{\bR} e_{[e^{-t}\lambda, \infty)}(d\phi/d\chi)\,dt,
\end{equation}
where a disintegration symbol is used to express the projection
$E_{[\lambda, \infty)}(e^t \otimes (d\phi/d\chi))$.
Thus, spectral projections $E_I(( e^t \otimes (d\phi/d\chi))^{-1})$ for the
positive form  $( e^t \otimes (d\phi/d\chi) )^{-1}$ (on $L^2(\bR,dt)\otimes \cH$)
are given by
\begin{eqnarray*}
&&
E_{\{ +\infty \}} (( e^t \otimes (d\phi/d\chi))^{-1})
=E_{\{0\}}(e^t \otimes (d\phi/d\chi))=1_{L^2(\bR,dt)} \otimes (1-s(\phi)),\\
&&
E_{\{0\}} (( e^t \otimes (d\phi/d\chi))^{-1})
=E_{\{+\infty\}} (e^t \otimes (d\phi/d\chi))=0,\\
&&
E_{(0, \lambda]}(( e^t \otimes (d\phi/d\chi))^{-1} )
=E_{ [1/\lambda,\infty) } (e^t \otimes (d\phi/d\chi) )\\
&&
\qquad \qquad
=\int^{\oplus}_{\bR} 
e_{[e^{-t}/\lambda, \infty)}
(d\phi/d\chi)\,dt
=\int^{\oplus}_{\bR} 
e_{(0, e^{t}\lambda]}
((d\phi/d\chi)^{-1})\,dt.
\end{eqnarray*}
Note that \eqref{F-5.6} was used, and  $e_I((d\phi/d\chi)^{-1})$ of course means a
spectral projection for $(d\phi/d\chi)^{-1}$.
The above $E_{(0, \lambda]}
(( e^t \otimes (d\phi/d\chi))^{-1})$
becomes the spectral projection
of the tensor product of the multiplication operator $e^{-t} \left(=m_{e^{-\cdot}}\right)$
and the inverse of $d\phi/d\chi$ (considered as a non-singular operator on $s(\phi)\cH$).
From the above computations it is easy to see that the positive form
$\left( e^t \otimes (d\phi/d\chi)\right)^{-1}$ (on $L^2(\bR,dt) \otimes \cH$) is given by
$$
\left(( e^t \otimes (d\phi/d\chi))^{-1} \xi,\xi\right)
=\int_{\bR}e^{-t}( (d\phi/d\chi)^{-1}\xi(t),\xi(t))\,dt
$$
for $\xi \in  L^2(\bR,dt) \otimes \cH$.
Here, $(d\phi/d\chi)^{-1}$ should be understood as a positive form 
(with $e_{\{+\infty\}}((d\phi/d\chi)^{-1})=e_{\{0\}}(d\phi/d\chi)=1-s(\phi)$).

We obviously have the same formula as above for $d\psi/d\chi$ so that the (form) sum
$\left( e^t \otimes (d\phi/d\chi)\right)^{-1}+\left( e^t \otimes (d\psi/d\chi)\right)^{-1}$ 
(on $L^2(\bR,dt) \otimes \cH$) is given by
\begin{eqnarray*}
&&
\left( (( e^t \otimes (d\phi/d\chi))^{-1}+( e^t \otimes (d\psi/d\chi))^{-1} )\xi,\xi \right)
\\
&&
\qquad
=\int_{\bR}e^{-t}( (d\phi/d\chi)^{-1}\xi(t),\xi(t))\,dt+\int_{\bR}e^{-t}( (d\psi/d\chi)^{-1}\xi(t),\xi(t))\,dt\\
&&
\qquad
=\int_{\bR}e^{-t} ( ( (d\phi/d\chi)^{-1} + (d\psi/d\chi)^{-1})   \xi(t),\xi(t))\,dt
\end{eqnarray*}
with the form sum $(d\phi/d\chi)^{-1} + (d\psi/d\chi)^{-1}$ (on $\cH$).
Thus, spectral projections for
$( e^t \otimes (d\phi/d\chi))^{-1}+( e^t \otimes (d\psi/d\chi))^{-1}$ and those for
$(d\phi/d\chi)^{-1} + (d\psi/d\chi)^{-1}$ are related in the following fashion:
$$
E_{(0, \lambda]}
(( e^t \otimes (d\phi/d\chi))^{-1}+( e^t \otimes (d\psi/d\chi))^{-1})
=\int^{\oplus}_{\bR}
e_{(0,e^t\lambda]}
((d\phi/d\chi)^{-1} + (d\psi/d\chi)^{-1})\,dt
$$
with
\begin{eqnarray*}
&&
E_{\{0\}} (( e^t \otimes (d\phi/d\chi))^{-1}+( e^t \otimes (d\psi/d\chi))^{-1} )=0,\\
&& 
E_{\{+\infty\}}
(( e^t \otimes (d\phi/d\chi))^{-1}+( e^t \otimes (d\psi/d\chi))^{-1})\\
&& \qquad \qquad \qquad
=1_{L^2(\bR,dt)} \otimes 
e_{\{+\infty\}}
((d\phi/d\chi)^{-1}+(d\psi/d\chi)^{-1}).
\end{eqnarray*}
Hence, spectral projections for the parallel sum in question are given by
\begin{eqnarray*}
&&
E_{[\lambda,\infty)}
(
(
( e^t \otimes (d\phi/d\chi))^{-1}+( e^t \otimes (d\psi/d\chi))^{-1}
)^{-1}
)\\
&& \qquad \qquad
=E_{(0, 1/\lambda]}
(
( e^t \otimes (d\phi/d\chi))^{-1}+( e^t \otimes (d\psi/d\chi))^{-1}
)
\\
&&
\qquad \qquad
=\int^{\oplus}_{\bR}
e_{(0,e^t/\lambda]}
(
(d\phi/d\chi)^{-1} + (d\psi/d\chi)^{-1}
)
\,dt
\\
&& \qquad \qquad
=\int^{\oplus}_{\bR}
e_{[e^{-t}\lambda, \infty)}
(
(
(d\phi/d\chi)^{-1} + (d\psi/d\chi)^{-1}
)^{-1}
)
\,dt.
\end{eqnarray*}
We also have 
\begin{eqnarray}
&&
E_{\{+\infty\}}
(
(
( e^t \otimes (d\phi/d\chi))^{-1}+( e^t \otimes (d\psi/d\chi))^{-1}
)^{-1}
)
=0,
\label{F-5.7}\\%%%%%%%%%%%%%%%%%%%%%%%%%%%%%%%%%%%%% \label{F-5.7}
&&
E_{\{0\}}
(
(
( e^t \otimes (d\phi/d\chi))^{-1}+( e^t \otimes (d\psi/d\chi))^{-1}
)^{-1}
)
\nonumber
\\
&& \qquad \qquad \qquad
=1_{L^2(\bR,dt)} \otimes 
e_{\{+\infty\}}
(
(d\phi/d\chi)^{-1}+(d\psi/d\chi)^{-1}
)
\nonumber
\\
&&
\qquad \qquad \qquad
=1_{L^2(\bR,dt)} \otimes 
e_{\{0\}}
(
(
(d\phi/d\chi)^{-1} + (d\psi/d\chi)^{-1}
)^{-1}
).
\label{F-5.8} %%%%%%%%%%%%%%%%%%%%%%%%%%%%%%%%%%%%% \label{F-5.8}
\end{eqnarray}
Thus, we conclude
$$
\left(\left( e^t \otimes (d\phi/d\chi)\right)^{-1}+\left( e^t \otimes (d\psi/d\chi)\right)^{-1}\right)^{-1}
=
e^t \otimes \left((d\phi/d\chi)^{-1} + (d\psi/d\chi)^{-1}\right)^{-1},
$$
which is exactly \eqref{F-5.5}.
\end{proof}

The relation \eqref{F-5.7} corresponds to the fact that the parallel sum is a
densely defined positive self-adjoint operator while \eqref{F-5.8} indicates possibility
of a non-trivial kernel. We are now ready to prove the next result.

%%%%%%%%%%%%%%%%%%%%%% \label{P-5.3} %%%%%%%%%%%%%%%%%%%%%%
\begin{proposition}\label{P-5.3}
For $\phi$, $\psi$ in $P_0(\cM, \bC)$ the parallel sum 
$$
(d\phi/d\chi):(d\psi/d\chi)=\bigl( (d\phi/d\chi)^{-1}+(d\psi/d\chi)^{-1} \bigr)^{-1}
$$
$($which is defined in \S\ref{S-2.1} as a positive form, but actually a densely defined positive self-adjoint 
operator as was pointed out before$)$ is $(-1)$-homogeneous relative to $\chi$.
\end{proposition}

\begin{proof}
From \eqref{F-5.3} and \eqref{F-5.5} we have
\begin{align}\label{F-5.9}
U(h_{\phi}:h_{\psi})U^*=e^t \otimes \left((d\phi/d\chi):(d\psi/d\chi)\right).
\end{align}
Thus, it remains to show that
the parallel sum $h_{\phi}:h_{\psi}$ belongs to ${\bold H}$ (see \eqref{F-5.2}).  
Let $q_{\phi}$, $q_{\psi}$ be positive forms corresponding to positive self-adjoint operators $h_{\phi}$,
$h_{\psi}$ respectively. Since $h_{\phi}$ is affiliated with $\cR$ and has the scaling
property $\theta_s(h_{\phi})=e^{-s}h_{\phi}$, the associated form $q_{\phi}$ satisfies
$$
q_{\phi}(u'\xi)=q_{\phi}(\xi) \quad \mbox{and} \quad q_{\phi}(\mu(s)^*\xi)=e^{-s}q_{\phi}(\xi)
$$
for any unitary $u' \in \cR'$ and $\mu(s)$ (defined by \eqref{F-A.1}).

By the variational expression for inverses (see \eqref{F-2.2} in \S\ref{S-2.1}) we have
\begin{align*}
q_{\phi}^{-1}(u'\xi)
&=\sup_{\zeta} \frac{|(u'\xi,\zeta)|^2}{q_{\phi}(\zeta)}
=\sup_{\zeta} \frac{|(\xi,u'^*\zeta)|^2}{q_{\phi}(\zeta)}\\
&=\sup_{\zeta} \frac{|(\xi,\zeta)|^2}{q_{\phi}(u'\zeta)}
=\sup_{\zeta} \frac{|(\xi,\zeta)|^2}{q_{\phi}(\zeta)}=q_{\phi}^{-1}(\xi),\\
q_{\phi}^{-1}(\mu(s)^*\xi)
&=\sup_{\zeta} \frac{|(\mu(s)^*\xi,\zeta)|^2}{q_{\phi}(\zeta)}
=\sup_{\zeta} \frac{|(\xi,\mu(s)\zeta)|^2}{q_{\phi}(\zeta)}\\
&=\sup_{\zeta} \frac{|(\xi,\zeta)|^2}{q_{\phi}(\mu(s)^*\zeta)}
=\sup_{\zeta} \frac{|(\xi,\zeta)|^2}{e^{-s}q_{\phi}(\zeta)}=e^s q_{\phi}^{-1}(\xi)
\end{align*}
(where sup is taken over all vectors $\zeta$),  and of course the same 
holds true for $q_{\psi}$. 
It is then trivial from the definition that the form sum $q_{\phi}^{-1}+q_{\psi}^{-1}$
has the same invariance, that is,
\begin{eqnarray*}
&&
\bigl(q_{\phi}^{-1}+q_{\psi}^{-1}\bigr)(u'\xi) = \bigl(q_{\phi}^{-1}+q_{\psi}^{-1}\bigr)(\xi),\\
&&
\bigl(q_{\phi}^{-1}+q_{\psi}^{-1}\bigr)(\mu(s)^*\xi) = e^s\bigl(q_{\phi}^{-1}+q_{\psi}^{-1}\bigr)(\xi). 
\end{eqnarray*}
Thus, by repeating parallel arguments as above, we conclude 
\begin{eqnarray*}
&&
\bigl(q_{\phi}^{-1}+q_{\psi}^{-1}\bigr)^{-1}(u'\xi) = \bigl(q_{\phi}^{-1}+q_{\psi}^{-1}\bigr)^{-1}(\xi),\\
&&
\bigl(q_{\phi}^{-1}+q_{\psi}^{-1}\bigr)^{-1}(\mu(s)^*\xi) = e^{-s}\bigl(q_{\phi}^{-1}+q_{\psi}^{-1}\bigr)^{-1}(\xi). 
\end{eqnarray*}
This means that the positive self-adjoint operator $h_{\phi}:h_{\psi}$
(corresponding to the positive form $\bigl(q_{\phi}^{-1}+q_{\psi}^{-1}\bigr)^{-1}$)
is affiliated with $\cR$
and $\theta_s\left(h_{\phi}:h_{\psi}\right)=e^{-s} (h_{\phi}:h_{\psi})$ holds true, that is, 
$h_{\phi}:h_{\psi} \in {\bold H}$.
\end{proof}

%%%%%%%%%%%%%%%%%%%% \label{D-5.4} %%%%%%%%%%%%%%%%%%%%%%%%%
\begin{definition}\label{D-5.4}\rm 
We assume $\phi, \psi \in P_0(\cM, \bC)$. 
The weight in $P_0(\cM, \bC)$ corresponding  $($via \eqref{F-5.1}$)$ to the $(-1)$-homogeneous
operator $(d\phi/d\chi):(d\psi/d\chi) \in {\bold H}$ (in Proposition \ref{P-5.3})
is called the \emph{parallel sum} of $\phi$ and $\psi$,
and denoted by $\phi:\psi$, that is,
$$
d(\phi:\psi)/d\chi=(d\phi/d\chi):(d\psi/d\chi)
\ \left(\in \overline{\cM^+_{-1}}\right).
$$
\end{definition}

\begin{remark}\label{R-5.5}\rm
The parallel sum $\phi:\psi$ is independent of not only the choice of a representing Hilbert space
$\cH$ for $\cM$ but also the choice of a faithful semi-finite normal weight $\chi$ on $\cM'$. Indeed,
assume that $\kappa:\cM$ (on $\cH$) $\to\cM_1$ (on $\cH_1$) is an isomorphism, and let $\ffi_0$
be a faithful normal semi-finite weight on $\cM$ and $\ffi_1:=\ffi_0\circ\kappa^{-1}$ on $\cM_1$.
Also, let $\chi$ and $\chi_1$ be faithful semi-finite normal weights on $\cM'$ and $\cM_1'$
respectively, for which $U$ and $U_1$ are defined as in Lemma \ref{L-5.1}. The
$\kappa:\cM\to\cM_1$ extends to an isomorphism
$\kappa:\cR=\cM\rtimes_{\sigma^{\ffi_0}}\bR\to\cR_1=\cM_1\rtimes_{\sigma^{\ffi_1}}\bR$
so that the canonical trace and the dual action on $\cR_1$ are $\tau_1=\tau\circ\kappa^{-1}$ and
$\theta_s^1:=\kappa\circ\theta_s\circ\kappa^{-1}$. Let $\phi,\psi\in P_0(\cM,\bC)$ and set
$\phi_1:=\phi\circ\kappa^{-1}$, $\psi_1:=\psi\circ\kappa^{-1}\in P_0(\cM_1,\bC)$. Then it is standard
to see that $h_{\phi_1}=\kappa(h_\phi)$, $h_{\psi_1}=\kappa(h_\psi)$ and
$h_{\phi_1}:h_{\psi_1}=\kappa(h_\phi:h_\psi)$, where $\kappa$ gives rise to a map
$\kappa:{\bold H}(\cR,\theta)\to{\bold H}(\cR_1,\theta^1)$ in a natural way. Since \eqref{F-5.9} and
Definition \ref{D-5.4} imply
\begin{align*}
Uh_{\phi:\psi}U^*&=e^t\otimes(d(\phi:\psi)/d\chi)=U(h_\phi:h_\psi)U^*, \\
U_1h_{\phi_1:\psi_1}U_1^*&=e^t\otimes(d(\phi_1:\psi_1)/d\chi_1)=U_1(h_{\phi_1}:h_{\psi_1})U_1^*,
\end{align*}
it follows that $h_{\phi:\psi}=h_\phi:h_\psi$ and $h_{\phi_1:\psi_1}=h_{\phi_1}:h_{\psi_1}$. Therefore,
$h_{\phi_1:\psi_1}=\kappa(h_{\phi:\psi})$ so that we have the desired identity
$$
\phi_1:\psi_1=(\phi:\psi)\circ\kappa^{-1}.
$$
Note that the fact of independence shown above is included in Theorem \ref{T-5.7} below while the
proof of the latter is based on the former.
\end{remark}

%%%%%%%%%%%%%%%%%%%%%%%%%%%%%%%%%%%%%%%%%%%%%%%%%
%%%%%%%%%%% \subsection{Properties} %%%%%%%%%%%%%%%%%%%%%%%%%
%%%%%%%%%%%%%%%%%%%%%%%%%%%%%%%%%%%%%%%%%%%%%%%%%
\subsection{Properties}\label{S-5.2}

%%%%%%%%%%%%%%%%%%% \label{L-5.6} %%%%%%%%%%%%%%%%%%%%%%%
\begin{lemma}\label{L-5.6}
For $\phi, \psi \in P_0(\cM, \bC)$ their parallel sum $\phi:\psi \ (\in P_0(\cM, \bC) )$
satisfies
$$ 
(\phi:\psi)((x_1+x_2)^*(x_1+x_2)) \leq \phi(x_1^*x_1) + \psi(x_2^*x_2) \ (\leq +\infty)
$$
for all $x_1, x_2 \in \cM$.
\end{lemma}
\begin{proof}
We make use of a family $\{\xi_{\iota}\}_{\iota \in I}$ of vectors in $D(\cH,\chi)$ stated in Lemma \ref{L-B.4}.
Since $\sum_{\iota \in I}\theta^{\chi}(\xi_{\iota}, \xi_{\iota})=1$, we have
\begin{align*}
(x_1+x_2)^*(x_1+x_2)&=\sum_{{\iota} \in I} \theta^{\chi}((x_1+x_2)^*\xi_{\iota},(x_1+x_2)^*\xi_{\iota}),\\
x_i^*x_i&=\sum_{{\iota} \in I} \theta^{\chi}(x_i^*\xi_{\iota}, x_i^*\xi_{\iota})
\qquad  (\mbox{for \ $i=1,2$})
\end{align*}
thanks to \eqref{F-B.3}. With these expressions we compute
\begin{eqnarray*}
&&
(\phi:\psi)((x_1+x_2)^*(x_1+x_2))
=
\sum_{\iota \in I}\bigl(\phi:\psi)(\theta^{\chi}((x_1+x_2)^*\xi_{\iota},(x_1+x_2)^*\xi_{\iota}\bigr)\\
&& \qquad \qquad
=
\sum_{\iota \in I}  q_{\phi:\psi}\left((x_1+x_2)^*\xi_{\iota}\right)\\
&& \qquad \qquad
=
\sum_{\iota \in I} \|(d(\phi:\psi)/d\chi)^{1/2} (x_1+x_2)^*\xi_{\iota} \|^2\\
&& \qquad \qquad
=
\sum_{\iota \in I} \|(
\bigl(d\phi/d\chi):(d\psi/d\chi)\bigr)^{1/2} (x_1+x_2)^*\xi_{\iota} \|^2\\
&& \qquad \qquad
\leq
\sum_{\iota \in I}
\|(d\phi/d\chi)^{1/2} x_1^*\xi_{\iota} \|^2 
+
\sum_{\iota \in I}
\|(d\psi/d\chi)^{1/2} x_2^*\xi_{\iota} \|^2\\
&& \qquad \qquad
=
\phi(x_1^*x_1)+\psi(x_2^*x_2).
\end{eqnarray*}
Here, the inequality appearing in the middle of computations
follows from Theorem \ref{T-2.2},
and also normality of involved weights, \eqref{F-B.4}, Definition \ref{D-B.2}
and Definition \ref{D-5.4} are repeatedly used.
\end{proof}

%%%%%%%%%%%%%%%%%%% \label{T-5.7} %%%%%%%%%%%%%%%%%%%%%%%%%
\begin{theorem}\label{T-5.7}
We assume $\phi, \psi \in P_0(\cM,{\bold C})$. 
The parallel sum $\phi:\psi \ (\in P_0(\cM,{\bold C}) )$ is the maximum
of all semi-finite normal weights $\omega$ satisfying
\begin{equation}\label{F-5.10} %%%%%%%%%%%%%%%%%%%%%%%%%%%%% \label{F-5.10} 
\omega((x_1+x_2)^*(x_1+x_2)) \leq \phi(x_1^*x_1) + \psi(x_2^*x_2) \ (\leq +\infty)
\end{equation}
for all $x_1, x_2 \in \cM$.
\end{theorem}
\begin{proof}
Normality and semi-finiteness (for weights) are expressed in terms of the $\sigma$-weak topology,
and this topology  is independent of the Hilbert space on which $\cM$ acts. Therefore,
in view of Remark \ref{R-5.5} we may and do assume that the action of $\cM$
on $\cH$ is standard. We have shown Lemma \ref{L-5.6}, and it remains to show the maximality of 
the parallel sum $\phi:\psi$. Let us assume that a (semi-finite) normal weight $\omega$
satisfies the inequality \eqref{F-5.10} (and $\omega \leq \phi:\psi$ is to be shown).

We start from two vectors
$$
\zeta_1 \in D(\cH,\chi) \cap {\mathcal D}((d\phi/d\chi)^{1/2})
\quad
\mbox{and}
\quad
\zeta_2 \in D(\cH,\chi) \cap {\mathcal D}((d\psi/d\chi)^{1/2}).
$$
Both of $R^{\chi}(\zeta_i)$ ($i=1,2$) are bounded operators from $\cH_{\chi}$ to $\cH$ and
satisfy $y'R^{\chi}(\zeta_i) = R^{\chi}(\zeta_i)\pi_{\chi}(y')$ for $y' \in \cM'$.
Since the action of $\cM$ on ${\mathcal H}$ is standard, there exists a surjective
isometry $I: \cH \rightarrow \cH_{\chi}$ intertwining the respective two $\cM'$-actions, i.e., $\pi_{\chi}(y')I=Iy'$.
This means that the compositions $R^{\chi}(\zeta_i)I$ are bounded operators on $\cH$ and
commute with $y' \in \cM'$, that is, $R^{\chi}(\zeta_i)I \in \cM$. We note
\begin{align*}
\theta^{\chi}(\zeta_i,\zeta_i)
&=R^{\chi}(\zeta_i)R^{\chi}(\zeta_i)^*
=(R^{\chi}(\zeta_i)I)(R^{\chi}(\zeta_i)I)^* \quad (\mbox{for $i=1,2$}),\\
\theta^{\chi}(\zeta_1+\zeta_2, \zeta_1+\zeta_2)
&=R^{\chi}(\zeta_1+\zeta_2)R^{\chi}(\zeta_1+\zeta_2)^*\\
&=\left(R^{\chi}(\zeta_1)+R^{\chi}(\zeta_2)\right)\left(R^{\chi}(\zeta_1)+R^{\chi}(\zeta_2)\right)^*\\
&=\left(R^{\chi}(\zeta_1)I+R^{\chi}(\zeta_2)I\right)\left(R^{\chi}(\zeta_1)I+R^{\chi}(\zeta_2)I\right)^*.
\end{align*}
We then compute
\begin{align}\label{F-5.11} %%%%%%%%%%%%%%%%%%%%%%%%%%%%% \label{F-5.11}
\|(d\omega/d\chi)^{1/2}(\zeta_1+\zeta_2)\|^2
&=q_{\omega}(\zeta_1+\zeta_2)
=\omega\bigl( \theta^{\chi}(\zeta_1+\zeta_2 ,\zeta_1+\zeta_2) \bigr)
\nonumber\\
&=\omega\bigl(
\left(R^{\chi}(\zeta_1)I+R^{\chi}(\zeta_2)I\right)\left(R^{\chi}(\zeta_1)I+R^{\chi}(\zeta_2)I\right)^*\bigr)
\nonumber\\
&\leq\phi\bigl(
\left(R^{\chi}(\zeta_1)I\right)\left(R^{\chi}(\zeta_1)I\right)^*\bigr)
+\psi\bigl(
\left(R^{\chi}(\zeta_2)I\right)\left(R^{\chi}(\zeta_2)I\right)^*\bigr)
\nonumber\\
&=\phi(\theta^{\chi}(\zeta_1,\zeta_1)) + \psi(\theta^{\chi}(\zeta_2,\zeta_2))
=q_{\phi}(\zeta_1) + q_{\psi}(\zeta_2)
\nonumber\\
&=\|(d\phi/d\chi)^{1/2}\zeta_1\|^2  + \|(d\psi/d\chi)^{1/2}\zeta_2\|^2.
\end{align} 
Here, the inequality comes from the assumption \eqref{F-5.10} (with $R^{\chi}(\zeta_i)I \in \cM$) and
of course \eqref{F-B.4} and Definition \ref{D-B.2} are repeatedly used. 

To get the desired conclusion, it suffices to show that the inequality \eqref{F-5.11} remains valid
for arbitrary vectors $\zeta_1, \zeta_2 \in \cH$. Indeed, this would mean 
$$
d\omega/d\chi \leq (d\phi/\chi):(d\phi/d\chi)
$$
thanks to the maximality of $(d\phi/\chi):(d\phi/d\chi)$ stated in Theorem \ref{T-2.2}.
However, since $(d\phi/\chi):(d\phi/d\chi) = d(\phi:\psi)/d\chi$ according to Definition \ref{D-5.4}, 
we have $d\omega/d\chi \leq d(\phi:\psi)/d\chi$, showing $\omega \leq \phi:\psi$.

To show \eqref{F-5.11} in full generality, we may and do assume $\zeta_1 \in {\mathcal D}((d\phi/\chi)^{1/2})$
and $\zeta_2 \in {\mathcal D}((d\psi/\chi)^{1/2})$.
Indeed, otherwise \eqref{F-5.11} trivially holds true, the right hand side being $+\infty$.
Thanks to the core condition stated in Remark \ref{R-B.3},(iii) one can take sequences
$\{ \zeta_{1,n}  \}_{n \in \bN}$ from $D(\cH,\chi) \cap {\mathcal D}((d\phi/d\chi)^{1/2})$ 
and
$\{ \zeta_{2,n}  \}_{n \in \bN}$ from $D(\cH,\chi) \cap {\mathcal D}((d\psi/d\chi)^{1/2})$
satisfying
\begin{eqnarray*}
&&
\zeta_{1,n} \rightarrow \zeta_1, \qquad (d\phi/d\chi)^{1/2}\zeta_{1,n}  \rightarrow (d\phi/d\chi)^{1/2}\zeta_1,\\
&&
\zeta_{2,n} \rightarrow \zeta_2, \qquad (d\psi/d\chi)^{1/2}\zeta_{2,n}  \rightarrow (d\psi/d\chi)^{1/2}\zeta_2
\end{eqnarray*}
as $n \rightarrow \infty$. Since $\zeta_{1,n}+\zeta_{2,n} \rightarrow \zeta_1+\zeta_2$ as $n \rightarrow \infty$,
lower semi-continuity of the positive form $\|(d\omega/d\chi)^{1/2} \cdot  \|^2$ and the first half of the proof
(i.e., \eqref{F-5.11} under the additional assumption) imply
\begin{align*}
\|(d\omega/d\chi)^{1/2}(\zeta_1+\zeta_2)\|^2
&\leq\liminf_{n \to \infty}\|(d\omega/d\chi)^{1/2}(\zeta_{1,n}+\zeta_{2,n})\|^2\\
&\leq\liminf_{n \to \infty}
\bigl(\|(d\phi/d\chi)^{1/2}\zeta_{1,n}\|^2  + \|(d\psi/d\chi)^{1/2}\zeta_{2,n}\|^2\bigr)\\
&=\|(d\phi/d\chi)^{1/2}\zeta_{1}\|^2  + \|(d\psi/d\chi)^{1/2}\zeta_{2}\|^2
\end{align*}
as desired.
\end{proof}

%%%%%%%%%%%%%%%%%%%%% \label{R-5.8} %%%%%%%%%%%%%%%%%%%
\begin{remark}\label{R-5.8}\rm
\mbox{}

\begin{itemize}

\item[(i)]
This theorem should be compared with the variational expression \eqref{F-4.2} for 
positive functionals $\phi, \psi \in \cM_*^+$.

\item[(ii)]
All expected properties such as monotonicity, concavity and transformer
inequality (i.e., $a^*(\phi:\psi)a\le(a^*\phi a):(a^*\psi a)$ for all $a\in\cM$)
are valid since parallel sums for weights are defined by using corresponding spatial
derivatives and the latter possesses these properties $($see \cite{Ko6}$)$.
It is also possible to see these properties from Theorem \ref{T-5.7}.

\item[(iii)]
We assume that $\{\omega_n\}_{n \in \bN}$ is a decreasing sequence of semi-finite normal weights.
Let us denote the maximum of all semi-finite normal weights majorized by the point-wise infimum
$\inf_n \omega_n$ by $\Inf_n\,\omega_n$ $($see Remark \ref{R-C.5},(i) in Appendix \ref{S-C}$)$. 
Then we have 
\begin{eqnarray*}
&&
d\left(\Inf_n\,\omega_n\right)/d\chi=\Inf_n\left(d\omega_n/d\chi\right),\\
&&
d\omega_n/d\chi \ \longrightarrow  \ d\left(\Inf_n\,\omega_n\right)/d\chi 
\quad \mbox{in the strong resolvent sense}
\end{eqnarray*}
$($see Theorem \ref{T-C.4} and Remark \ref{R-C.5},(i)$)$.
\end{itemize}
\end{remark}

From the above remark (iii) and Theorem \ref{T-2.3} (see also \cite[Theorem 19]{Ko6})
we have the continuity property for decreasing sequences.

%%%%%%%%%%%%%%%%%%%%% \label{C-5.9} %%%%%%%%%%%%%%%%%%%%%%%
\begin{corollary}\label{C-5.9}
If $\{\phi_n\}_{n \in \bN}$ and $\{\psi_n\}_{n \in \bN}$ are two decreasing 
sequences of semi-finite normal weights, then
$$
\Inf_n \,(\phi_n:\psi_n)= \left(\Inf_n\, \phi_n\right):\left(\Inf_n\,\psi_n\right).
$$
\end{corollary}

\medskip
We now recall Connes' canonical order reversing correspondence between 
\begin{eqnarray*}
&&
P(\cM, \bC):=\mbox{the set of all faithful semi-finite normal weights on $\cM$},\\[3mm]
&&
P(B(\cH), \cM'):=\mbox{the set of all faithful semi-finite normal}\\
&&  \hskip 4cm   \mbox{operator valued weights from $B({\mathcal H})$ to $\cM'$},
\end{eqnarray*}
which is briefly explained in \S\ref{S-B.3}.
We make use of this apparatus to our investigation on parallel sums for weights.
For $\phi, \psi \in P(\cM, \bC)$ we have the canonical operator valued weights 
$\phi^{-1}, \psi^{-1} \in P(B(\cH),\cM')$ characterized by
$$
\chi\circ\phi^{-1}=\Tr((d\chi/d\phi)\,\cdot)
\quad \mbox{and} \quad
\chi\circ\psi^{-1}=\Tr((d\chi/d\psi)\,\cdot)
$$
(see \eqref{F-B.9}), and they sum up to
\begin{align}\label{F-5.12} %%%%%%%%%%%%%%%%%%%%%%%%%%%%% \label{F-5.12}
\chi\circ(\phi^{-1}+\psi^{-1})
&=\Tr((d\chi/d\phi + d\chi/d\psi)\,\cdot)
\nonumber\\
&=\Tr\left(\bigl((d\phi/d\chi)^{-1} + (d\psi/d\chi)^{-1}\bigr)\,\cdot\right).
\end{align}
We assume that the sum $\phi^{-1}+\psi^{-1}$ is a semi-finite operator valued weight  from $B({\mathcal H})$
to $\cM'$
$($which means that $d\chi/d\phi + d\chi/d\psi$ has a dense domain$)$.
Then the above left side is nothing but
$$
\Tr\left(\left( d\chi/d(\phi^{-1}+\psi^{-1})^{-1} \right)\,\cdot\right)
$$
with the inverse weight $(\phi^{-1}+\psi^{-1})^{-1} \in P(\cM, \bC)$
(see \eqref{F-B.10}). Therefore, from \eqref{F-5.12} we get
$$
(d\phi/d\chi)^{-1} + (d\psi/d\chi)^{-1}=d\chi/d(\phi^{-1}+\psi^{-1})^{-1}
\ \left(= \Bigl(d(\phi^{-1}+\psi^{-1})^{-1}/d\chi \Bigr)^{-1} \right)
$$
and hence by taking the inverses of the both sides we conclude
$$
d(\phi^{-1}+\psi^{-1})^{-1}/d\chi=\bigl((d\phi/d\chi)^{-1} + (d\psi/d\chi)^{-1}\bigr)^{-1}
=(d\phi/d\chi):(d\psi/d\chi).
$$
Since the far right side is $d(\phi:\psi)/d\chi$ by the very definition (see Definition \ref{D-5.4}), we have
$
d(\phi^{-1}+\psi^{-1})^{-1}/d\chi = d(\phi:\psi)/d\chi
$.
Therefore, we arrive at the following most natural-looking expression:

%%%%%%%%%%%%%%%%%%% \label{T-5.10} %%%%%%%%%%%%%%%%%%%%%%%%%%
\begin{theorem}\label{T-5.10}
For two weights $\phi, \psi \in P(\cM, \bC)$ with the semi-finite sum 
$\phi^{-1}+\psi^{-1}$ of inverses we have
$$
\phi:\psi=\left(\phi^{-1}+\psi^{-1}\right)^{-1}.
$$ 
Here, the right hand side should be understood in terms of Connes' canonical order reversing 
correspondence between $P(\cM, \bC)$ and $P(B(\cH), \cM')$.
\end{theorem}
The semi-finiteness requirement for $\phi^{-1}+\psi^{-1}$ in the theorem is 
probably redundant. 
One can cut $\cM \subseteq B(\cH)$ by relevant projections, and the
correspondence $P(\cM, \bC) \leftrightarrow P(B(\cH), \cM')$
may  be enlarged to any (operator valued) weights without faithfulness and/or semi-finiteness
requirement.
Then $(\phi^{-1}+\psi^{-1})^{-1}$ might be justified in full generality.

%%%%%%%%%%%%%%%%%%%%%%%%%%%%%%%%%%%%%%%%%%%%%%%%%%%
%%% \subsection{Connections of weights} %%%%%%%%%%%%%%%%%%%%%%%%%%%%%
%%%%%%%%%%%%%%%%%%%%%%%%%%%%%%%%%%%%%%%%%%%%%%%%%%%
\subsection{Connections of weights}\label{S-5.3}

Let us start from a connection $\sigma$ in the Kubo-Ando sense. 
We assume that its representing function $f$, a non-negative operator monotone function on 
$(0,\infty)$, is 
$$
f(s)=\alpha+\beta s+\int_{(0,\infty)}{(1+t)s\over s+t}\,d\mu(s)
$$
with $\alpha,\beta\ge0$ and a finite positive measure $\mu$ on $(0,\infty)$ 
(see \eqref{F-2.7} and \eqref{F-2.8}).
For each
$\phi,\psi\in P_0(\cM,\bC)$ we define a functional $\phi\sigma\psi:\cM_+\to[0,+\infty]$ by
\begin{align}\label{F-5.13} %%%%%%%%%%%%%%%%%%%%%%%%%%%%%%% \label{F-5.13}
(\phi\sigma\psi)(x):=\alpha\phi(x)+\beta\psi(x)
+\int_{(0,\infty)}{1+t\over t}((t\phi):\psi)(x)\,d\mu(t),\qquad x\in\cM_+.
\end{align}
We note that the function $((t\phi):\psi)(x)$ is non-decreasing in $t$ (by 
Remark \ref{R-5.8},(ii)) so that it is a measurable function of $t \in (0,\infty)$.

We have to check if $\phi\sigma\psi$ defined by \eqref{F-5.13} is a normal weight.
For an increasing net $\{x_{\iota}\}_{\iota \in I}$ in $\cM_+$ with $\sup_{\iota}x_{\iota}=x$
we obviously have $\sup_{\iota} ((t\phi):\psi)(x_{\iota})=((t\phi):\psi)(x)$ for each $t>0$.
However, in the proof of the next lemma one cannot use the monotone convergence
theorem to show $\sup_{\iota} (\phi\sigma\psi)(x_{\iota})=(\phi\sigma\psi)(x)$ since
$\{x_{\iota}\}_{\iota \in I}$ is not necessarily a sequence.

%%%%%%%%%%%%%%%%%%%%%% \label{L-5.11} %%%%%%%%%%%%%%%%%
\begin{lemma}\label{L-5.11}
The above $\phi\sigma\psi$ $($defined by \eqref{F-5.13}$)$ is a normal weight on $\cM$
for any $\phi,\psi\in P_0(\cM,\bC)$.
Moreover,
if $\phi+\psi$ is semi-finite, then $\phi\sigma\psi$ is a semi-finite normal weight on $\cM$
$($for any connection $\sigma$$)$.
\end{lemma}

\begin{proof}
The estimate in Lemma \ref{L-3.13} deals with positive forms, 
but the same estimate clearly remains
valid for weights (with identical arguments) so that semi-finiteness in the second part is
obvious. It remains to show normality. 

We take $x\in\cM_+$ and observe that
$$
\left\{
\begin{array}{l}
((t\phi):\psi)(x)\le(\phi:\psi)(x)\le(\phi:(t^{-1}\psi))(x)=t^{-1}((t\phi):\psi)(x)
\quad \mbox{for $0 < t <1$},
\\[2mm]
t^{-1}((t\phi):\psi)(x)=(\phi:(t^{-1}\psi))(x)\le(\phi:\psi)(x)\le((t\phi):\psi)(x)
\quad \mbox{for $t \geq 1$}.
\end{array}
\right.
$$
Hence we have either $((t\phi):\psi)(x)<+\infty$ for all $t>0$ or
$((t\phi):\psi)(x)=+\infty$ for all $t>0$. 
In the former case, since $((t\phi):\psi)(x)$ is
concave in $t>0$ (by Remark \ref{R-5.8},(ii)), it is continuous in $t>0$. 
For each $n\in\bN$ we set 
$$
t_{n,k}:=k/2^n  \quad (k=1,2,\dots,n2^n)
$$
and define
$$
f_n(t):=\sum_{k=1}^{n2^n}1_{[t_{n,k},t_{n,k+1})}(t) \,
{1+t_{n,k+1}\over t_{n,k+1}} \, ((t_{n,k}\phi):\psi)(x),\qquad t>0.
$$
We claim that
$$
\int_{(0,\infty)}f_n(t)\,d\mu(t)
=\sum_{k=1}^{n2^n}\lambda_{n,k}((t_{n,k}\phi):\psi)(x)
\quad\mbox{with}\quad\lambda_{n,k}:={1+t_{n,k+1}\over t_{n,k+1}}\,\mu([t_{n,k},t_{n,k+1}))
$$
increases to 
$$
\int_{(0,\infty)} \frac{1+t}{t}\,((t\phi):\psi)(x)\,d\mu(t)
$$ 
as $n\to\infty$. 

Indeed, from the fact noted above,  we may and do assume that $((t\phi):\psi)(x)$ is finite
(for all $t>0$). We note the following:
\begin{itemize}
\item[(a)] 
$t\mapsto((t\phi):\psi)(x)$ is a non-decreasing continuous function on $(0,\infty)$ (as was
remarked above);
\item[(b)]
$(1+t)/t$ is decreasing on  $(0,\infty)$;
\item[(c)] The $n$th function $f_n$ takes a constant value
${1+t_{n,k+1}\over t_{n,k+1}} \, ((t_{n,k}\phi):\psi) (x)$ on  each interval
$[t_{n,k}, t_{n,k+1})$ (of length $2^{-n}$), and here the values at the right and left
end-points  (of the functions in (b),\,(a) respectively) are used (to get an increasing
sequence). Also, this interval is further  divided into two subintervals 
$[t_{n,k}, t_{n,k+1})=[t_{n+1,2k}, t_{n+1,2k+1}) \cup  [t_{n+1,2k+1}, t_{n+1,2(k+1)})$
(of length $2^{-(n+1)}$) when one defines the next function $f_{n+1}$.
\end{itemize}
Based on these we can immediately show 
$$
f_n(t)\nearrow \frac{1+t}{t}\,((t\phi):\psi)(x) \quad \mbox{for all $t>0$}
$$
as $n \to \infty$.
Hence the claim is a consequence of the monotone convergence theorem. 

From the claim together with the definition \eqref{F-5.13} we observe
\begin{equation}\label{F-5.14} %%%%%%%%%%%%%%%%%%%%%%%%%%%%  \label{F-5.14} 
(\phi\sigma\psi)(x)=\sup_n\biggl[\alpha\phi(x)+\beta\psi(x)
+\sum_{k=1}^{n2^n}\lambda_{n,k}((t_{n,k}\phi):\psi)(x)\biggr],\qquad x\in\cM_+.
\end{equation}
Since the functional inside the bracket here is obviously a normal weight, so is 
$\phi\sigma\psi$ thanks to \cite[Theorem 1.8]{Haa1}.
\end{proof}

In view of the above lemma we define the following:

%%%%%%%%%%%%%%%%%%%%%%% \label{D-5.12} %%%%%%%%%%%%%%%%%%%%%
\begin{definition}\label{D-5.12}
\rm
For every $\phi,\psi\in P_0(\cM,\bC)$ we call the normal weight $\phi\sigma\psi$
(given by \eqref{F-5.13}) on $\cM$ the
\emph{connection} of $\phi,\psi$. If $\phi+\psi$ is semi-finite, then we have
$\phi\sigma\psi\in P_0(\cM,\bC)$.
\end{definition}

From the definition \eqref{F-5.13} monotonicity, concavity and transformer inequality
stated in Remark \ref{R-5.8},(ii) immediately extend to connections $\phi\sigma\psi$ for
$\phi,\psi\in P_0(\cM,\bC)$, although it is unknown to us whether or not the
transformer equality in \eqref{F-4.1} extends to $\phi,\psi\in P_0(\cM,\bC)$.
Also, the transpose equality $\phi\tilde\sigma\psi=\psi\sigma\phi$ holds as in Proposition
\ref{P-2.6}.

In the rest of this subsection we consider the situation where $\cM$ is semi-finite.

We begin with the special case $\cM=B(\cH)$ with the usual trace $\Tr$. Set $\chi$
on $\cM'=\bC1$ by $\chi(\lambda1)=\lambda$. Let $a,b$ be positive self-adjoint operators on
$\cH$. It is well-known that $\Tr(a\,\cdot)\in P_0(\cM,\bC)$ and $d\Tr(a\,\cdot)/d\chi=a$.
Hence the connection $a\sigma b$ in Definition \ref{D-2.5} via positive forms becomes a
special case of connections of semi-finite normal weights (while Definition \ref{D-5.4} is
in turn based on Definition \ref{D-2.5}).

Next we consider a general semi-finite von Neumann algebra $\cM$ with a faithful semi-finite
normal trace $\tau$. For simplicity let us assume that $\cM$ is represented in the standard form
$$
\<\pi_\ell(\cM),\,L^2(\cM,\tau),\,J=\,^*,\,L^2(\cM,\tau)_+\>
$$
($\pi_\ell$ is the left multiplication). Set $\chi=\tau'$ by $\tau'(x'):=\tau(Jx'^*J)$,
$x'\in\cM_+'$. 

Let $a,b$ be positive self-adjoint operators $a,b$ affiliated with $\cM$. It is
well-known that $\tau_a=\tau(a\,\cdot)\in P_0(\cM,\bC)$ \cite{PT} and 
\begin{equation}\label{F-5.15} %%%%%%%%%%%%%%%%%%%%%%%%%%% \label{F-5.15}
d\tau_a/d\tau'=a.
\end{equation}
Indeed, since $(D\tau_a:D\tau)_t=a^{it}$ and $d\tau/d\tau'=1$ in this case,
$$
(d\tau_a/d\tau')^{it}=(D\tau_a:D\tau)_t(d\tau/d\tau')^{it}=a^{it}
$$
(see (c) of \S\ref{S-B.3}). The next result shows that consideration on connections
of weights in this subsection  is quite consistent with what was done in \S\ref{S-3}. 
%%%%%%%%%%%%%%%%%%%%%% \label{P-5.13} %%%%%%%%%%%%%%%%%%%%%
\begin{proposition}\label{P-5.13}
Let $a,b$ be positive self-adjoint operators affiliated with $\cM$ and we assume that 
the connection $a\sigma b$ $($in the sense of Definition \ref{D-2.5}$)$ is densely defined.
Under these circumstances we have
$$
\tau_{a\sigma b}=\tau_a\sigma\tau_b.
$$
Here, the right hand side is defined in the sense of Definition \ref{D-5.12}. 
\end{proposition}

\begin{proof}
Definition \ref{D-5.4} means 
$$
d(\tau_a:\tau_b)/d\tau'=(d\tau_a/d\tau):(d\tau_b/d\tau)
=a:b=d\tau_{a:b}/d\tau'
$$
(due to \eqref{F-5.15})
so that $\tau_{a:b}=\tau_a:\tau_b$. 
Hence by the definition \eqref{F-5.13} we have for every $x\in\cM_+$,
\begin{align*}
(\tau_a\sigma\tau_b)(x)
&=\alpha\tau_a(x)+\beta\tau_b(x)+\int_{(0,\infty)}((t\tau_a):\tau_b)(x)\,d\mu(t) \\
&=\alpha\tau(ax)+\beta\tau(bx)+\int_{(0,\infty)}{1+t\over t}\,\tau((ta):b)(x)\,d\mu(t).
\end{align*}
The formula \eqref{F-5.14} applied to $\tau_a$ and $\tau_b$ shows
\begin{align}\label{F-5.16} %%%%%%%%%%%%%%%%%%%%%%%%%%% \label{F-5.16}
(\tau_a\sigma\tau_b)(x)
&=\sup_n\biggl[\alpha\tau(ax)+\beta\tau(bx)
+\sum_{k=1}^{n2^n}\lambda_{n,k}\tau((t_{n,k}a):b)x)\biggr] \nonumber\\
&=\sup_n\tau\biggl(\biggl[\alpha a+\beta b
+\sum_{k=1}^{n2^n}\lambda_{n,k}((t_{n,k}a):b)\biggr]x\biggr)=\sup_n\tau_{A_n}(x)
\end{align}
with the increasing sequence $\{A_n\}_{n \in \bN}$ given by
$$
A_n:=\alpha a+\beta b+\sum_{k=1}^{n2^n}\lambda_{n,k}((t_{n,k}a):b).
$$
Note that the connection $a\sigma b$ is defined
via the positive form
$$
q_{a\sigma b}(\xi)=\alpha q_a(\xi)+\beta q_b(\xi)+\int_{(0,\infty)}((tq_a):q_b)(\xi)\,d\mu(t),
\qquad\xi\in\cH.
$$
Almost parallel arguments to those in the proof of Lemma \ref{L-5.11} also show
\eqref{F-5.14} for positive forms and consequently we have 
$$
q_{a\sigma b}(\xi)=\sup_n\biggl[\alpha q_a(\xi)+\beta q_b(\xi)
+\sum_{k=1}^{n2^n}\lambda_{n,k}((t_{n,k}q_a):q_b)(\xi)\biggr]=\sup_nq_{A_n}(\xi),
$$
which means that $A_n\nearrow a\sigma b$ in the strong resolvent sense (i.e., as positive
forms). By \cite[Theorem 1.12,(1)]{Haa3} this implies that $\tau_{A_n}\nearrow\tau_{a\sigma b}$
so that the desired result follows from \eqref{F-5.16}.
\end{proof}

In particular, when $a,b\in\overline\cM_+$ ($\tau$-measurable), $\tau_{a \sigma b}$ with the
density $a\sigma b$ in \S\ref{S-3} agrees with the connection $\tau_a\sigma\tau_b$
in Definition \ref{D-5.12} (note that $\tau_a+\tau_b=\tau_{a+b}$ is semi-finite in this case).
Indeed, the definition of $a\sigma b$ in Definition \ref{D-3.16} is independent of the
representing Hilbert space $\cH$, and it agrees with that in the sense of Definition \ref{D-2.5}
(or Definition \ref{D-3.15}).

%%%%%%%%%%%%%%%%%%%%%%%%%%%%%%%%%%%%%%%%%%%%%%%%%
%%% Lebesgue decomposition in non-commutative $L^p$-spaces %%%%%%%%%%%%%%
%%%%%%%%%%%%%%%%%%%%%%%%%%%%%%%%%%%%%%%%%%%%%%%%%
\section{Lebesgue decomposition in non-commutative $L^p$-spaces}\label{S-6}

In this section we study Lebesgue decomposition in the setting of non-commutative 
$L^p$-spaces.
For bounded positive operators quite a satisfactory theory of Lebesgue decomposition
was worked out by Ando \cite{An1}.
For positive bounded operators $a,b$ the increasing sequence $\{ a:(nb)\}_{n \in \bN}$ 
is bounded by $a$ so that 
\begin{equation}\label{F-6.1} %%%%%%%%%%%%%%%%%%%%%%%%%% \label{F-6.1}
a[b]=\sup_n (a:(nb)) \ (\leq a)
\end{equation}
exists as the strong limit.
It was proved in \cite{An1} that
$$
a=a[b]+(a-a[b])
$$
is Lebesgue decomposition of $a$ with respect to $b$, where $a[b]$ and  $a-a[b]$
are $b$-absolutely continuous and  $b$-singular respectively (see Definition \ref{D-6.12}
for absolute continuity and singularity).
This result can be shown by (ingenious) use of basic properties of parallel sums
(see \cite{An1} for details).
The above strong limit $a[b]$ is known to be the largest $b$-absolutely continuous operator majorized 
by $a$, but uniqueness of decomposition (into $b$-absolutely continuous and $b$-singular operators) 
generally fails to hold.
Actually decomposition of each $a \geq 0$ is unique if and only $b$ has a closed range.   
These results (together with some others) were obtained in \cite{An1}

Recall that we have a reasonable notion of parallel sums (with all the expected properties)
for positive elements
in non-commutative $L^p$-spaces (see \S\ref{S-4} and \S\ref{S-5}). Thus, if $a,b$ are replaced 
by positive elements in some $L^p$-space ($p \in [1,\infty)$), 
then one can play the same game (see Remark \ref{R-6.14}). 
However, our approach in this section  (which is akin to that in \cite{Ko3})
is somewhat different. Namely, by using
relevant relative modular operators and Radon-Nikodym cocycles, we try
to express Lebesgue decomposition  in a more explicit manner.
Arguments involving relative modular operators are unavoidable so that $L^p$-spaces
consisting of powers of these operators seem fitting. 
Therefore, we start from a standard form and deal with Hilsum's $L^p$-spaces \cite{H} 
with respect to a fixed faithful state on the commutant
(instead of Haagerup's $L^p$-spaces).
These two $L^p$-spaces are isometrically isomorphic, 
and a brief description of Hilsum's $L^p$-spaces is included in Appendix \S\ref{S-B.2}
for the reader's convenience.

%%%%%%%%%%%%%%%%%%%%%%%%%%%%%%%%%%%%%%%%%%%%%%%%%%%
%%%%%%%%%%%%%%%%%%%%%%%%%%%%%%%%%%%%%%%%%%%%%%%%%%%
\subsection{Preliminaries}\label{S-6.1}
%%%%%%%%%%%%%%%%%%%%%%%%%%%%%%%%%%%%%%%%%%%%%%%%%%%
%%%%%%%%%%%%%%%%%%%%%%%%%%%%%%%%%%%%%%%%%%%%%%%%%%%

Let $\cM$ be a von Neumann algebra with a standard form
$\langle \cM, \cH, J, {\mathcal P} \rangle$.
Throughout this section we fix two faithful positive linear functionals  $\vpo, \psi$
in the predual $\cM_*^+$, whose unique implementing vectors 
in the natural cone ${\mathcal P}$ will be denoted by 
$\xi_{\varphi_0}, \xi_{\psi}$ respectively (i.e., $\vpo=\omega_{\xi_{\varphi_0}}$ and 
$\psi=\omega_{\xi_{\psi}}$). The latter functional $\psi$ (or more precisely $\psi'$ to be explained shortly)
will be needed just to define our $L^p$-spaces.
We note that $\xi_{\varphi_0}, \xi_{\psi}$ are cyclic and separating vectors 
(due to faithfulness of $\varphi_0$ and $\psi$). 
For each $\varphi\in\cM_*^+$ we have the \emph{relative modular operator} $\rmvppsi$ as well as
$\rmvpopsi$ (see the last part of Appendix \S\ref{S-B.1}).
The positive self-adjoint operator $\rmvppsi$ is exactly the \emph{spatial derivative}
$d\varphi/d\psi'$ where $\psi' \in \cM'{}_*^+$ is defined by
$$
\psi'(x')=(x'\xi_{\psi}, \xi_{\psi})=\psi(Jx'^*J) 
$$ 
for $x' \in \cM'$ (see \eqref{F-B.6}).

In this section Hilsum's $L^p$-spaces $\ellp$ ($1 \leq p <\infty)$ will be used.
We have, as explained in Appendix \S\ref{S-B.2},
$$
\ellp_+=\{\rmvppsi^{1/p}; \, \varphi \in  \cM_*^+ \}, 
$$
and $\rmvppsi^{1/p}  \, (=(d\varphi/d\psi')^{1/p}) \in \ellp_+$
corresponds to $h_{\varphi}^{1/p}$ in the Haagerup $L^p$-space $L^p(\cM)$ 
(see \eqref{F-B.7}).

We will study various relations (such as absolute continuity and so on) between 
$\rmvppsi^{1/p}$ and $\rmvpopsi^{1/p}$ in $\ellp_+$.

%%%%%%%%%%%%%%%%%%%%%% \label{D-6.1} %%%%%%%%%%%%%%%%%%%%%%
\begin{definition}\label{D-6.1}\rm
For a functional $\varphi \in \cM_*^+$ $($and a fixed faithful functional $\vpo \in \cM_*^+$$)$
we define the operator $T_{\varphi}$ with ${\mathcal D}(T_{\varphi})=\cM'\xi_{\varphi_0}$ by
$$
T_{\varphi} : \ JxJ\xi_{\varphi_0} \in \cM'\xi_{\varphi_0} \ \longmapsto
\ JxJ\rmvpvpo^{1/2p}\, \xi_{\varphi_0} \in \cH
\quad (\mbox{with $x \in \cM$}).
$$
\end{definition}

For $\varphi\in\cM_*^+$ let $(D\varphi:D\vpo)_t$ ($t\in\bR$) be \emph{Connes'
Radon-Nikodym cocycle} \cite{C2}, which is written in terms of relative modular operators as
follows (see \cite{C3}):
\begin{align}
(D\varphi:D\vpo)_t&=\Delta_{\varphi\varphi_0}^{it}\Delta_{\varphi_0}^{-it}
\ \bigl(=(d\varphi/d\varphi_0')^{it}
(d\varphi_0/d\varphi_0')^{-it}\bigr) \nonumber\\[3mm]
&=\Delta_{\varphi\psi}^{it}\rmvpopsi^{-it}
\ \bigl(=(d\varphi/d\psi')^{it}
(d\varphi_0/d\psi')^{-it}
\bigr)\qquad(t\in\bR)
\label{F-6.2} 
%%%%%%%%%%%%%%%%%%%%%%%%%%%%%%%%%%%%%%%%%%%% \label{F-6.2}
\end{align}
When $\varphi \leq \ell \vpo$ (or more generally $\rmvppsi^{1/p} \leq \ell \, \rmvpopsi^{1/p}$)
for some $\ell>0$, $(D\varphi:D\vpo)_{-i/2p}$ makes sense as an element in $\cM$.
To be more precise, there is a $\sigma$-weakly $\cM$-valued continuous function
$f(z)$ on the strip $-1/2p\le\Im z\le0$ which is analytic in the interior and satisfies
$f(t)=(D\varphi:D\vpo)_t$ ($t\in\bR$). Then $(D\varphi:D\vpo)_{-i/2p}$ is determined as $f(-i/2p)$,
and we have
$$
JxJ\Delta_{\varphi\vpo}^{1/2p\,}\xi_{\varphi_0}=JxJ(D\varphi:D\vpo)_{-i/2p}\,\xi_{\varphi_0}
=(D\varphi:D\vpo)_{-i/2p}\,JxJ\xi_{\varphi_0},
$$
showing that $T_{\varphi}$ is just the restriction of $(D\varphi:D\vpo)_{-i/2p}$ 
to the dense subspace $\cM'\xi_{\varphi_0}$ in this situation.

We will have to deal with increasing sequences in $\ellp_+$ quite often.
For an increasing sequence $\{A_n\}_{n\in\bN}$ and $(A_n \leq)\, A$ in $\ellp_+$ the notation
$A_n \nearrow A$ means $A_n \rightarrow A$ in the  $\sigma (L^p,L^q)$-topology (see \eqref{F-B.8}),
i.e., $\langle A_n, B\rangle \nearrow \langle A, B\rangle$ for each $B \in \ellq_+$ where
$1/p+1/q=1$. Note that we have $\|A-A_n\|_p \rightarrow 0$ in this case. In fact, for $p=1$ 
the $L^1$-space $L^1(\cM,\psi')$ can be identified with the predual $\cM_*$ and 
we have
$\|A-A_n\|_1=\langle A-A_n, 1\rangle \rightarrow 0$. On the other hand, for $p \in (1,\infty)$ 
the assertion follows from the well-known uniform convexity of $L^p$-spaces 
(see \cite[\S 5]{FK} for instance) together with
$\lim_{n \to \infty} \|A_n\|_p=\|A\|_p$, which is a consequence of
$$
\left\{
\begin{array}{l}
\|A_n\|_p \leq \|A\|_p \quad (\mbox{due to $A_n \leq A$}),\\[3mm]
\|A\|_p \leq \liminf_{n \to \infty}\|A_n\|_p
\quad(\mbox{due to lower semi-continuity of $\|\cdot\|_p$ in $\sigma (L^p,L^q)$}).
\end{array}
\right.
$$

The next proposition as well as its proof will be of fundamental importance in the rest of the section.
Indeed, the property (ii) below will be used as the definition of $\rmvpopsi^{1/p}$-absolute continuity
(for $\rmvppsi^{1/p}$) in \S\ref{S-6.4} (see Definition \ref{D-6.12}) 
so that this proposition is a characterization result for absolute continuity.  
We point out that another criterion (in terms of a suitable Radon-Nikodym cocycle) will be added in 
Remark \ref{R-6.9},(ii). 
We will repeatedly use arguments in the proof below and their variants (in \S\ref{S-6.3}). 

%%%%%%%%%%%%%%%%%%%%%%%% \label{P-6.2} %%%%%%%%%%%%%%%%%%%%
\begin{proposition}\label{P-6.2}
We assume $1 \leq p < \infty$ and $1/p+1/q=1$. For $\varphi \in \cM_*^+$ the following three properties 
are mutually equivalent$:$
\begin{itemize}
\item[\rm(i)] 
the operator $T_{\varphi}$ $($Definition \ref{D-6.1}$)$ is closable$;$
\item[\rm(ii)] 
there exists an increasing sequence $\{A_n\}$ in $\ellp_+$ satisfying
$$
A_n \nearrow \rmvppsi^{1/p} \quad
\mbox{and}
\quad
A_n \leq \ell_n \rmvpopsi^{1/p} \ \mbox{for some $\ell_n > 0$};
$$
\item[\rm(iii)]
the map
$$
x\xi_{\varphi_0} \in \cM\xi_{\varphi_0} \  \longmapsto \  
\langle  \rmvppsi^{1/p}, \rmvpopsi^{1/2q}x^*x\rmvpopsi^{1/2q} \rangle \in \bR_+ 
$$
is lower semi-continuous.
\end{itemize}
\end{proposition}
\begin{proof}
We will show (i)$\implies$(ii), (ii)$\implies$(iii), and (iii)$\implies$(i).
Firstly we assume (i).  From the definition $T_{\varphi}$ satisfies $u'T_{\varphi}u'^* =T_{\varphi}$
for an arbitrary unitary $u' \in \cM'$ so that the closure $\overline{T_{\varphi}}$ satisfies 
$u'\overline{T_{\varphi}}u'^* =\overline{T_{\varphi}}$, 
that is, $\overline{T_{\varphi}}$ is affiliated with $\cM$ (see also Remark \ref{R-6.9}).
Let
$$
T_{\varphi}^*\overline{T_{\varphi}}=\int_0^{\infty}\, \lambda\, de_{\lambda}
$$
be the spectral decomposition. We set
$$
h_n:=\int_0^n \, \lambda \,de_{\lambda} \in \cM_+ 
\quad \mbox{and} \quad 
A_n:=\rmvpopsi^{1/2p}h_n\rmvpopsi^{1/2p} \in \ellp_+.
$$
Clearly $\{A_n\}$ is an increasing sequence in $\ellp_+$ and $A_n \leq n\rmvpopsi^{1/p}$.
For each $x \in \cM$ we compute
\begin{eqnarray}\label{F-6.3} %%%%%%%%%%%%%%%%%%%%%%%%%%%%%% \label{F-6.3}
&&
\langle A_n, \rmvpopsi^{1/2q} x^*x \rmvpopsi^{1/2q} \rangle
=\mbox{tr}\left(\rmvpopsi^{1/2p}h_n\rmvpopsi^{1/2p}\rmvpopsi^{1/2q} x^*x \rmvpopsi^{1/2q}\right)
\nonumber\\
&&
\qquad
=\mbox{tr}\left( h_n\rmvpopsi^{1/2} x^*x \rmvpopsi^{1/2} \right)
=\mbox{tr}\left(\left(h_n\rmvpopsi^{1/2}\right) \left(\rmvpopsi^{1/2}x^*x\right)^* \right)
\nonumber\\
&&
\qquad
=(h_n\xi_{\varphi_0},Jx^*xJ\xi_{\varphi_0})
=(h_nJxJ\xi_{\varphi_0},JxJ\xi_{\varphi_0}).
\end{eqnarray}
Here we have used the fact that $\rmvpopsi^{1/2} \in L^2(\cM,\psi')_+$ corresponds 
to $\xi_{\varphi_0} \in {\mathcal P}$.
On the other hand, we have
\begin{equation}\label{F-6.4} %%%%%%%%%%%%%%%%%%%%%%%%%%%%%%% \label{F-6.4}
\|\overline{T_{\varphi}}JxJ\xi_{\varphi_0}\|^2
=\|T_{\varphi}JxJ\xi_{\varphi_0}\|^2=\|JxJ\Delta_{\varphi\vpo}^{1/2p}\,\xi_{\varphi_0}\|^2.
\end{equation}
The vectors $\Delta_{\varphi\vpo}^{1/2p}\xi_{\varphi_0}$, $JxJ\Delta_{\varphi\vpo}^{1/2p}\xi_{\varphi_0}$
here correspond to $h_{\varphi}^{1/2p}h_{\vpo}^{1/2q}$,  $h_{\varphi}^{1/2p}h_{\vpo}^{1/2q}x^*$
respectively 
in the Haargerup $L^2$-space $L^2(\cM)$ (see \cite[\S 2]{Ko9} and \cite[\S 1]{Ko1})
and hence to $\rmvppsi^{1/2p} \rmvpopsi^{1/2q}$, $\rmvppsi^{1/2p} \rmvpopsi^{1/2q} x^*$
respectively in our $L^2(\cM, \psi')$.
Hence, the far right side of (\ref{F-6.4}) is equal to
\begin{eqnarray}\label{F-6.5} %%%%%%%%%%%%%%%%%%%%%%%%%%%%%%% \label{F-6.5} 
&&
\mbox{tr}\left( \left(\rmvppsi^{1/2p}\rmvpopsi^{1/2q}x^*\right)
\left(\rmvppsi^{1/2p}\rmvpopsi^{1/2q}x^*\right)^*  \right)
=\mbox{tr}\left( \rmvppsi^{1/2p}\rmvpopsi^{1/2q}x^*x\rmvpopsi^{1/2q}\rmvppsi^{1/2p}  \right)
\nonumber\\
&&
\qquad \qquad
=\mbox{tr}\left( \rmvppsi^{1/p}\rmvpopsi^{1/2q}x^*x\rmvpopsi^{1/2q}  \right)
=\langle  \rmvppsi^{1/p}, \rmvpopsi^{1/2q}x^*x\rmvpopsi^{1/2q}  \rangle
\end{eqnarray}
in our $L^p$-$L^q$-duality notation (see \eqref{F-B.8}). Thus, we have shown
\begin{equation}\label{F-6.6} %%%%%%%%%%%%%%%%%%%%%%%%%%%%%%% \label{F-6.6}
\|\overline{T_{\varphi}}JxJ\xi_{\varphi_0}\|^2
=\langle  \rmvppsi^{1/p}, \rmvpopsi^{1/2q}x^*x\rmvpopsi^{1/2q}  \rangle.
\end{equation}
Therefore, from (\ref{F-6.3}) and (\ref{F-6.6}) we observe 
\begin{eqnarray*}
&&\langle A_n, \rmvpopsi^{1/2q} x^*x \rmvpopsi^{1/2q} \rangle  
\leq \langle  \rmvppsi^{1/p}, \rmvpopsi^{1/2q}x^*x\rmvpopsi^{1/2q} \rangle,
\nonumber\\
&&
\langle A_n, \rmvpopsi^{1/2q} x^*x \rmvpopsi^{1/2q} \rangle  
\nearrow \langle  \rmvppsi^{1/p}, \rmvpopsi^{1/2q}x^*x\rmvpopsi^{1/2q} \rangle
\quad (\mbox{as $n \to \infty$}).
\end{eqnarray*}
The set of all elements of the form $\rmvpopsi^{1/2q} x^*x \rmvpopsi^{1/2q}$ is a dense subset in $\ellq_+$
(as will be seen in Lemma \ref{L-6.3})
so that we have $A_n \leq \rmvppsi^{1/p}$ and
$\langle A_n, C\rangle  \nearrow\langle  \rmvppsi^{1/p}, C \rangle$ for each $C \in \ellq_+$, i.e.,
$A_n \nearrow \rmvppsi^{1/p}$. Thus, (ii) is shown.

We next assume (ii). 
The assumption $A_n \leq \ell_n\rmvpopsi^{1/p}$ guarantees $A_n^{1/2}=u_n\rmvpopsi^{1/2p}$ 
for some $u_n \in \cM$ with $\|u_n\| \leq \sqrt{\ell_n}$
so that we get 
$$
\langle A_n, \rmvpopsi^{1/2q} x^*x \rmvpopsi^{1/2q} \rangle
= \langle \rmvpopsi^{1/2p} u_n^*u_n\rmvpopsi^{1/2p},
\rmvpopsi^{1/2q} x^*x \rmvpopsi^{1/2q} \rangle
=\|u_nJxJ\xi_{\varphi_0} \|^2,
$$
where the second equality follows from the same computations as (\ref{F-6.3}).
Since $A_n \nearrow \rmvppsi^{1/p}$,  we have
$$
\langle \rmvppsi^{1/p}, \rmvpopsi^{1/2q} x^*x \rmvpopsi^{1/2q} \rangle
=
\sup_n \langle A_n, \rmvpopsi^{1/2q} x^*x \rmvpopsi^{1/2q} \rangle 
= \sup_n\|u_nJxJ\xi_{\varphi_0} \|^2.
$$
Being the supremum of continuous functions
$x\xi_{\varphi_0} \mapsto (u_n^*u_nJx\xi_{\varphi_0},Jx\xi_{\varphi_0})$
($n=1,2,\cdots$), the above quantity is lower semi-continuous and (iii) is shown.

Finally let us assume (iii). 
In the very first part of the proof  the equation \eqref{F-6.6} was shown from (\ref{F-6.4}) and (\ref{F-6.5}), 
but obviously (\ref{F-6.4}), (\ref{F-6.5}) always yield
$$
\|T_{\varphi}JxJ\xi_{\varphi_0}\|^2
=\langle  \rmvppsi^{1/p}, \rmvpopsi^{1/2q}x^*x\rmvpopsi^{1/2q}  \rangle
$$
(since the second equality in (\ref{F-6.4}) is just the definition of $T_{\varphi}$ in Definition \ref{D-6.1}).
This shows the quadratic form associated with $T_{\varphi}$ defined on
${\mathcal D}(T_{\varphi})=\cM'\xi_{\varphi_0}$ 
(i.e., the above left side) is lower semi-continuous from the assumption (iii),
and hence $T_{\varphi}$ is closable as desired. 
Indeed, let us assume that a sequence $\eta_n$ in ${\mathcal D}(T_{\varphi})$
 satisfies $\eta_n \to 0$ and $T_{\varphi}\eta_n \to \zeta$. 
Since $\{T_{\varphi}\eta_n\}$ is a Cauchy sequence,
for each $\varepsilon>0$ we have $N=N_{\varepsilon}$ 
such that $\|T_{\varphi}\eta_n-T_{\varphi}\eta_m\| \leq \varepsilon$
for $n,m \geq N$. We fix $n \geq N$ and let $m \to \infty$. 
Since $\eta_n-\eta_m \to \eta_n$, lower semi-continuity shows
$$
\|T_{\varphi}\eta_n\| \leq \liminf_{m \to \infty}\|T_{\varphi}(\eta_n-\eta_m)\| 
\leq \varepsilon \quad (\mbox{for $n \geq N$}).
$$
Thus, we have $\lim_{n \to \infty}\|T_{\varphi}\eta_n\|=0$, i.e., $\zeta=0$, and hence 
$\overline{\Gamma(T_{\varphi})}$ meets the ``$y$-axis" $0\oplus\cH$ in the trivial way.
\end{proof}

The next result (required in the above proof) is certainly known to specialists. 
However, we present a proof for the reader's convenience.
%%%%%%%%%%%%%%%%%%%%%% \label{L-6.3} %%%%%%%%%%%%%%%%%%%%%%
\begin{lemma}\label{L-6.3}
The subset $\rmvpopsi^{1/2q}\cM_+\rmvpopsi^{1/2q}  \, \left( \subseteq \ellq_+ \right)$
is dense in $\ellq_+$.
\end{lemma}
\begin{proof}
We may and do assume $1 \leq q  < \infty$.
We note that 
$$
{\mathcal A}:=\{\omega \in \cM_*^+;\, \mbox{$\omega \leq \ell\, \vpo$ for some $\ell>0$} \}
$$
is dense in  $\cM_*^+$.  In fact, $\omega_{x'\xi_{\varphi_0}} =(\cdot\,x'\xi_{\varphi_0},x'\xi_{\varphi_0})$
with $x' \in \cM'$ belongs to ${\mathcal A}$ due to $\omega_{x'\xi_{\varphi_0}} \leq \|x'\|^2\vpo$.
Thus, the density follows from the following two facts:
\begin{itemize}
\item[(i)] 
$\cM' \xi_{\varphi_0}$ is dense in $\cH$,
\item[(ii)] 
$\|\omega_{\xi_1}-\omega_{\xi_2}\| \leq \|\xi_1+\xi_2\|\cdot\|\xi_1-\xi_2\|$ for $\xi_i \in \cH$.
\end{itemize}

Then, we observe that the set of all
$\Delta_{\omega\psi}^{1/q}, \, \omega \in {\mathcal A}$, is dense in $\ellq_+$ 
(since $\omega \in \cM_*^+ \mapsto \Delta_{\omega\psi}^{1/q} \in \ellq_+$ is a continuous surjection
due to the generalized Powers-S\o rmer inequality \cite[Appendix]{HN} for instance). 
For $\omega \in {\mathcal A}$ the Radon-Nikodym cocycle 
$(D\omega:D\vpo)_t=\Delta_{\omega\psi}^{it}\rmvpopsi^{-it}$ (see \eqref{F-6.2}) extends
to a bounded continuous function on the strip $-1/2 \leq \Im z \leq 0$ which is analytic in the interior.
Therefore, with $x=(D\omega:D\vpo)_{-i/2q}$ 
we have 
$$
\Delta_{\omega \psi}^{1/q}=\rmvpopsi^{1/2q}x^*x\rmvpopsi^{1/2q} \in \rmvpopsi^{1/2q}\cM_+\rmvpopsi^{1/2q}
$$
thanks to $\Delta_{\omega\psi}^{1/2q}=x\Delta_{\ffi_0\psi}^{1/2q}$, and we are done.
\end{proof}

%%%%%%%%%%%%%%%%%%%%%%%%%%%%%%%%%%%%%%%%%%%%%%%%%%%
%%%%%%%%%%%%%%%%%%%%%%%%%%%%%%%%%%%%%%%%%%%%%%%%%%%
\subsection{Graph analysis}\label{S-6.2}
%%%%%%%%%%%%%%%%%%%%%%%%%%%%%%%%%%%%%%%%%%%%%%%%%%%
%%%%%%%%%%%%%%%%%%%%%%%%%%%%%%%%%%%%%%%%%%%%%%%%%%%

To investigate closability of $T_{\varphi}$ (mentioned in Proposition \ref{P-6.2}), 
it is natural to see the graph of this operator. It requires some manipulations of
relevant relative modular operators and Radon-Nikodym cocycles. 

Let $\sigma_t \, (=\sigma_t^{\vpo})$ be the modular automorphism group associated with $\vpo$, 
and we set
$$
\cM_0:=\{x \in \cM;\, \mbox{$t \in \bR \mapsto \sigma_t(x) \in \cM$ extends to an entire function}\}.
$$
The modular operator associated with $\vpo'=\vpo\circ j \in \cM'^+_*$ 
(with $j(x')=Jx'^*J$ for $x' \in \cM'$)  is $J\Delta J=\Delta^{-1}$. Therefore, $J\cM_0 J$ consists of
all elements $x' \in \cM'$ for which $\sigma'_t(x') \ \left(=\sigma^{\vpo'}_t(x')\right)$ extend to
entire functions.
We note $J\cM_0 J\xi_{\varphi_0}=\cM_0\xi_{\varphi_0}$.

Let us begin with the next general result on a core for relative modular operators.  
%%%%%%%%%%%%%%%%%%%%% \label{L-6.4} %%%%%%%%%%%%%%%%%%%%%%
\begin{lemma}\label{L-6.4}
For each $\omega \in \cM_*^+$ $($and a fixed faithful functional $\varphi_0 \in\cM_*^+ $$)$
the subspace $\cM_0 \xi_{\varphi_0}$ is a core for $\Delta_{\omega \vpo}^{\alpha}$ with $\alpha \in [0,1/2]$.
\end{lemma}
\begin{proof}
Obviously ${\mathcal D}(\Delta_{\omega \vpo}^{1/2}) \ \left(\subseteq  {\mathcal D}(\Delta_{\omega \vpo}^{\alpha}) \right)$
is a core for $\Delta_{\omega \vpo}^{\alpha}$, that is,  
${\mathcal D}(\Delta_{\omega \vpo}^{1/2})$ is dense in ${\mathcal D}(\Delta_{\omega \vpo}^{\alpha})$
with respect to the graph norm of  $\Delta_{\omega \vpo}^{\alpha}$. 
Thus, it suffices to show that $\cM_0 \xi_{\varphi_0}$
is dense in ${\mathcal D}(\Delta_{\omega \vpo}^{1/2})$ 
with respect to the graph norm of  $\Delta_{\omega \vpo}^{\alpha}$,
which follows from the following two facts:
\begin{itemize}
\item[(i)] 
$\cM_0 \xi_{\varphi_0}$ is a core for $\Delta_{\omega \vpo}^{1/2}$;
\item[(ii)]
the graph norm of $\Delta_{\omega \vpo}^{\alpha}$ is majorized by
a scalar multiple of that of $\Delta_{\omega \vpo}^{1/2}$ on ${\mathcal D}(\Delta_{\omega \vpo}^{1/2})$.
\end{itemize}
From the definition $\cM \xi_{\varphi_0}$ is a core for $\Delta_{\omega \vpo}^{1/2}$. 
We note that $\cM_0$ is a dense *-subalgebra of $\cM$ in the $\sigma$-weak topology
and hence in the strong* topology (thanks to the Kaplansky density theorem), that is,
for $x \in \cM$ we can take $y \in \cM_0$ with arbitrarily small
$\|(x-y)\xi_{\varphi_0}\|$ and $\|(x^*-y^*)\xi_{\omega}\|
\ \left(=\|\Delta_{\omega \vpo}^{1/2}(x-y)\xi_{\varphi_0}\|\right)$.  
Therefore, (i) holds true. 
On the other hand, (ii) is a consequence
of the following obvious estimate for $\eta \in{\mathcal D}(\Delta_{\omega \vpo}^{1/2})$:
\begin{align*}
\|\eta\|^2+\|\Delta_{\omega \vpo}^{\alpha}\eta\|^2
&=\int_0^{\infty}\,(1+\lambda^{2\alpha})\,d\|e_{\lambda}\eta\|^2\\
&=\int_0^{1}\,(1+\lambda^{2\alpha})\,d\|e_{\lambda}\eta\|^2+
\int_1^{\infty}\,(1+\lambda^{2\alpha})\,d\|e_{\lambda}\eta\|^2\\
&\leq\int_0^{1}\,2\,d\|e_{\lambda}\eta\|^2+\int_1^{\infty}\,(1+\lambda)\,d\|e_{\lambda}\eta\|^2\\
&\leq 3\int_0^{\infty}\,(1+\lambda)\,d\|e_{\lambda}\eta\|^2
=3\left(\|\eta\|^2+\|\Delta_{\omega \vpo}^{1/2}\eta\|^2\right),
\end{align*}
where $\Delta_{\omega \vpo}=\int_0^{\infty}\,\lambda\,de_{\lambda}$ is the spectral decomposition.
\end{proof}

We note 
$\Delta_{\varphi\vpo}^{1/p}+\Delta_{\varphi_0}^{1/p} \in L^p(\cM,\vpo')_+$
and it must be of the form
\begin{equation}\label{F-6.7} %%%%%%%%%%%%%%%%%%%%%%%%%%%%%% \label{F-6.7}
\Delta_{\varphi\vpo}^{1/p}+\Delta_{\varphi_0}^{1/p} = \Delta_{\chi\vpo}^{1/p}
\end{equation}
with some $\chi \in \cM_*^+$. We observe that $\Delta_{\chi\vpo}^{1/p}$ is non-singular 
(since so is  $\Delta_{\varphi_0}^{1/p}$), showing that $\chi$ is faithful.
When $p=1$, we simply have $\chi=\varphi+\varphi_0$ (the usual sum as functionals), 
which is exactly the situation dealt with in \cite{Ko3}.
When $p=2$, $\chi$ is the functional coming from the vector sum
$\xi_{\varphi}+\xi_{\varphi_0} \in {\mathcal P}$, i.e., 
$$
\chi(x)=\left(x\,(\xi_{\varphi}+\xi_{\varphi_0}), (\xi_{\varphi}+\xi_{\varphi_0})\right).
$$
We now go back to a general $p$, and for convenience we set
\begin{equation}\label{F-6.8} %%%%%%%%%%%%%%%%%%%%%%%%%%%%%% \label{F-6.8}
\zeta_{\chi}:=\Delta_{\chi \vpo}^{1/2p}\,\xi_{\varphi_0}.
\end{equation}
Majorization $\Delta_{\varphi_0}^{1/p}, \, \Delta_{\varphi\vpo}^{1/p} \leq \Delta_{\chi\vpo}^{1/p}$
guarantees that 
both of the Radon-Nikodym cocycles $(D\vpo:D\chi)_t$ and $(D\varphi:D\chi)_t$
admit bounded continuous extensions on the strip $-1/2p \leq \Im z \leq 0$ 
which are analytic in the interior.

%%%%%%%%%%%%%%%%%%%%% \label{D-6.5} %%%%%%%%%%%%%%%%%%%%%%%
\begin{definition}\label{D-6.5}\rm
With $\chi \in \cM_*^+$ determined by \eqref{F-6.7} we set
$$
a=(D\vpo:D\chi)_{-i/2p}
\quad \mbox{and} \quad
b=(D\varphi:D\chi)_{-i/2p},
$$
which are contractions in $\cM$.
\end{definition}

%%%%%%%%%%%%%%%%%%%%%%% \label{L-6.6} %%%%%%%%%%%%%%%%%%%%%
\begin{lemma}\label{L-6.6}
We have $\xi_{\varphi_0}=a\zeta_{\chi}$ and $\Delta_{\varphi \vpo}^{1/2p}\,\xi_{\varphi_0}=b\zeta_{\chi}$.
\end{lemma}
\begin{proof}
We note
\begin{eqnarray*}
&&
(D\vpo:D\chi)_t\Delta_{\chi \vpo}^{it}\xi_{\varphi_0}
=\Delta_{\varphi_0}^{it}\Delta_{\chi \vpo}^{-it}\Delta_{\chi \vpo}^{it} \xi_{\varphi_0}=\Delta_{\varphi_0}^{it}\xi_{\varphi_0}=\xi_{\varphi_0},\\
&&
(D\varphi:D\chi)_t\Delta_{\chi \vpo}^{it}\xi_{\varphi_0}
=\Delta_{\varphi \vpo}^{it}\Delta_{\chi \vpo}^{-it}\Delta_{\chi \vpo}^{it}\xi_{\varphi_0}
=\Delta_{\varphi \vpo}^{it}\xi_{\varphi_0}.
\end{eqnarray*}
We can substitute $t=-i/2p$ here, which gives us the desired conclusion.
\end{proof}

%%%%%%%%%%%%%%%%%%%%%%% \label{T-6.7} %%%%%%%%%%%%%%%%%%%%%
\begin{theorem}\label{T-6.7}
The closure of the graph $\Gamma(T_{\varphi}) \ \left(\subseteq \cH \oplus \cH \right)$ 
of the operator $T_{\varphi}$ $($see Definition \ref{D-6.1}$)$ is given by
$$
\overline{\Gamma(T_{\varphi})}=\{(a\xi,b\xi) \in \cH \oplus \cH ;\, \xi \in \cH\}
$$
$($with the contractions $a, b$ in Definition \ref{D-6.5}$)$.
\end{theorem}
\begin{proof}
For $x' \in \cM'$ Lemma \ref{L-6.6} yields 
$$
x'\xi_{\varphi_0}=x'a\zeta_{\chi}=ax'\zeta_{\chi} 
\quad
\mbox{and}
\quad 
x'\rmvpvpo^{1/2p}\,\xi_{\varphi_0}=x'b\zeta_{\chi}=bx'\zeta_{\chi},
$$
and hence Definition \ref{D-6.1} says
$$
\Gamma(T_{\varphi})=\{(ax'\zeta_{\chi}, bx'\zeta_{\chi}); \, x' \in \cM'\}.
$$

It is important to notice
\begin{equation}\label{F-6.9} %%%%%%%%%%%%%%%%%%%%%%%%%%%%%% \label{F-6.9}
\overline{J\cM_0 J\zeta_{\chi}}=\cH.
\end{equation}
Indeed, we observe
$$
\overline{J\cM_0 J \zeta_{\chi}}=\overline{J\cM_0 J \Delta_{\chi \vpo}^{1/2p}\, \xi_{\varphi_0}}
=\overline{\Delta_{\chi \vpo}^{1/2p}J\cM_0 J\xi_{\varphi_0}}
=\overline{{\mathcal R}(\Delta_{\chi \vpo}^{1/2p})}=\cH.
$$
Here, the first and fourth equalities follow from (\ref{F-6.8}) and faithfulness of $\chi$ respectively. 
The third equality holds true due to the fact that 
$\cM_0\xi_{\varphi_0}=J\cM_0 J\xi_{\varphi_0}$ is a core for $\Delta_{\chi \vpo}^{1/2p}$ (thanks to Lemma \ref{L-6.4})
while the second is a consequence of
$(-1/2p)$-homogeneity of $\Delta_{\chi \vpo}^{1/2p}$ relative to $\varphi'_0$ 
(see \cite[Chap.~III, Corollary 34]{Te1} or Lemma \ref{L-C.2} in \S\ref{S-C}),  
that is, 
\begin{equation}\label{F-6.10} %%%%%%%%%%%%%%%%%%%%%%%%%%%%%% \label{F-6.10}
x'\Delta_{\chi\vpo}^{1/2p} \subseteq \Delta_{\chi\vpo}^{1/2p}\sigma'_{-i/2p}(x')
\end{equation}
holds true for  each $x' \in J\cM_0 J$.
Since $ \cM' \zeta_{\chi}\,  (\supseteq J\cM_0 J\zeta_{\chi})$ is dense in $\cH$, we have
$$
\overline{\Gamma(T_{\varphi})} \supseteq
\{(a\xi, b\xi); \, \xi \in \cH\}
\supseteq \Gamma(T_{\varphi}),
$$
and it remains to show that the set $\{(a\xi, b\xi); \, \xi \in \cH\}$ is closed.

To show this closedness, we note
\begin{equation}\label{F-6.11} %%%%%%%%%%%%%%%%%%%%%%%%%%%%% \label{F-6.11}
a^*a+b^*b=1,
\end{equation}
which is equivalent to
$$
((a^*a+b^*b)x'\zeta_{\chi}, x'\zeta_{\chi})=(x'\zeta_{\chi},x'\zeta_{\chi})
\quad (x' \in J\cM_0 J)
$$
thanks to (\ref{F-6.9}). The left hand side here is equal to
\begin{eqnarray*}
&&
\|ax'\zeta_{\chi}\|^2+\|bx'\zeta_{\chi}\|^2
=\|x'a\zeta_{\chi}\|^2+\|x'b\zeta_{\chi}\|^2\\
&&
\qquad 
=\|x'\xi_{\varphi_0}\|^2+\|x'\Delta_{\varphi \vpo}^{1/2p}\,\xi_{\varphi_0}\|^2
\quad (\mbox{by Lemma \ref{L-6.6}})\\
&&
\qquad
=\|\Delta_{\varphi_0}^{1/2p}\sigma'_{-i/2p}(x')\,\xi_{\varphi_0}\|^2
+\|\Delta_{\varphi\vpo}^{1/2p}\sigma'_{-i/2p}(x')\,\xi_{\varphi_0}\|^2\\
&&
\qquad
=
\|\Delta_{\chi\vpo}^{1/2p}\sigma'_{-i/2p}(x')\,\xi_{\varphi_0}\|^2
\quad (\mbox{by (\ref{F-6.7})})\\
&&
\qquad
=\|x'\Delta_{\chi\vpo}^{1/2p}\,\xi_{\varphi_0}\|^2=\|x'\zeta_{\chi}\|^2
\end{eqnarray*}
thanks to
 $(-1/2p)$-homogeneity of $\Delta_{\varphi \vpo}^{1/2p}$ and $ \Delta_{\chi\vpo}^{1/2p}$
(see (\ref{F-6.10}))
so that (\ref{F-6.11}) has been proved.

To see the closedness in question, let us assume
that $(\xi_1,\xi_2)$ is in the closure of $\{(a\xi, b\xi); \, \xi \in \cH\}$.
This means that there is a sequence $\{\eta_n\}$ such that
$$
\xi_1=\lim_{n \to \infty} a\eta_n 
\quad 
\mbox{and}
\quad
\xi_2=\lim_{n \to \infty} b\eta_n.
$$
Thanks to (\ref{F-6.11}) we have
$$ 
\eta_n=a^*a\eta_n+b^*b\eta_n \longrightarrow a^*\xi_1+b^*\xi_2 \quad (\mbox{as $n \to \infty$}).
$$
Therefore, with $\eta=a^*\xi_1+b^*\xi_2$
we get $(\xi_1,\xi_2)=(a\eta,b\eta) \in \{(a\xi, b\xi); \, \xi \in \cH\}$, and we are done.
\end{proof}
%%%%%%%%%%%%%%%%%%%%% \label{C-6.8} %%%%%%%%%%%%%%%%%%%%%%%
\begin{corollary}\label{C-6.8}
The projection $P$ from $\cH \oplus \cH$ onto $\overline{\Gamma(T_{\varphi})}$ is given by
$$
P=
\begin{bmatrix}
aa^*&ab^*\\
ba^*&bb^*
\end{bmatrix}
$$
$($with the contractions $a, b$ in Definition \ref{D-6.5}$)$.
\end{corollary}
\begin{proof}
Thanks to \eqref{F-6.11}
$P$ is indeed a projection and 
$
P
\begin{bmatrix}
a\xi\\
b\xi
\end{bmatrix}
=
\begin{bmatrix}
a\xi\\
b\xi
\end{bmatrix}
$
holds true.
On the other hand, Theorem \ref{T-6.7} shows
$$
P
\begin{bmatrix}
\xi_1\\
\xi_2
\end{bmatrix}
=
\begin{bmatrix}
a(a^*\xi_1+b^*\xi_2)\\
b(a^*\xi_1+b^*\xi_2)
\end{bmatrix}
\in \overline{\Gamma(T_{\varphi})}
$$
for $\xi,\xi_i \in \cH$.
\end{proof}

%%%%%%%%%%%%%%%%%%%%%%% \label{R-6.9} %%%%%%%%%%%%%%%%%%%%%
\begin{remark}\label{R-6.9}\rm
\mbox{}
\begin{itemize}
\item[(i)]
Obstruction for closability of $T_{\varphi}$ is 
$$
\overline{\Gamma(T_{\varphi})} \cap (0 \oplus \cH) \cong
\{b\xi; \, a\xi=0 \}
=b\ker a.
$$
The polar decompositions of $a, b$ $($in Definition \ref{D-6.5}$)$ are of the forms
$$
a=v(1-h^2)^{1/2}, \quad b=uh
$$
with a positive contraction $h$ $($because of \eqref{F-6.11}$)$.
Thus, we observe
$$
\ker a=\ker (1-h^2)^{1/2}= \ker (1-h)=\{\xi \in \cH;\, \xi=h\xi\}
$$ 
and
$b\ker (1-h)=u\ker (1-h)$.
Since $\ker(1-h) \subseteq \overline{{\mathcal R}(h)}$,
the initial space of the partial isomery $u$, the projection $\tilde{q}$ onto $b\ker a$ is given by
$$
\tilde{q}=u\tilde{q}_0u^*
$$ 
with the projection $\tilde{q}_0$ onto $\ker(1-h)=\ker a$. 

\item[(ii)]
In particular, $T_{\varphi}$ is closable if 
and only if $\ker a=0$ $($see Definition \ref{D-6.1} and Proposition \ref{P-6.2}$)$.
Here, 
$$
a=(D\vpo:D\chi)_{-i/2p}
$$
is one of the contractions in Definition \ref{D-6.5}.

\item[(iii)]
We set 
\begin{equation}\label{F-6.12} %%%%%%%%%%%%%%%%%%%%%%%%%%%% \label{F-6.12}
\tp=1-\tilde{q}. 
\end{equation}
Then $(T_{\varphi})_{c}=\tp\, T_{\varphi}$ $($the obstruction $\tilde{q}$ for
closability is removed and $(T_{\varphi})_{c}$ is closable$)$  is known as the \emph{closable part} 
or the operator part of $T_{\varphi}$ $($see \cite{J}$)$. We note
\begin{eqnarray*}
&&
\tilde{q}ba^*=u\tilde{q}_0u^*uh(1-h^2)^{1/2}v^*=u\tilde{q}_0h(1-h^2)^{1/2}v^*
=uh(1-h^2)^{1/2}\tilde{q}_0v^*=0,\\
&&
\tilde qbb^*=u\tilde q_0u^*uh^2u^*=u\tilde q_0h^2u^*=uh^2\tilde q_0u^*=bb^*\tilde q,
\end{eqnarray*}
showing 
$$
\tilde{p}ba^*=ba^* \quad \mbox{and} \quad \tilde{p}bb^*=bb^*\tilde{p}.
$$
Hence the projection onto the graph of $\overline{(T_{\varphi})_{c}}$ $($i.e., the characteristic
matrix of  $\overline{(T_{\varphi})_{c}}$$)$ is given by
$$
\begin{bmatrix}
aa^*&ab^*\\
\tp\, ba^*&\tp\, bb^*
\end{bmatrix}
=
\begin{bmatrix}
aa^*&ab^*\\
ba^*&\tp\, bb^*
\end{bmatrix}.
$$
This matrix obviously commutes with $u' \otimes 1_{M_2({\bold C})}$ $($with a unitary $u' \in \cM'$$)$,
showing that $\overline{(T_{\varphi})_{c}}$ is always affiliated with $\cM$.
\end{itemize}
\end{remark}
So far we have presented direct self-contained arguments (based on Corollary \ref{C-6.8})
in our special situation.
A general theory on the (maximal) closable part of an operator was actually developed in \cite{J}, and
it goes as follows: 
It is well-known that a densely defined operator $T$ is closable if and only if the adjoint $T^*$ has
a dense domain. Let
$$
\begin{bmatrix}
p_{11}&p_{12}\\
p_{21}&p_{22}
\end{bmatrix}
$$
be the projection onto the closure $\overline{\Gamma(T)}$. Then, we have $\cD(T^*)^{\perp}=\ker(1-p_{22})$
(see \cite[Theorem 3.1,(d)]{J}), which corresponds to obstruction for closablity. 
In our special case $p_{22}$ is $bb^*$ and hence the projection onto $\ker(1-p_{22})$ is nothing but 
$\tilde{q} \, (= u\tilde{q}_0u^*)$ (in Remark \ref{R-6.9},(ii)), i.e., $\tilde{p}=1-\tilde{q}$ is the projection 
onto $\overline{\cD(T^*)}$.
In \cite{J}  $T_{c}=\tilde{p}\,T$ is defined as the closable part of $T$.  
It was shown there that $T_{c}$ is indeed closable with the characteristic matrix
$$
\begin{bmatrix}
p_{11}&p_{12}\\
\tilde{p}\,p_{21}&\tilde{p}\,p_{22}
\end{bmatrix}
=
\begin{bmatrix}
p_{11}&p_{12}\\
p_{21}&\tilde{p}\,p_{22}
\end{bmatrix}.
$$

%%%%%%%%%%%%%%%%%%%%%%%%%%%%%%%%%%%%%%%%%%%%%%
%%%%%%%%%%%%%%%%%%%%%%%%%%%%%%%%%%%%%%%%%%%%%%
\subsection{Technical lemmas}\label{S-6.3}
%%%%%%%%%%%%%%%%%%%%%%%%%%%%%%%%%%%%%%%%%%%%%%
%%%%%%%%%%%%%%%%%%%%%%%%%%%%%%%%%%%%%%%%%%%%%%

The projection $\tp$ defined by \eqref{F-6.12} 
(in Remark \ref{R-6.9},(iii)) plays a special role in the next result
(and also in the rest of the section).
%%%%%%%%%%%%%%%%%%%%%%% \label{L-6.10} %%%%%%%%%%%%%%%%%%%%%
\begin{lemma}\label{L-6.10}
For $\varphi \in \cM_*^+$
we have an increasing sequence $\{A_n\}_{n \in \bN}$ in $\ellp_+$ satisfying
$$
A_n \nearrow \rmvppsi^{1/2p} \, \tilde{p} \, \rmvppsi^{1/2p}
\quad \mbox{and} \quad  A_n \leq n\rmvpopsi^{1/p},
$$
and the map
\begin{equation}\label{F-6.13} %%%%%%%%%%%%%%%%%%%%%%%%%%%%% \label{F-6.13}
x\xi_{\varphi_0} \in \cM\xi_{\varphi_0} \ \longmapsto \  
\langle  \rmvppsi^{1/2p} \,  \tp \, \rmvppsi^{1/2p}, \rmvpopsi^{1/2q}x^*x\rmvpopsi^{1/2q} \rangle \in \bR_+
\end{equation}
is always lower semi-continuous.
\end{lemma}
\begin{proof}
Recall that $\tp\, T_{\varphi}\,(= (T_{\varphi})_{c})$
is closable and its closure $\overline{\tp\, T_{\varphi}}$ is affiliated with $\cM$
(see Remark \ref{R-6.9},(iii)).
Then, with $\tp \,T_{\varphi}$ (instead of $T_{\varphi}$) we can repeat almost identical arguments as 
in the proof of  (i) $\Longrightarrow$ (ii) in Proposition \ref{P-6.2}.
We will just sketch arguments with explanations on some differences.
Using the spectral decomposition of $|\overline{\tp \, T_{\varphi}}|^2$ we construct
$h_n \in \cM_+$ and $A_n=\rmvpopsi^{1/2p} h_n \rmvpopsi^{1/2p} \in \ellp_+$.
Note that \eqref{F-6.4} is modified to
$$
\|\overline{\tp\, T_{\varphi}}JxJ\xi_{\varphi_0}\|^2
=\|JxJ\overline{\tp\, T_{\varphi}}\xi_{\varphi_0}\|^2=\|JxJ\tp\, T_{\varphi}\xi_{\varphi_0}\|^2
=\|JxJ\tp \,\Delta_{\varphi\vpo}^{1/2p}\xi_{\varphi_0}\|^2
$$
and hence \eqref{F-6.5} changes to
\begin{eqnarray}
&&
\mbox{tr}\left(
\left(\tp\, \rmvppsi^{1/2p} \rmvpopsi^{1/2q} x^* \right) 
\left(\tp\, \rmvppsi^{1/2p} \rmvpopsi^{1/2q} x^* \right)^* 
\right)
=\mbox{tr}\left( \tp\,\rmvppsi^{1/2p}\rmvpopsi^{1/2q}x^*x\rmvpopsi^{1/2q}\rmvppsi^{1/2p}\, \tp \right)
\nonumber\\
&&
\qquad
=\mbox{tr}\left( \rmvppsi^{1/2p}\, \tp \,\rmvppsi^{1/2p}\rmvpopsi^{1/2q}x^*x\rmvpopsi^{1/2q} \right)
=\langle  \rmvppsi^{1/2p} \, \tp \, \rmvppsi^{1/2p}    , \rmvpopsi^{1/2q}x^*x\rmvpopsi^{1/2q}  \rangle.
\label{F-6.14} %%%%%%%%%%%%%%%%%%%%%%%%%%%%%%%%%%%%% \label{F-6.14}
\end{eqnarray}
Thus, we have 
\begin{eqnarray*}
&&
\langle \rmvppsi^{1/2p} \, \tp \,  \rmvppsi^{1/2p}, \rmvpopsi^{1/2q}x^*x\rmvpopsi^{1/2q} \rangle
=\|\overline{\tp\, T_{\varphi}}JxJ\xi_{\varphi_0}\|^2\\
&& \qquad \qquad
=\sup_n \, (h_n JxJ\xi_{\varphi_0},JxJ\xi_{\varphi_0})
=\sup_n \, \langle A_n, \rmvpopsi^{1/2q}x^*x\rmvpopsi^{1/2q} \rangle.
\end{eqnarray*}
Here, the above last equality holds true since the computation \eqref{F-6.3} remains in the same form.
This means $A_n \nearrow \rmvppsi^{1/2p} \, \tp \, \rmvppsi^{1/2p}$ as before (i.e., thanks to
Lemma \ref{L-6.3} again).

Since  
$ 
\langle \rmvppsi^{1/2p} \, \tp \,  \rmvppsi^{1/2p}, \rmvpopsi^{1/2q}x^*x\rmvpopsi^{1/2q} \rangle
=
\sup_n \, (h_n JxJ\xi_{\varphi_0},JxJ\xi_{\varphi_0})
$ 
as seen above,
lower semi-continuity of the map defined by \eqref{F-6.13} is obvious.
\end{proof}

Lower semi-continuity of \eqref{F-6.13} shown above and the obvious inequality
$$
\rmvppsi^{1/2p} \, \tilde{p} \, \rmvppsi^{1/2p} \leq \rmvppsi^{1/p}
$$
yield
\begin{align}\label{F-6.15} %%%%%%%%%%%%%%%%%%%%%%%%%%%% \label{F-6.15}
\langle \rmvppsi^{1/2p}\,  \tp \, \rmvppsi^{1/2p}, \rmvpopsi^{1/2q}x^*x\rmvpopsi^{1/2q} \rangle 
&=\inf \Bigl(\liminf_{n \to \infty}    
\, \langle \rmvppsi^{1/2p}\,  \tp \, \rmvppsi^{1/2p}, \rmvpopsi^{1/2q}x_n^*x_n\rmvpopsi^{1/2q} \rangle 
\Bigr)
\nonumber\\
&\leq \inf \Bigl(\liminf_{n \to \infty}    
\, \langle \rmvppsi^{1/p}, \rmvpopsi^{1/2q}x_n^*x_n\rmvpopsi^{1/2q} \rangle 
\Bigr)
\end{align}
for each $x \in \cM$.
Here, the infimum is taken over all sequences $\{x_n\}_{n \in \bN}$ in $\cM$ satisfying
$x_n\xi_{\varphi_0} \to x\xi_{\varphi_0}$ in $\cH$.

Actually, the three quantities in \eqref{F-6.15} are all identical.
%%%%%%%%%%%%%%%%% \label{L-6.11} %%%%%%%%%%%%%%%%%%%%%%%%%%
\begin{lemma}\label{L-6.11}
For each $x \in \cM$ we have
$$
\langle \rmvppsi^{1/2p}\,  \tp\,  \rmvppsi^{1/2p}, \rmvpopsi^{1/2q}x^*x\rmvpopsi^{1/2q} \rangle 
=
\inf \Bigr(\liminf_{n \to \infty}    
\, \langle  \rmvppsi^{1/p}, \rmvpopsi^{1/2q}x_n^*x_n\rmvpopsi^{1/2q} \rangle \Bigr),
$$
where the infimum is taken over all sequences $\{x_n\}_{n \in \bN}$ in $\cM$ satisfying
$x_n\xi_{\varphi_0} \to x\xi_{\varphi_0}$ in $\cH$.
\end{lemma}
\begin{proof}
For $x \in \cM$ we have $T_{\varphi}JxJ\xi_{\varphi_0}=JxJ\Delta_{\varphi\vpo}^{1/2p}\xi_{\varphi_0}$
(Definition \ref{D-6.1}).
On the other hand, let us recall  the projection $\tilde{q}=1-\tilde{p}$, which corresponds to
the intersection of $\overline{\Gamma(T_{\varphi})}$ and the $y$-axis $0 \oplus \cH$
(as was discussed in Remark \ref{R-6.9}).
Since 
$(0,\tilde{q} JxJ  \Delta_{\varphi\vpo}^{1/2p}\xi_{\varphi_0})$
belongs to $\overline{\Gamma(T_{\varphi})}$, there is a sequence $\{y_n\}$ in $\cM$ satisfying 
$$
Jy_nJ\xi_{\varphi_0} \ \rightarrow \ 0
\quad \mbox{and} \quad 
T_{\varphi}Jy_nJ\xi_{\varphi_0} \ \rightarrow \ \tilde{q}JxJ\Delta_{\varphi\vpo}^{1/2p}\xi_{\varphi_0}.
$$
We set $x_n=x-y_n \in \cM$. We observe $x_n\xi_{\varphi_0} \rightarrow x\xi_{\varphi_0}$ and
$T_{\varphi}Jx_nJ\xi_{\varphi_0} \ \left(= Jx_nJ\Delta_{\varphi\vpo}^{1/2p}\xi_{\varphi_0}\right)$
tends to
$$  
JxJ\Delta_{\varphi\vpo}^{1/2p}\xi_{\varphi_0}-\tilde{q}JxJ\Delta_{\varphi\vpo}^{1/2p}\xi_{\varphi_0}
=\tilde{p}JxJ\Delta_{\varphi\vpo}^{1/2p}\xi_{\varphi_0}
=JxJ\tilde{p}\,\Delta_{\varphi\vpo}^{1/2p}\xi_{\varphi_0}.
$$ 
Hence, we have  
$$
\lim_{n \to \infty} \|Jx_nJ\Delta_{\varphi\vpo}^{1/2p}\xi_{\varphi_0}\|^2
=\|JxJ\tilde{p}\,\Delta_{\varphi\vpo}^{1/2p}\xi_{\varphi_0}\|^2,
$$
but this means
$$
\lim_{n \to \infty}\,\langle  \rmvppsi^{1/p}, \rmvpopsi^{1/2q}x_n^*x_n\rmvpopsi^{1/2q} \rangle
=\langle \rmvppsi^{1/2p}\,  \tp\,  \rmvppsi^{1/2p}, \rmvpopsi^{1/2q}x^*x\rmvpopsi^{1/2q} \rangle
$$
by usual computations that have been carried out repeatedly (see \eqref{F-6.5} and \eqref{F-6.14}).
Thus, the far right side  in \eqref{F-6.15}  is  smaller than 
$\langle \rmvppsi^{1/2p}\,  \tp\,  \rmvppsi^{1/2p}, \rmvpopsi^{1/2q}x^*x\rmvpopsi^{1/2q} \rangle$
(i.e., the far left side  in \eqref{F-6.15}), that is, the three quantities there are all identical.
\end{proof}

%%%%%%%%%%%%%%%%%%%%%%%%%%%%%%%%%%%%%%%%%%%%%%%%%%%
%%%%%%%%%%%%%%%%%%%%%%%%%%%%%%%%%%%%%%%%%%%%%%%%%%%
\subsection{Lebesgue decomposition}\label{S-6.4}
%%%%%%%%%%%%%%%%%%%%%%%%%%%%%%%%%%%%%%%%%%%%%%%%%%%
%%%%%%%%%%%%%%%%%%%%%%%%%%%%%%%%%%%%%%%%%%%%%%%%%%%

In this subsection we obtain results on Lebesgue-type decomposition in $\ellp_+$
among others by combining various technical results obtained so far.
We begin with definitions of absolute continuity and singularity.

%%%%%%%%%%%%%%%%%%%%%%% \label{D-6.12} %%%%%%%%%%%%%%%%%%%%%
\begin{definition}\label{D-6.12}\rm
 $(${\bf Absolute continuity and singularity}$)$\
We assume $1 \leq p < \infty$.
A positive element $A \in \ellp_+$ is \emph{absolutely continuous with respect to 
$\rmvpopsi^{1/p}$},
or $\rmvpopsi^{1/p}$-absolutely continuous in short, if there is an increasing sequence
$\{A_n\}_{n \in \bN}$ in $\ellp_+$
satisfying $A_n \nearrow A$ and $A_n \leq \ell_n \rmvpopsi^{1/p}$ for some $\ell_n>0$.
On the other hand, $A$ is defined to be \emph{$\rmvpopsi^{1/p}$-singular} 
if $B \in  \ellp_+ $ must be $0$ whenever $B \leq A$ and $B \leq \rmvpopsi^{1/p}$. 
\end{definition}
Note that Proposition \ref{P-6.2} and Remark \ref{R-6.9},(ii) characterize absolute continuity in terms of
the associated operator $T_{\varphi}$ (given in Definition \ref{D-6.1}).
Our proof of the next theorem is motivated by discussions in \cite[\S 3]{Ko2} and \cite{Si1}.

%%%%%%%%%%%%%%%%%%%%% \label{T-6.13} %%%%%%%%%%%%%%%%%%%%%%%%
\begin{theorem}\label{T-6.13}
$(${\bf Maximal absolutely continuous part}$)$ \
We assume that $\varphi_0$ is a faithful positive linear functional in  $\cM_*^+$.
For a given  positive operator $\rmvppsi^{1/p} \in \ellp_+$ $($with $\varphi \in \cM_*^+$$)$
the operator $\rmvppsi^{1/2p} \, \tilde{p} \, \rmvppsi^{1/2p} \in \ellp_+$
with the projection $\tilde{p} \in \cM$ defined by \eqref{F-6.12}
$($see Remark \ref{R-6.9}$)$ is 
$\rmvpopsi^{1/p}$-absolutely continuous. Moreover, it is the maximum among all
$\rmvpopsi^{1/p}$-absolutely continuous elements in $\ellp_+$ majorized by $\rmvppsi^{1/p}$.
\end{theorem}
\begin{proof}
The first statement is just Lemma \ref{L-6.10}. Hence, it remains to show maximality stated
in the last part.
Let us assume that $\Delta_{\varphi_1\psi}^{1/p}  \in \ellp_+$ (with $\varphi_1 \in \cM_*^+  $)
is $\rmvpopsi^{1/p}$-absolutely continuous
and $\Delta_{\varphi_1\psi}^{1/p} \leq \rmvppsi^{1/p}$. 
Proposition \ref{P-6.2},(iii)  (applied for $\varphi_1$) says that the map
$$
x\xi_{\varphi_0} \in \cM\xi_{\varphi_0} \ \longmapsto \  
\langle  \Delta_{\varphi_1\psi}^{1/p}, \rmvpopsi^{1/2q}x^*x\rmvpopsi^{1/2q} \rangle \in \bR_+
$$
is lower semi-continuous so that we have
$$
\langle  \Delta_{\varphi_1\psi}^{1/p}, \rmvpopsi^{1/2q}x^*x\rmvpopsi^{1/2q} \rangle
=
\inf \Bigr(\liminf_{n \to \infty}    
\, \langle  
\Delta_{\varphi_1\psi}^{1/p}, \rmvpopsi^{1/2q}x_n^*x_n\rmvpopsi^{1/2q} 
\rangle
\Bigr)
$$
for each $x \in \cM$. 
Here (and below) the infimum is taken over all sequences $\{x_n\}_{n \in \bN}$ in $\cM$ 
satisfying $x_n\xi_{\varphi_0} \to x\xi_{\varphi_0}$. 
Thanks to $\Delta_{\varphi_1\psi}^{1/p} \leq \rmvppsi^{1/p}$
the above right side is smaller than
$$
\inf \Bigr(\liminf_{n \to \infty}    
\, \langle  \rmvppsi^{1/p}, \rmvpopsi^{1/2q}x_n^*x_n\rmvpopsi^{1/2q} \rangle \Bigr),
$$
which is equal to $\langle \rmvppsi^{1/2p}\,  \tp \, \rmvppsi^{1/2p}, \rmvpopsi^{1/2q}x^*x\rmvpopsi^{1/2q} \rangle$
according to Lemma \ref{L-6.11}.  Thus, we have shown 
$$
\langle  \Delta_{\varphi_1\psi}^{1/p}, \rmvpopsi^{1/2q}x^*x\rmvpopsi^{1/2q} \rangle
\leq 
\langle \Delta_{\varphi\vpo}^{1/2p} \, \tp\,  \Delta_{\varphi\vpo}^{1/2p}, \rmvpopsi^{1/2q}x^*x\rmvpopsi^{1/2q} \rangle
$$
for each $x \in \cM$. 
Then  Lemma \ref{L-6.3} enables us to change $\rmvpopsi^{1/2q}x^*x\rmvpopsi^{1/2q}$ (in the above
inequality) to arbitrary positive elements in $\ellq_+$ and consequently we have
$\Delta_{\varphi_1\psi}^{1/p} \leq \rmvppsi^{1/2p}\, \tp\, \rmvppsi^{1/2p}$ as desired.
\end{proof}

%%%%%%%%%%%%%%%%%% \label{R-6.14} %%%%%%%%%%%%%%%%%%%%%
\begin{remark}\label{R-6.14}\rm
As was mentioned at the beginning of the section.
Ando \cite{An1} defined $a[b]=\sup_n (a:(nb))$ for positive bounded operators $a,b$
(to study Lebesgue decomposition).
Since $\{a:(nb)\}_n$ is an increasing sequence majorized by $a$, this sequence
converges to $a[b]$ in the strong operator topology.

Let us mimic this procedure in our non-commutative $L^p$-space setting.
We have to begin by understanding the meaning of $\sup$.
We assume that $\{A_n\}$ is an increasing sequence in $\ellp_+$ satisfying
$A_n \leq A$ for some $A \in \ellp_+ $. Then, $\{A-A_n\}$ is a decreasing 
sequence in  $\ellp_+$. We set $X=\Inf_n (A-A_n)$.  Indeed, we know $X \in \ellp_+$
and $A-A_n \to X$ in the strong resolvent sense and also in the $L^p$-norm.
(It is possible to use Proposition \ref{P-4.2}  here 
since Haagerup and Hilsum $L^p$-spaces
are  isometrically order-isomorphic.) 
Therefore, $A_n=A-(A-A_n) \to A-X$ in the $L^p$-norm.
Since $A-A_n \geq X$, we have $A-X \geq A_n$.
Conversely, if $Y \geq A_n$ $($for each $n$$)$, then $A-A_n \geq A-Y$ 
and hence $X \geq A-Y$, i.e., $Y \geq A-X$.
Thus, $A-X$ is the minimal element majorizing all $A_n$, i.e., $A-X=\sup_n A_n$.
Therefore, we have seen that 
\begin{quote}
an increasing sequence having an upper bound (in $\ellp_+$)
converges to its supremum in the $L^p$-norm. 
\end{quote}
We recall \eqref{F-6.1} (which is the strong limit) and
start from $\rmvppsi^{1/p}$, $\rmvpopsi^{1/p}$ in $\ellp_+$.
Parallel sums in non-commutative $L^p$-spaces were
studied in \S\ref{S-4}.
The increasing sequence $\{\rmvppsi^{1/p} : (n\rmvpopsi^{1/p})\}_{n \in \bN}$
bounded by $\rmvppsi^{1/p}$ from the above converges to the supremum in the $L^p$-norm
(as was remarked above).
Let us denote this supremum by 
$$
\rmvppsi^{1/p} [\rmvpopsi^{1/p}] =\sup_n \left(\rmvppsi^{1/p} : (n\rmvpopsi^{1/p})\right).
$$
As is explained at the beginning of \S \ref{S-6}
$($see also \cite[Corollary 3]{Ko7}$)$ $\rmvppsi^{1/p} [\rmvpopsi^{1/p}]$
is the maximum of all  $\rmvpopsi^{1/p}$-absolutely continuous operators $($in $\ellp_+$$)$
majorized by $\rmvppsi^{1/p}$, and hence we conclude
$$
\rmvppsi^{1/p} [\rmvpopsi^{1/p}] = \rmvppsi^{1/2p} \, \tilde{p} \, \rmvppsi^{1/2p}, 
$$ 
i.e., the maximal $\rmvpopsi^{1/p}$-absolutely continuous part captured 
in Theorem \ref{T-6.13}.
\end{remark}

%%%%%%%%%%%%%%%%%%%% \label{T-6.15} %%%%%%%%%%%%%%%%%%%%%%%%
\begin{theorem}\label{T-6.15}
$(${\bf Lebesgue decomposition}$)$ \
With the same notations as in the previous theorem the difference
$$
\rmvppsi^{1/p}-\rmvppsi^{1/2p} \,\tilde{p}\, \rmvppsi^{1/2p}
=\rmvppsi^{1/2p}(1- \tilde{p})\rmvppsi^{1/2p}
\ \bigl(\in \ellp_+\bigr)
$$
is $\rmvpopsi^{1/p}$-singular and hence
$$
\rmvppsi^{1/p}
=
\rmvppsi^{1/2p} \, \tilde{p }\, \rmvppsi^{1/2p}
+ \rmvppsi^{1/2p} \, (1- \tilde{p}) \, \rmvppsi^{1/2p}
$$
gives us a Lebesgue decomposition of $\rmvppsi^{1/p}$ into $($maximal$)$
$\rmvpopsi^{1/p}$-absolutely continuous and $\rmvpopsi^{1/p}$-singular parts.
\end{theorem}
\begin{proof}
Let us assume that $B \in \ellp_+ $ satisfies 
$B \leq \rmvppsi^{1/2p}(1- \tilde{p})\rmvppsi^{1/2p}$ 
and $B \leq \rmvpopsi^{1/p}$.
At first we observe that the sum of two $\rmvpopsi^{1/p}$-absolutely continuous elements remains to be 
$\rmvpopsi^{1/p}$-absolutely continuous, which is clear from the definition.
Note that $B \, \left(\leq \rmvpopsi^{1/p}\right)$ is obviously $\rmvpopsi^{1/p}$-absolutely continuous 
and hence so is the sum  $B+\rmvppsi^{1/2p}\, \tilde{p}\, \rmvppsi^{1/2p}$.
However, we have $B+\rmvppsi^{1/2p}\, \tilde{p} \, \rmvppsi^{1/2p} \leq \rmvppsi^{1/p}$
because of $B \leq \rmvppsi^{1/2p}(1- \tilde{p})\rmvppsi^{1/2p}$.
Thus, the maximality stated in the preceding theorem implies  
$B + \rmvppsi^{1/2p}\,\tilde{p}\,\rmvppsi^{1/2p} 
\leq \rmvppsi^{1/2p}\,\tilde{p}\,\rmvppsi^{1/2p}$,
showing $B=0$.
\end{proof}

Recall that 
via $\rmvppsi^{1/p} \leftrightarrow h_{\varphi}^{1/p}$ one can identify $\ellp_+$ with
Haagerup's $L^p(\cM)_+$. The latter is independent of the choice of $\psi$ so that 
absolute continuity, Lebesgue decomposition and so on  between $\rmvppsi, \rmvpopsi \in \ellp_+$
are the same as those of $\Delta_{\varphi\psi_1}, \Delta_{\varphi_0\psi_1} \in L^p(\cM, \psi_1)_+$
with a different faithful $\psi_1 \in \cM_*^+$.
More precisely, everything is determined by just $\varphi$ and $\varphi_0$ (and no role is played
by $\psi$). Observe that $T_{\varphi}$ in Definition \ref{D-6.1} 
and
the projection $\tilde{p}$ in Remark \ref{R-6.9}
are determined just by $\varphi, \varphi_0$. 

We point out that for a finite von Neumann algebra the concept of absolute continuity
is not so meaningful. In fact, (whenever $\varphi_0$ is faithful) any $\rmvppsi^{1/p}$ is
$\rmvpopsi^{1/p}$-absolutely continuous. Indeed, we take a finite trace as $\psi$.
Then the Hilsum $L^p(\cM,\tau')$ is just the classical $L^p$-space $L^p(\cM,\tau)_+$ 
(in the setting of semi-finite von Neumann algebras).
Thus, what we are claiming is the following:
\begin{quote}
If $H, K \in L^p(\cM,\tau)_+ $ with $K$ non-singular,
then
$H$ is always $K$-absolutely continuous. 
\end{quote}
This comes from the classical fact 
(which goes back to Murray-von Neumann)
that all (unbounded) operators affiliated with $\cM$ can be manipulated 
without worrying domain questions,  closablity and so on, 
that is, everything is $\tau$-measurable.
Thus, $X=K^{-1/2}HK^{-1/2}$ always makes a perfect sense.
Thus, with the spectral projections $\{e_{\lambda}\}$ of $X$ 
we have 
$$
H_n=K^{1/2}\left(\int_0^{n} \lambda\,de_{\lambda}\right)K^{1/2} \nearrow K^{1/2}XK^{1/2}=H
$$
(as $n \to \infty$) with the obvious estimate  $H_n \leq nK$.

On the other hand, a lot of pathological phenomena are known even in $B(\cH)$ 
(see \cite[\S 10]{Ko3} for instance)
and hence so is the case for properly infinite von Neumann algebras $\cM \cong \cM \otimes B(\cH)$.

%%%%%%%%%%%%%%%%%%%%%%%%%%%%%%%%%%%%%%%%%%%%%%%%
%%%% Appendix %%%%%%%%%%%%%%%%%%%%%%%%%%%%%%%%%%%%%%%
%%%%%%%%%%%%%%%%%%%%%%%%%%%%%%%%%%%%%%%%%%%%%%%%
\appendix

%%%%%%%%%%%%%%%%%%%%%%%%%%%%%%%%%%%%%%%%%%%%%%%%%
%%%%%%%%% Appendix A    A Transformer equality %%%%%%%%%%%%%%%%%%%%
%%%%%%%%%%%%%%%%%%%%%%%%%%%%%%%%%%%%%%%%%%%%%%%%%

%%%%%%%%%%%%%%%%%%%%%%%%%%%%%%%%%%%%%%%%%%%%%%%%%
%%%%%%%%%%% Appendix A.    Haagerup's $L^p$-spaces %%%%%%%%%%%%%%%%
%%%%%%%%%%%%%%%%%%%%%%%%%%%%%%%%%%%%%%%%%%%%%%%%%
\section{Haagerup's $L^p$-spaces}\label{S-A}

In this appendix we give a brief survey on Haagerup's $L^p$-spaces as a preliminary for
\S\ref{S-4} for the reader's convenience.

Let $\cM$ be a general von Neumann algebra on a Hilbert space $\cH$.
We fix a faithful semi-finite normal weight $\varphi_0$ on $\cM$ and the associated modular
automorphims $\sigma_t \, (=\sigma_t^{\varphi_0})$. Let
$$
\cR:=\cM \rtimes_{\sigma} {\bR}
$$
be the crossed product (acting on $L^2(\bR, \cH) = \cH \otimes L^2({\bR},dt)$), 
and the \emph{dual action} on $\cR$ is 
given by $\theta_s=\Ad\, \mu(s)\vert_{\cR}$ ($s \in \bR$) with 
\begin{equation}\label{F-A.1} %%%%%%%%%%%%%%%%%%%%%%%%%%%%%% \label{F-A.1}
(\mu(s)\xi)(t)=e^{-ist}\xi(t)  \quad (\mbox{for $\xi  \in  L^2(\bR, \cH)$}).
\end{equation}
(See \cite{Ta} for basics of crossed products.)
Let $T$ be the operator valued weight from $\cR$ to $\cM$ defined by
$$
T(x):=\int_{-\infty}^{\infty} \theta_s(x)\,ds \in \widehat{\cM}_+ \quad (\mbox{for $x \in \cR_+$}),
$$
where $\widehat\cM_+$ is the \emph{extended positive part} of $\cM$ (see \cite{Haa3,Te1}).
We note the invariance  $T\circ \theta_s=T$ ($s\in\bR$). Let $P_0(\cM,\bC)$ denote the set
of all semi-finite normal weights on $\cM$ (as in \S\ref{S-5}).
For each weight $\varphi \in P_0(\cM, \bC)$ the composition 
$$
\widehat{\varphi}=\varphi \circ T \in P_0(\cR, \bC) 
$$
 (where $\varphi$ is extended to a normal weight on $\widehat{\cM}_+$) 
is nothing but  the \emph{dual weight} of $\varphi$ (see \cite{Haa2} for details).
Then, the modular automorphism group $\sigma_t^{\widehat{\varphi}_0}$ on
$\cR=\cM \rtimes_{\sigma} {\bR}$ is implemented by the translation
operators $\lambda(t)$, i.e., $\sigma_t^{\widehat\sigma_0}=\Ad\,\lambda(t)$ with
$$
(\lambda(t)\xi)(s)=\xi(s-t) \quad (\xi \in L^2(\bR, \cH))
$$
(or more precisely $1 \otimes \lambda(t)$ on $L^2({\bR}, \cH) \cong \cH \otimes L^2(\bR,ds)$).
This is of course one of generators in the crossed product $\cM \rtimes_{\sigma} {\bR}$,
and hence  $\sigma_t^{\widehat{\varphi}_0}$ is inner, that is, $\cR$ is semi-finite.
Stone's theorem says $\lambda(t)=e^{itH_0}$ with some self-adjoint operator $H_0$, and
$\tau=\widehat{\varphi}_0(H^{-1}\cdot)$ (with the non-singular positive self-adjoint 
operator $H=e^{H_0}$) is a trace on $\cR$. 
This is often referred to as the \emph{canonical trace}
and possesses the following trace-scaling property:
$$
\tau \circ \theta_s=e^{-s}\tau \quad (s \in {\bR}).
$$
It is an easy exercise to show that via Fourier transformation $\mathcal F$ the operator $H$
is transformed to
\begin{equation}\label{F-A.2} %%%%%%%%%%%%%%%%%%%%%%%%%%%%%% \label{F-A.2}
{\mathcal F}H{\mathcal F}^*=m_{e^t},
\end{equation}
the multiplication operator induced by the function $t \mapsto e^t$.
For a weight $\varphi \in P_0(\cM,\bC)$ let $h_{\phi}$ be the Radon-Nikodym derivative 
of the dual weight $\widehat{\varphi}$ with respect to $\tau$, 
i.e., $\widehat{\varphi}=\tau(h_{\varphi}\,\cdot)$. Since $\widehat{\varphi}\circ \theta_s=\widehat{\varphi}$
from the definition, we have $\theta_s(h_{\varphi})=e^{-s}h_{\varphi}$. This scaling property
is known to characterize all (densely defined) positive self-adjoint operators $h$ (on $L^2(\bR, \cH)$)
of the form $h=h_{\varphi}$ for some $\varphi \in P_0(\cM,\bC)$.
We set
\begin{align*} 
{\bold H}&:=\mbox{the set of all positive self-adjoint operators $h$ affiliated with $\cR$}\\
& \qquad \qquad \qquad \mbox{satisfying $\theta_s(h)=e^{-s}h$ for each $s \in \bR$}
\end{align*}
so that
\begin{equation}\label{F-A.3} %%%%%%%%%%%%%%%%%%%%%%%%%%%%%% \label{F-A.3}
\varphi \in P_0(\cM, \bC) \ \longleftrightarrow \ h_{\varphi} \in {\bold H}
\end{equation}
is an order preserving one-to-one bijective correspondence.

One of the most important ingredients in Haagerup's theory on non-commutative $L^p$-spaces is 
the fact that $h_{\varphi}$ is a $\tau$-measurable operator exactly when $\varphi \in \cM_*^+$,
that is, restricting \eqref{F-A.3} to $\cM_*^+$ we have an order preserving bijection
$\ffi\in\cM_*^+\leftrightarrow h_\ffi\in{\bold H}\cap\overline\cR$, where $\overline\cR$ is
the space of $\tau$-measurable operators affiliated with $\cR$. For $0<p\le\infty$
\emph{Haagerup's $L^p$-space} is defined by
\begin{align}\label{F-A.4}%%%%%%%%%%%%%%%%%%%%%%%%%%%%%%%%% \label{F-A.4}
L^p(M):=\{a\in\overline\cR;\,\theta_s(a)=e^{-s/p}a,\,s\in\bR\},
\end{align}
and its positive part is $L^p(\cM)_+:=L^p(\cM)\cap\overline\cR_+$.
In particular, $L^\infty(\cM)=\cR^\theta$ (the $\theta$-fixed point algebra) $=\cM$.
Note $L^1(\cM)={\bold H}\cap\overline\cR$ and the above bijection extends to
$\rho\in\cM_*\leftrightarrow h_\rho\in L^1(\cM)$ by linearity. Thus, a positive linear
functional $\tr$ is defined on $L^1(\cM)$ as
$$
\tr(h_\rho)=\rho(1)\quad(\rho\in\cM_*).
$$
For
$0<p<\infty$ the $L^p$-(quasi)norm is defined by
$$
\|a\|_p:=\bigl(\tr(|a|^p)\bigr)^{1/p}\quad(a\in L^p(\cM)).
$$
Also $\|\cdot\|_\infty$ denotes the operator norm on $\cM$. When
$1\le p<\infty$, $L^p(M)$ is a Banach space with the norm $\|\cdot\|_p$, and its dual Banach
space is $L^q(\cM)$, where $1/p+1/q=1$, by the duality form
\begin{align}\label{F-A.5}  %%%%%%%%%%%%%%%%%%%%%%%%%%%%%%%% \label{F-A.5}
\<a,b\>_{p,q}:=\tr(ab)\ (=\tr(ba))\quad \mbox{for $a\in L^p(\cM)$, $b\in L^q(\cM)$}.
\end{align}
In particular, $L^2(\cM)$ is a Hilbert space with the inner product
$$
(a,b):=\tr(b^*a)\ (=\tr(ab^*)).
$$
Then the quadruple 
$$
\<\pi_{\ell}(M),L^2(\cM),\,{}^*,L^2(\cM)_+\>
$$ 
(where $\pi_{\ell}$ means the left multiplication) is a standard form of $\cM$.

Note that the space $L^p(\cM)$ is independent (up to an isometric isomorphism) of the
choice of $\ffi_0$. In fact, set $\cR_1:=\cM\rtimes_{\sigma^{\ffi_1}}\bR$ for
another faithful semi-finite normal weight $\ffi_1$ on $\cM$. Then there exists an ismorphism
$\kappa:\cR\to\cR_1$ such that $\kappa\circ\theta_s\circ\kappa^{-1}$ is the dual action on
$\cR_1$ and $\tau\circ\kappa^{-1}$ is the canonical trace on $\cR_1$ (see
\cite[Chap.~II, Theorem 37]{Te1}). The $\kappa$ extends to a homeomorphic isomorphism
$\overline\cR\to\overline\cR_1$, and it induces an isometric isomorphism between $L^p(\cM)$
in $\overline\cR$ and that in $\overline\cR_1$. When $\cM$ is semi-finite with a trace
$\tau_0$, $L^p(\cM)$ is naturally identified with the classical $L^p$-space $L^p(\cM,\tau_0)$ 
(\cite{Di,Ku,Y})  with respect to
$\tau_0$ (see the last part of \cite[Chap.~II]{Te1}).

%%%%%%%%%%%%%%%%%%%%%%%%%%%%%%%%%%%%%%%%%%%%%%%%%
%%%%% Appendix B  Connes' spatial theory %%%%%%%%%%%%%%%%%%%%%%%%%
%%%%%%%%%%%%%%%%%%%%%%%%%%%%%%%%%%%%%%%%%%%%%%%%%
\section{Connes' spatial theory}\label{S-B}

Materials in \S\ref{S-5} and \S\ref{S-6} require some familiarity on Connes' spatial theory (\cite{C3}),
and brief explanation on this theory is given for the reader's convenience.
However, for the definition of spatial derivatives we mainly follow accounts in \cite{Te1}. 
Actually slightly different (but of course equivalent) explanations 
are presented in these two articles. The main difference is that use of extended 
positive parts is emphasized in the latter 
(see Remark \ref{R-B.1} and  Remark \ref{R-B.3}),
and this viewpoint seems to be more fitting to our purpose here.

%%%%%%%%%%%%%%%%%%%%%%%%%%%%%%%%%%%%%%%%%%%%%%%%%
%%%%% B.1 Spatial derivatives %%%%%%%%%%%%%%%%%%%%%%%%%%%%%%%%
%%%%%%%%%%%%%%%%%%%%%%%%%%%%%%%%%%%%%%%%%%%%%%%%%
\subsection{Spatial derivatives}\label{S-B.1}

Throughout let $\cM$ be a von Neumann algebra acting on a Hilbert space $\cH$
and $\chi$ be a (fixed) faithful  semi-finite normal weight on the commutant $\cM'$.
We will use the following standard notations:
\begin{align*}
n_{\chi}&:=\{y' \in \cM'; \ \chi(y'^*y') < \infty\},\\
\cH_{\chi}&:=\mbox{the Hilbert space completion of $n_{\chi}$ relative to the inner product}\\
&  \hskip 5cm   (y_1', y_2') \mapsto \chi(y_2'\mbox{}^*y_1'),\\ 
\Lambda_{\chi}&:= \mbox{the canonical injection $n_{\chi} \rightarrow \cH_{\chi}$},\\
\pi_{\chi}&:= \mbox{the GNS regular representation of $\cM'$ on $\cH_{\chi}$}.
\end{align*}
For each $\xi \in \cH$ we set
$$
R^{\chi}(\xi)\Lambda_{\chi}(y')=y'\xi \quad (\mbox{for $y' \in \cM'$}),
$$
which is a densely defined ($\cD(R^{\chi}(\xi))=\Lambda_{\chi}(n_{\chi})$) operator from $\cH_{\chi}$
to $\cH$. Obviously $\xi \mapsto R^{\chi}(\xi)$ is linear, and it is also plain to see
\begin{eqnarray}
&&\label{F-B.1} %%%%%%%%%%%%%%%%%%%%%%%%%%%%%%%%%%%%% \label{F-B.1}
R^{\chi}(x\xi)=xR^{\chi}(\xi) \quad \mbox{for $x \in \cM$},\\
&&\label{F-B.2} %%%%%%%%%%%%%%%%%%%%%%%%%%%%%%%%%%%%% \label{F-B.2}
y'R^{\chi}(\xi) \subseteq R^{\chi}(\xi)\pi_{\chi}(y') \quad \mbox{for $y' \in \cM'$}.
\end{eqnarray}
A vector $\xi \in \cH$ is said to be \emph{$\chi$-bounded} when $R^{\chi}(\xi)$ is a bounded operator, that is,
$$
\|y'\xi\|^2 \leq C\|\Lambda_{\chi}(y')\|^2=C\chi(y'{}^*y') \quad (y' \in n_{\chi})
$$
for some constant $C>0$.
We set
$$
D(\cH, \chi)=\{\xi \in \cH; \, \mbox{$\xi$ is a $\chi$-bounded vector}\},
$$
which is $\cM$-invariant due to \eqref{F-B.1}.
One can show that $D(\cH, \chi)$ is a dense subspace in $\cH$ (\cite[Lemma 2]{C3}). 

When $\xi \in D(\cH, \chi)$, $R^{\chi}(\xi)$ is extended to a bounded operator from $\cH_{\chi}$ to $\cH$
(for which the same symbol $R^{\chi}(\xi)$ will be used). For $\xi \in D(\cH, \chi)$ we set
$$
\theta^{\chi}(\xi,\xi)=R^{\chi}(\xi)R^{\chi}(\xi)^* \in \cM_+.
$$
Indeed, $R^{\chi}(\xi)R^{\chi}(\xi)^*$ is a bounded operator from $\cH$ to itself, which commutes
with each $y' \in \cM'$ (due to \eqref{F-B.2}).
Thanks to  \eqref{F-B.1} we have
\begin{equation}\label{F-B.3} %%%%%%%%%%%%%%%%%%%%%%%%%%%%%%% \label{F-B.3}
\theta^{\chi}(x\xi, x\xi)=x\theta^{\chi}(\xi,\xi)x^*
\end{equation}
for $x \in \cM$ and $\xi \in D(\cH, \chi)$.

%%%%%%%%%%%%%%%%%%%%%%% \label{R-B.1} %%%%%%%%%%%%%%%%%%%%%%
\begin{remark}\label{R-B.1}\rm

For a generic vector $\xi \in \cH$ $($not necessarily $\xi \in D(\cH, \chi)$$)$
$R^{\chi}(\xi)^*$ is a $($possibly non-densely defined$)$ closed operator from 
$\cH_{\chi}$ to $\cH$. Therefore, the relation
$$
\langle \omega_{\zeta},\theta^{\chi}(\xi,\xi) \rangle=
\left\{
\begin{array}{cl}
\|R^{\chi}(\xi)^*\zeta\|^2&\mbox{if $\zeta \in {\mathcal D}(R^{\chi}(\xi)^*)$},\\[2mm]
+\infty         & \mbox{otherwise}
\end{array}
\right.
$$
uniquely determines an element $\theta^{\chi}(\xi,\xi)$ in the extended positive part
$\widehat{\cM}_+$ $($see \cite[\S 1]{Haa3} and \cite[Chap.~III, Definition 5]{Te1} for details$)$.
Here, $\omega_{\zeta}$ means a positive linear functional $(\cdot\,\zeta,\zeta)$.
Then, \eqref{F-B.3} remains valid with the understanding
$$
\langle \omega_{\zeta}, x\theta^{\chi}(\xi,\xi)x^* \rangle
=\langle x^*\omega_{\zeta}x, \theta^{\chi}(\xi,\xi) \rangle 
=\langle \omega_{\zeta}(x \cdot x^*), \theta^{\chi}(\xi,\xi) \rangle 
=\langle \omega_{x^*\zeta}, \theta^{\chi}(\xi,\xi) \rangle.
$$
\end{remark}

Let $\psi$ be a semi-finite normal weight on $\cM$. We set
\begin{equation}\label{F-B.4} %%%%%%%%%%%%%%%%%%%%%%%%%%%%%%% \label{F-B.4}
q_{\psi}(\xi)=\psi(\theta^{\chi}(\xi,\xi)) \in [0,\infty] \quad \mbox{for $\xi \in \cH$}
\end{equation}
(where $\psi$ is extended to a weight on  $\widehat{\cM}_+$).
It is easy to see that $\xi \in \cH \mapsto q_{\psi}(\xi) \in [0,\infty]$ is a quadratic form, i.e.,
$$
q_{\psi}(\xi_1+\xi_2)+q_{\psi}(\xi_1-\xi_2)=2q_{\psi}(\xi_1)+2q_{\psi}(\xi_2), 
\quad q_{\psi}(\lambda\xi)=|\lambda|^2q_{\psi}(\xi)
$$
for $\xi_1,\xi_2,\xi \in \cH$ and $\lambda \in \bC$. Moreover, $q_{\psi}(\cdot)$
is lower semi-continuous (and hence $q_{\psi}$ is a positive form in the sense explained in \S\ref{S-2})
as shown below.

To show lower semi-continuity, we begin with the special case 
$\psi=\omega_{\zeta}$ with a vector $\zeta \in \cH$.
Since ${\mathcal D}(R^{\chi}(\xi))=\Lambda_{\chi}(n_{\chi})$, 
for $\zeta \in {\mathcal D}(R^{\chi}(\xi)^*)$ we have
\begin{align}\label{F-B.5} %%%%%%%%%%%%%%%%%%%%%%%%%%%%%% \label{F-B.5)
q_{\omega_{\zeta}}(\xi)&=\|R^{\chi}(\xi)^*\zeta\|^2
=\sup_{\chi(y'^*y') \leq 1} |(R^{\chi}(\xi)^*\zeta,\Lambda_{\chi}(y'))|^2
\nonumber\\
&=\sup_{\chi(y'^*y') \leq 1} |(\zeta,R^{\chi}(\xi)\Lambda_{\chi}(y'))|^2
=\sup_{\chi(y'^*y') \leq 1} |(\zeta,y'\xi)|^2.
\end{align}
This variational expression \eqref{F-B.5} remains valid for $\zeta \not\in {\mathcal D}(R^{\chi}(\xi)^*)$ as well. 
In this case  we have $q_{\omega_{\zeta}}(\xi)=+\infty$ from the definition and
$\sup_{\chi(y'^*y') \leq 1} |(\zeta,R^{\chi}(\xi)\Lambda_{\chi}(y'))|^2=+\infty$ also holds true.
Indeed, if this $\sup$ were $C <\infty$, then we would get 
$|(\zeta,R^{\chi}(\xi)\Lambda_{\chi}(y'))|\leq C^{1/2}\|\Lambda_{\chi}(y')\|$ for each $y' \in n_{\chi}$ and hence 
$(\zeta,R^{\chi}(\xi)\Lambda_{\chi}(y'))=(\eta,\Lambda_{\chi}(y'))$ for some vector $\eta \in \cH_{\chi}$,
contradicting $\zeta \not\in {\mathcal D}(R^{\chi}(\xi)^*)$.
Since the map $\xi \mapsto (\zeta,y'\xi)$ is continuous for each fixed $y'$, 
$q_{\omega_{\zeta}}(\cdot)$
is lower semi-continuous thanks to \eqref{F-B.5}.
Finally a general weight $\psi$ can be expressed as $\psi=\sum_{\iota \in I}\omega_{\zeta_{\iota}}$
with a suitable family $\{\zeta_{\iota}\}_{\iota \in I}$ of vectors (see \cite{Haa1}) so that
$
q_{\psi}=\sum_{\iota \in I} q_{\omega_{\zeta_{\iota}}}
$
is also lower semi-continuous.

We set
$$
{\mathcal D}( q_{\psi})=\{ \xi \in \cH; \, q_{\psi}(\xi) < \infty\}.
$$
Since $\psi$ is semi-finite, we can prove that ${\mathcal D}( q_{\psi}) $ is dense in $\cH$. 
This follows from the following observation: for $x \in n_{\psi}=\{x \in \cM;\, \psi(x^*x)<\infty\}$ and
$\xi \in D(\cH, \chi)$ we have
$$
q_{\psi}(x^*\xi)=\psi(\theta^{\chi}(x^*\xi,x^*\xi))
=\psi(x^*\theta^{\chi}(\xi,\xi)x) \leq \|\theta^{\chi}(\xi,\xi)\|\psi(x^*x) < \infty
$$
due to \eqref{F-B.3} (see \cite[Lemma 6]{C3} for details).
Therefore, the restriction
$$
q_{\psi}\vert_{{\mathcal D}( q_{\psi})}: \
\xi \in {\mathcal D}( q_{\psi}) \  \mapsto \ q_{\psi}(\xi) \in [0,\infty)
$$ 
is a densely defined quadratic form.
Lower semi-continuity explained above means that this quadratic form is  closed 
(see \cite[Proposition 10.1]{Sch} for instance).

With preparations so far we are now ready to use the standard representation theorem 
for densely defined closed quadratic forms 
(see \cite[Theorem VIII.15]{RS} or \cite[Theorem 10.7]{Sch} for instance).
%%%%%%%%%%%%%%%%%%%% \label{D-B.2} %%%%%%%%%%%%%%%%%%%%%%%%%
\begin{definition}\label{D-B.2}\rm
The positive self-adjoint operator associated with the densely defined closed quadratic 
form $q_{\psi}\vert_{{\mathcal D}( q_{\psi})}$ $($see \eqref{F-B.4}$)$ is denoted by $d\psi/d\chi$, 
the \emph{spatial derivative} of a $($semi-finite normal$)$ weight $\psi$ on $\cM$ relative to
a $($faithful semi-finite normal$)$ weight $\chi$ on $\cM'$. 
Since we have captured $d\psi/d\chi$ via a closed quadratic form, we have
$$
q_{\psi}(\xi)=
\left\{
\begin{array}{cl}
\|(d\psi/d\chi)^{1/2}\xi\|^2
& \
\mbox{when $\xi \in {\mathcal D}((d\psi/d\chi)^{1/2})$},\\[2mm]
+\infty & \ \mbox{otherwise}
\end{array}
\right.
$$
with
${\mathcal D}(q_{\psi})={\mathcal D}((d\psi/d\chi)^{1/2})$.
\end{definition}

A few remarks are in order:
%%%%%%%%%%%%%%%%%%%%%%% \label{R-B.3} %%%%%%%%%%%%%%%%%%%%%%
\begin{remark}\label{R-B.3}\rm
\mbox{}
\begin{itemize}
\item[(i)]  
The extended positive part $\widehat{\cM}_+$ is not used in \cite{C3} and slightly
different explanation is given there.
Namely, the quadratic form $q_{\psi}$ $($see \eqref{F-B.4}$)$ is just defined on $D(\cH, \chi)$. 
Then, it is extended to a lower semi-continuous quadratic form on the whole space $\cH$ 
$($based on the idea borrowed from \cite{Si1}$)$.
Its restriction to the ``finite part" $($where finite values are taken$)$ is  
a closable quadratic form due to lower semi-continuity.
Then  the spatial derivative $d\psi/d\chi$ is defined as the positive
self-adjoint operator associated with the closure of this closable form.
Equivalence of the two definitions is explained in \cite[Chap.~III]{Te1}. 

\item[(ii)]
The support $($as a positive self-adjoint operator$)$ of $d\psi/d\chi$ is the support $($as a weight$)$
of $\psi$ $($see \cite[Corollary 12]{C3}$)$.

\item[(iii)]
The intersection $D(\cH, \chi) \cap {\mathcal D}\left((d\psi/d\chi)^{1/2}\right)$ 
is known to be a core for the square root $(d\psi/d\chi)^{1/2}$ $($see
\cite[Chap.~III, Proposition 22, 2)]{Te1}$)$.
\end{itemize}
\end{remark}

We record the following fact (\cite[Proposition 3]{C3}) (that was needed in \S\ref{S-5}):

%%%%%%%%%%%%%%%%%%%% \label{L-B.4} %%%%%%%%%%%%%%%%%%%%%%%%
\begin{lemma}\label{L-B.4}
There exists a family $\{\xi_{\iota}\}_{\iota \in I}$ of vectors in $D(\cH, \chi)$, the set of $\chi$-bounded vectors, 
such that
$$
\sum_{\iota \in I}\theta^{\chi}
(\xi_{\iota}, \xi_{\iota})=1.
$$
\end{lemma}

Let us start from a standard form 
$\langle \cM, \cH, J, {\mathcal P} \rangle$ and 
$\varphi, \varphi_0  \in \cM_*^+$ with $\varphi_0$ faithful.
Let $\xi_{\varphi}, \xi_{\varphi_0}$  be unique implementing vectors in ${\mathcal P}$ for
$\varphi, \varphi_0$ respectively so that $\xi_{\varphi_0}$ is cyclic and separating.
Here we recall the notion of the relative modular operator due to Araki \cite{Ar}.
A densely defined (conjugate linear) operator
$$
S_{\varphi \varphi_0}: \ x\xi_{\varphi_0}  \in \cM \xi_{\varphi_0} \ \longmapsto
\ x^*\xi_{\varphi} \in \cM \xi_{\varphi}
$$
is closable, and its closure admits the polar decomposition of the form 
$$
\overline{S_{\varphi \varphi_0}}=J\Delta_{\varphi\varphi_0}^{1/2}.
$$
The positive self-adjoint operator $\rmvppsi$ here is referred to as the
\emph{relative modular operator} (of $\varphi$ relative to $\psi$) in the literature.
We set
$$
\varphi_0'(y')=(y'\xi_{\varphi_0}, \xi_{\varphi_0}) \quad (y' \in \cM')
$$
so that $\varphi_0' \in \cM'{}_*^{+}$ is the same as $\varphi_0'(y')=\varphi_0(Jy'^*J)$.
We consider the spatial derivative $d\varphi/d\varphi_0'$ in this special setting.
We have of course $n_{\varphi_0'}=\cM'$ and
$\cH_{\varphi_0'}, \Lambda_{\varphi_0'}, \pi_{\varphi_0'}$ can be identified with
$$
\cH_{\varphi_0'}=\cH, \quad  \Lambda_{\varphi_0'}(y')=y'\xi_{\varphi_0}, \quad \pi_{\varphi_0'}(y')\xi=y'\xi.
$$
An operator $R^{\varphi_0'}(\xi)$ (with $\xi \in \cH$) from $\cH$ to 
$\cH_{\varphi_0'}$ becomes an operator on $\cH$ and is
given by 
$$
R^{\varphi_0'}(\xi)(y'\xi_{\varphi_0})=y'\xi \quad (y' \in \cM').
$$
If $R^{\varphi_0'}(\xi)$ is a bounded operator on $\cH$, then $R^{\varphi_0'}(\xi)$ 
sits in $\cM'' =\cM$  (due to \eqref{F-B.2}) 
and this means $\xi=R^{\varphi_0'}(\xi)\xi_{\varphi_0} \in \cM\xi_{\varphi_0}$.
Therefore, we simply have  
\begin{eqnarray*}
&&
D(\cH,\varphi_0')=\cM\xi_{\varphi_0},\\
&&
R^{\varphi_0'}(x\xi_{\varphi_0})=x \quad \mbox{for $x \in \cM$}, \\
&&
\theta^{\varphi_0'}
(x\xi_{\varphi_0},x\xi_{\varphi_0})=R^{\varphi_0'}(x\xi_{\varphi_0})R^{\varphi_0'}(x\xi_{\varphi_0})^*
=xx^*.
\end{eqnarray*}
Therefore, the quadratic form $q_{\varphi}$ (used to
define the spatial derivative $d\varphi/d\varphi_0'$, see \eqref{F-B.4})
is simply
$$
x\xi_{\varphi_0} \in \cM\xi_{\varphi_0} \ \longmapsto \ q_{\varphi}(x\xi_{\varphi_0})=\varphi(xx^*) \in [0,\infty).
$$
We note
$$
\varphi(xx^*)=\|x^*\xi_{\varphi}\|^2 \ \left(=\|S_{\varphi\varphi_0}x\xi_{\varphi_0}\|^2
=\|\Delta_{\varphi\varphi_0}^{1/2}x\xi_{\varphi_0}\|^2\right).
$$
Therefore, the spatial derivative $d\varphi/d\varphi_0'$ in this special case
is nothing but the relative modular operator $\Delta_{\varphi\varphi_0}$, i.e.,
\begin{align}\label{F-B.6} %%%%%%%%%%%%%%%%%%%%%%%%%%%%%%%% \label{F-B.6}
d\varphi/d\varphi_0'=\Delta_{\varphi\varphi_0} \quad (\varphi\in M_*^+).
\end{align}

%%%%%%%%%%%%%%%%%%%%%%%%%%%%%%%%%%%%%%%%%%%%%%%%%
%%%%%%%%%% B.2  Hilsum $L^p$-spaces %%%%%%%%%%%%%%%%%%%%%%%%
%%%%%%%%%%%%%%%%%%%%%%%%%%%%%%%%%%%%%%%%%%%%%%%%%
\subsection{Hilsum's $L^p$-spaces}\label{S-B.2}

Let $\cM$ be a von Neumann algebra acting on $\cH$, and
we fix a faithful semi-finite normal weight $\chi$ on the commutant $\cM'$.
For a semi-finite normal weight $\varphi$ we denote Connes' spatial derivative 
of $\varphi$ relative to $\chi$ by $d\varphi/d\chi$ (see Appendix \S\ref{S-B.1}
for details).
Then, for $p \in [1,\infty)$ \emph{Hilsum's $L^p$-space} $L^p(\cM,\chi)$
consists of all (densely defined) closed operators $T$ on $\cH$ 
whose polar decompositions $T=u|T|$ satisfy $u \in \cM$ and
$|T|^p=d\varphi/d\chi$ with $\varphi \in \cM_*^+$ (see \cite{H}).
 In particular, we have
$$
L^p(\cM,\chi)_+=
\left\{  
(d\varphi/d\chi)^{1/p};\, \varphi \in \cM_*^+
\right\}.
$$
For a functional $\rho \in \cM_*$ with the polar decomposition $\rho=u\varphi$ (with
$\varphi=|\rho| \in \cM_*^+$) as a functional the notations $d\rho/d\chi=ud\varphi/d\chi$ and 
$$
\int d\rho/d\chi\,d\chi=\rho(1) \  (=\varphi(u))
$$
are used in \cite{H}.  Hilsum's $L^p$-space $L^p(\cM,\chi)$ can be identified
with the Haagerup $L^p$-space  $L^p(\cM)$ (see Appendix \ref{S-A}) via
$$
T \quad \longleftrightarrow \quad uh_{\varphi}^{1/p},
$$ 
where $T$ has the polar decomposition $T=u(d\varphi/d\chi)^{1/p}$.
In fact, through this identification all the operations (such as  sums, products and so on)
among  elements in $L^p(\cM,\chi)$ (which are unbounded operators) are justified 
(see \cite[\S 1]{H} and \cite[Chap.~IV]{Te1} for details).
Thus, one can manipulate them like $\tau$-measurable operators.
Via this identification between $L^1(\cM,\chi)$ and $L^1(\cM)$ the above 
$\int\cdot \,d\chi$ corresponds exactly
to the trace-like linear functional $\tr(\cdot)$ (explained in Appendix \ref{S-A}).
By this reason, we will use the the same symbol $\tr$ (instead of the integral notation
$\int\cdot \,d\chi$). For $T \in L^p(\cM,\chi)$ its $L^p$-norm is defined by
$$
\|T\|^p=\bigl(\tr(|T|^p)\bigr)^{1/p}.
$$
Thus, for $T$ with the polar decomposition  $u(d\varphi/d\chi)^{1/p}$
we have $\|T\|^p=\varphi(1)^{1/p}$.

In particular, when $\cM$ is represented in a standard form $\<\cM,\cH,J,\mathcal P\>$ and
$\chi=\varphi_0'$ with a faithful $\ffi_0\in\cM_*^+$ (where $\varphi_0'(x')=\ffi_0(Jx'^*J)$),
in view of \eqref{F-B.6} we can write
\begin{align}\label{F-B.7} %%%%%%%%%%%%%%%%%%%%%%%%%%%%%%%% \label{F-B.7}
L^p(\cM,\varphi_0')_+=\{\Delta_{\varphi\varphi_0}^{1/p};\, \ffi\in\cM_*^+\},
\end{align}
and $\Delta_{\varphi\varphi_0}^{1/p}=(d\varphi/d\varphi_0')^{1/p}$ corresponds to
$h_\varphi^{1/p}$ in Haagerup's $L^p$-space $L^p(\cM)$.

The well-known duality between $\ellp$ and $\ellq$ with $1/p+1/q=1$ is given by the bilinear
form $\langle \cdot,\cdot \rangle$ on $\ellp \times \ellq$ defined by
\begin{align}\label{F-B.8} %%%%%%%%%%%%%%%%%%%%%%%%%%%%%%%% \label{F-B.8}
\langle A, B \rangle \,(=\langle A, B \rangle_{p,q})=\mbox{tr}(AB)
\quad\mbox{for $A\in\ellp$, $B\in\ellq$}
\end{align}
(which of course corresponds to \eqref{F-A.5}).
Here, suffixes $p,q$ may be omitted when no confusion is possible.
The inner product in the $L^2$-space $L^2(\cM,\psi')$ is given by
$$
(A,B)=\mbox{tr}(AB^*).
$$
We note that the quadruple 
$$
\langle
\pi_{\ell}(\cM), \ L^2(\cM,\psi'), \ {}^*, \ L^2(\cM,\psi')_+
\rangle
$$
(where $\pi_{\ell}$ means the left multiplication) is a standard form.

%%%%%%%%%%%%%%%%%%%%%%%%%%%%%%%%%%%%%%%%%%%%%%%%%%%
%%% B.3  Canonical correspondence of  $P(\cM, \bC)$ and $P(B(\cH), \cM')$} %%%%%%%%%%%
%%%%%%%%%%%%%%%%%%%%%%%%%%%%%%%%%%%%%%%%%%%%%%%%%%%
\subsection{Canonical correspondence of $P(\cM, \bC)$ and $P(B({\mathcal H}), \cM')$}
\label{S-B.3}

For convenience we will use the following notations:
\begin{eqnarray*}
&&
P(\cM, \bC):=\mbox{the set of all faithful semi-finite normal weights on $\cM$},\\
&&
P(B(\cH), \cM'):=\mbox{the set of all faithful semi-finite normal}\\
&&    \hskip 4cm    \mbox{operator valued weights from $B(\cH)$ to $\cM'$}.
\end{eqnarray*}
A canonical order reversing  correspondence between 
$P(\cM, \bC)$ and $P(B(\cH), \cM')$ 
was constructed in \cite{C3}, and brief explanation on this correspondence  is presented.

We take faithful semi-finite normal weights $\phi,\phi_1$ on $\cM$
and  a faithful semi-finite normal weight $\chi$  on $\cM'$.
We will repeatedly use modular properties of spatial derivatives (see \cite{C3}):
\begin{itemize}
\item[(a)] 
$(d\phi/d\chi)^{-1}=d\chi/d\phi$,
\item[(b)]
We have
\begin{eqnarray*}
&& 
\sigma^{\phi}_t(x)=(d\phi/d\chi)^{it}x(d\phi/d\chi)^{-it}  \ \mbox{for} \ x \in \cM,\\ 
&&
\sigma^{\chi}_t(y')=(d\phi/d\chi)^{-it}y'(d\phi/d\chi)^{it}
\ \left(=  (d\chi/d\phi)^{it}y'(d\chi/d\phi)^{-it}  \right)
 \ \mbox{for} \ y' \in \cM',
\end{eqnarray*}
\item[(c)] $(D\phi:D\phi_1)_t=(d\phi/d\chi)^{it}(d\phi_1/d\chi)^{-it}$.
\end{itemize}

Firstly we start from $\phi \in P(\cM, \bC)$
(by fixing $\chi$ for a moment). We set $\omega=\Tr((d\chi/d\phi) \, \cdot)$
(with the standard trace on $B(\cH)$ and the density operator $d\chi/d\phi$) 
which is a weight on $B(\cH)$. We note
$$
\sigma_t^{\omega}=\Ad \, (d\chi/d\phi)^{it} \quad \mbox{and} \quad
\sigma_t^{\omega}\vert_{\cM'}=\sigma_t^{\chi}
$$
(as was mentioned in the above (b)).
Thus, there is a
unique operator valued weight $E  \in P(B(\cH), \cM')$ satisfying
$\omega=\chi\circ E$ (\cite[Theorem 5.1]{Haa3}). We note that $E$ does not depend on $\chi$.  
Indeed, for another faithful semi-finite normal weight $\chi'$ on $\cM'$ we have
$$
(D(\chi'\circ E):D(\chi\circ E))_t=(D\chi':D\chi)_t=(d\chi'/d\phi)^{it}(d\chi/d\phi)^{-it}.
$$
This means that the density operator for the weight $\chi'\circ E$ on $B(\cH)$ is $d\chi'/d\phi$,
that is, $\chi'\circ E=\Tr((d\chi'/d\phi) \, \cdot)$ and $\chi'$ gives us the same $E$. 
Let us use the notation $E=\phi^{-1}$.  
Discussions so far means that the defining
property for $\phi^{-1} \in P(B(\cH), \cM')$ is
\begin{equation}\label{F-B.9} %%%%%%%%%%%%%%%%%%%%%%%%%%%%%% \label{F-B.9}
\chi \circ \phi^{-1}=\Tr((d\chi/d\phi)\,\cdot)
\end{equation} 
for each faithful semi-finite normal weight $\chi$ on $\cM'$. 
Note that the value of \eqref{F-B.9} against a rank-one operator $\xi \otimes \xi^{c}$ is
$$
\Tr\left((d\chi/d\phi)  (\xi \otimes \xi^{c}) \right)
=\|(d\chi/d\phi)^{1/2} \xi \|^2=\chi\left( \theta^{\phi}(\xi,\xi) \right)
$$
(see \eqref{F-B.4}) so that \eqref{F-B.9} means 
$$
\phi^{-1}(\xi \otimes \xi^{c})=\theta^{\phi}(\xi,\xi).
$$

This time we start from $F \in P(B(\cH), \cM')$. We set $\omega'=\chi\circ F$.
This is a weight on $B(\cH)$ and of the form $\Tr(K\,\cdot)$ with some density operator $K$.
We have
$$
\sigma_t^{\omega'}=\Ad\, K^{it} \quad \mbox{and} \quad
\sigma_t^{\omega'}(y')=\sigma_t^{\chi}(y')=(d\phi/d\chi)^{-it}y'(d\phi/d\chi)^{it} \quad \mbox{for $y' \in \cM'$},
$$
and hence $D_t=K^{-it}(d\phi/d\chi)^{-it}$ falls into $\cM''=\cM$. We also note
\begin{align*}
D_{t+s}&=K^{-it}K^{-is}(d\phi/d\chi)^{-is}(d\phi/d\chi)^{-it}\\
&=K^{-it}(d\phi/d\chi)^{-it}(d\phi/d\chi)^{it}K^{-is}(d\phi/d\chi)^{-is}(d\phi/d\chi)^{-it}
=D_t\sigma_t^{\phi}(D_s).
\end{align*}
This means that $D_t$ is a $\sigma^{\phi}$-cocycle so that 
$$
D_t=(D\psi:D\phi)_t \left(=(d\psi/d\chi)^{it} (d\phi/d\chi)^{-it} \right)
$$ 
for some faithful semi-finite normal weight $\psi$ on $\cM$ (\cite[Th\'{e}or\`{e}me 1.2.4]{C1}),
showing $K=(d\psi/d\chi)^{-1}=d\chi/d\psi$ and
$$
\chi\circ F=\omega' \ \bigl(=\Tr(K\,\cdot) \bigr) =\Tr((d\chi/d\psi)\, \cdot).
$$
This expression shows that $\psi$ is determined by $F$ and $\chi$.
However, $\psi$ does not depend on $\chi$. Indeed, for another faithful semi-finite normal
weight $\chi'$ on $\cM'$ we have
$$
(D(\chi'\circ F): D(\chi\circ F))_t=(D\chi':D\chi)_t=(d\chi'/d\psi)^{it}(d\chi/d\psi)^{-it},
$$ 
showing $\chi'\circ F=\Tr((d\chi'/d\psi)\, \cdot)$.
Let us denote the above $\psi \in P(\cM, \bC)$ by $F^{-1}$.
Discussions so far mean that the defining property of $F^{-1}$ is
\begin{equation}\label{F-B.10} %%%%%%%%%%%%%%%%%%%%%%%%%%%%% \label{F-B.10} 
\chi \circ F=\Tr((d\chi/dF^{-1})\, \cdot)
\end{equation}
for each faithful semi-finite normal weight $\chi$ on $\cM'$.

It is easy to see $(\phi^{-1})^{-1}=\phi$ for $\phi \in P(\cM, \bC)$
and $(F^{-1})^{-1}=F$ for $F \in P(B(\cH), \cM')$ 
by repeated use of the defining properties \eqref{F-B.9} and \eqref{F-B.10}. 
Indeed, we have
\begin{eqnarray*}
&&
\Tr\left(\left( d\chi/d((\phi^{-1})^{-1})
\right) \cdot \right)
=\chi\circ \phi^{-1}=\Tr\left(\left(d\chi/d\phi\right) \cdot \right),\\
&&
\chi\circ (F^{-1})^{-1}=\Tr\left(\left( d\chi/dF^{-1}
\right) \cdot \right)=\chi\circ F.
\end{eqnarray*}
The properties \eqref{F-B.9} and \eqref{F-B.10}
also clearly show that taking the inverse is  order-reversing, which can be also seen from
\begin{align}\label{F-B.11} %%%%%%%%%%%%%%%%%%%%%%%%%%%%% \label{F-B.11}
(D\phi_1^{-1}: D\phi_2^{-1})_t
&= (D(\chi \circ\phi_1^{-1}): D(\chi\circ\phi_2^{-1}))_t
\nonumber\\
&= (d\chi/d\phi_1)^{it}(d\chi/d\phi_2)^{-it} \quad (\mbox{by \eqref{F-B.9}})
\nonumber\\
&= (d\phi_1/d\chi)^{-it}(d\phi_2/d\chi)^{it}
= (D\phi_1:D\phi_2)_{-t}
\nonumber\\
&= (D\phi_2:D\phi_1)_{-t}^*
\end{align}
for $\phi_1, \phi_2 \in P(\cM, \bC)$.
%%%%%%%%%%%%%%%%%%%%% \label{R-B.5} %%%%%%%%%%%%%%%%%%%%%%%%
\begin{remark}\label{R-B.5}\rm
Let us clarify the reason why order-reversing property can be seen from \eqref{F-B.11}.
We set $I_{-1/2}=\{z \in \bC;\,  -1/2\leq\Im z \leq 0 \}$ and 
$I_{1/2}=\{z \in \bC;\,  0 \leq \Im z \leq 1/2 \}$
for convenience. Then we observe$:$
\begin{quote}
$(D\phi_1^{-1}: D\phi_2^{-1})_t$ extends
to a bounded  continuous function on $I_{-1/2}$
which is analytic in the interior
if and only if  $(D\phi_2:D\phi_1)_t$ has the same extension property. Furthermore, when this extension
property is satisfied, we have
\begin{equation}\label{F-B.12} %%%%%%%%%%%%%%%%%%%%%%%%%%%%%% \label{F-B.12}
(D\phi_2:D\phi_1)_{z}=(D\phi_1^{-1}: D\phi_2^{-1})_{-\overline{z}}^*
\quad (\mbox{for $z \in I_{-1/2}$})
\end{equation}
and in particular $(D\phi_2:D\phi_1)_{-i/2}=(D\phi_1^{-1}: D\phi_2^{-1})_{-i/2}^*$.
\end{quote}
Therefore, the order reversing property holds true thanks to \cite[Th\'eor\`eme 3]{C2}.

To see the above extension property, let us assume that
$(D\phi_1^{-1}: D\phi_2^{-1})_t$ has the extension $(D\phi_1^{-1}: D\phi_2^{-1})_z$ $($$z \in I_{-1/2}$$)$
stated above.
The  obvious computation
$$
((D\phi_1^{-1}: D\phi_2^{-1})_{\overline{z}}^*\,\xi,\zeta)=(\xi,(D\phi_1^{-1}: D\phi_2^{-1})_{\overline{z}}\,\zeta)
=\overline{((D\phi_1^{-1}: D\phi_2^{-1})_{\overline{z}}\,\zeta,\xi)}
$$
$($for vectors $\xi,\zeta$$)$ shows that
$(D\phi_1^{-1}: D\phi_2^{-1})_{\overline{z}}^*$ is a bounded continuous function
on $\overline{I_{-1/2}}=I_{1/2}$ which is analytic in the interior.
Therefore,
$(D\phi_1^{-1}: D\phi_2^{-1})_{-\overline{z}}^*$
is a bounded continuous function on $-I_{1/2}=I_{-1/2}$ which is analytic in the interior.
The value of this function for $z=t \in \bR$ is
$(D\phi_1^{-1}: D\phi_2^{-1})_{-t}^*=(D\phi_2:D\phi_1)_{t}$ thanks to \eqref{F-B.11}.
This means that $(D\phi_2:D\phi_1)_t$ extends to a bounded  continuous function on $I_{-1/2}$
which is analytic in the interior and we have \eqref{F-B.12}.
Conversely, when $(D\phi_2:D\phi_1)_t$ has the extension $(D\phi_2:D\phi_1)_z$
$($$z \in I_{-1/2}$$)$ as stated above,
by similar arguments as above 
$(D\phi_2:D\phi_1)_{-\overline{z}}^*$ is a bounded continuous function on $I_{-1/2}$ 
which is analytic in the interior. 
Also for $z=t \in \bR$ the value of this function is
$(D\phi_2:D\phi_1)_{-t}^*=(D\phi_1^{-1}: D\phi_2^{-1})_t$,
that is, we have the extension $(D\phi_1^{-1}: D\phi_2^{-1})_z$ and
$(D\phi_1^{-1}: D\phi_2^{-1})_z=(D\phi_2:D\phi_1)_{-\overline{z}}^*$ $($$z \in I_{-1/2}$$)$.
This equation is of course equivalent to \eqref{F-B.12}.

In \cite{Haa3} an order-reversing bijection $($called $\alpha$$)$ from 
$P(\cM, \bC)$ onto $P(B(\cH), \cM')$ was 
constructed in a less canonical fashion, 
and the formula $(D\alpha(T_1):D\alpha(T_2))_t=(DT_2:DT_1)_t$ appears in \cite[p.~360]{Haa3}.
This is an obvious misprint, and $(D\alpha(T_1):D\alpha(T_2))_t=(DT_2:DT_1)_{-t}^*$ is the 
correct formula.
\end{remark}

%%%%%%%%%%%%%%%%%%%%%%%%%%%%%%%%%%%%%%%%%%%%%%%%%%
%%% Appendix C Infima of  decreasing sequences of weights  %%%%%%%%%%%%%%%%%%
%%%%%%%%%%%%%%%%%%%%%%%%%%%%%%%%%%%%%%%%%%%%%%%%%%
\section{Infima of decreasing sequences of weights}\label{S-C}

In the main body of the article we have encountered decreasing sequences of
various operators such as
$\tau$-measurable operators and positive elements in non-commutative $L^p$-spaces 
(see Theorem \ref{T-3.8} and Proposition \ref{P-4.2} for instance). 
More generally decreasing sequences of unbounded positive self-adjoint operators 
(or rather positive forms in the sense of \S\ref{S-2}) were considered in \cite{Ko6} 
(motivated by \cite {Si1}).
Here, for the sake of completeness we will study decreasing sequences 
of weights on a von Neumann algebra,
and probably no such study has  been made in literature. 

Let $\cM$ be a von Neumann algebra acting on a Hilbert space $\cH$
and $\chi$ be a (fixed) faithful  semi-finite normal weight on the commutant $\cM'$.
The modular automorphism group on $\cM'$ induced by $\chi$ will be denoted by
$\{\sigma_t'\}_{t \in \bR}$.

%%%%%%% \label{D-C.1} %%%%%%%%%%%%%%%%%%%%%%%%%%%%%%%%%%
\begin{definition}[\cite{C3,H,Te1}]\label{D-C.1}\rm
Assume $\gamma \in \bR$. A densely defined closed operator $T$ on $\cH$ with the polar
decomposition $T=u|T|$ is said to be \emph{$\gamma$-homogeneous} (relative to $\chi$) when
$$
\mbox{$u \in \cM$ and $\sigma_{\gamma t}'(y')|T|^{it}=|T|^{it}y'$
for each $t \in \bR, \ y' \in \cM'$.}
$$
\end{definition}

The next characterization is proved based on Carlson's theorem
for entire functions of exponential type  (see \cite[Lemma 2.1]{Ko1} and \cite[p.~338]{Te2}).
For instance it is difficult to see that the sum of $\gamma$-homogeneous operators 
possesses the same homogeneity from Definition \ref{D-C.1} whereas
this property can be immediately seen from the characterization in the next lemma. 

%%%%%%%%%%%%%%%%%%% \label{L-C.2} %%%%%%%%%%%%%%%%%%%%%%%%%
\begin{lemma}\label{L-C.2}
A densely defined closed operator $T$ on $\cH$ is $\gamma$-homogeneous if and only if
$$
y'T \subseteq T\sigma_{i\gamma}'(y')
$$
for each $y' \in \cM'$ analytic with respect to the modular automorphism group $\sigma_t'$.
\end{lemma}
This characterization for $\gamma$-homogeneity will be used to prove the next result.
%%%%%%%%%%%%%%%%%%% \label{L-C.3} %%%%%%%%%%%%%%%%%%%%%%%%%
\begin{lemma}\label{L-C.3}
Let $T_n$, $n=1,2,\cdots$, be a sequence of positive self-adjoint operators converging to
a positive self-adjoint operator $T$ in the strong resolvent sense
$($i.e., $(T_n+1)^{-1} \rightarrow (T+1)^{-1}$ in the strong operator topology$)$. 
If each $T_n$ is $\gamma$-homogeneous, then so is $T$.
\end{lemma}
\begin{proof}
We choose and fix $\xi \in {\mathcal D}(T)$ and set
$$
\xi_n=(T_n+1)^{-1}(T+1)\xi \in {\mathcal D}(T_n+1) = {\mathcal D}(T_n)
$$
for each $n=1,2,\cdots$.
Since $(T_n+1)^{-1} \rightarrow (T+1)^{-1}$ strongly, we have
\begin{equation} \label{F-C.1} %%%%%%%%%%%%%%%%%%%%%%%%%%%%%% \label{F-C.1}
\lim_{n \to \infty}\|\xi_n-\xi \|=\lim_{n \to \infty}\|\left((T_n+1)^{-1}-(T+1)^{-1}\right)(T+1)\xi \|=0.
\end{equation}
Note $T_n\xi_n=(T_n+1)\xi_n-\xi_n=(T+1)\xi-\xi_n$ and hence \eqref{F-C.1} also means
\begin{equation}\label{F-C.2} %%%%%%%%%%%%%%%%%%%%%%%%%%%%%% \label{F-C.2}
\lim_{n \to \infty}\|T_n\xi_n-T\xi \|=0.
\end{equation}

We take an analytic element $y' \in \cM'$ with respect to the modular 
automorphism group $\sigma_t'$.
Since $y'T_n \subseteq T_n\sigma_{i\gamma}'(y')$ by Lemma \ref{L-C.2}, 
we have  $\sigma_{i\gamma}'(y')\xi_n \in {\mathcal D}(T_n)$ and 
$y'T_n\xi_n =T_n\sigma_{i\gamma}'(y')\xi_n$.
Since
$$
T_n\sigma_{i\gamma}'(y')\xi_n-y'T\xi=y'T_n\xi_n-y'T\xi=y'\left(T_n\xi_n-T\xi \right),
$$
\eqref{F-C.2} implies
\begin{equation}\label{F-C.3} %%%%%%%%%%%%%%%%%%%%%%%%%%%%%% \label{F-C.3}
\lim_{n \to \infty}\|T_n\sigma_{i\gamma}'(y')\xi_n-y'T\xi\|=0.
\end{equation}
Since $\sigma_{i\gamma}'(y')\xi_n \in {\mathcal D}(T_n)$, we can set
$$
\eta_n=(T+1)^{-1}(T_n+1)\sigma_{i\gamma}'(y')\xi_n
\in {\mathcal D}(T+1) = {\mathcal D}(T).
$$
We observe 
$$
\eta_n-\sigma_{i\gamma}'(y')\xi_n
=\left((T+1)^{-1}-(T_n+1)^{-1} \right)
(T_n+1)\sigma_{i\gamma}'(y')\xi_n.
$$
We set
\begin{align*}
\zeta_n&=(T_n+1)\sigma_{i\gamma}'(y')\xi_n-\bigl(y'T\xi+\sigma_{i\gamma}'(y')\xi\bigr)\\
&=\bigl( T_n\sigma_{i\gamma}'(y')\xi_n-y'T\xi\bigr)
+\sigma_{i\gamma}'(y')\left(\xi_n-\xi\right) .
\end{align*}
The second expression, \eqref{F-C.3} and \eqref{F-C.1} guarantee
\begin{equation}\label{F-C.4} %%%%%%%%%%%%%%%%%%%%%%%%%%%%%% \label{F-C.4}
\lim_{n \to \infty}\|\zeta_n\|=0.
\end{equation}
By using this $\zeta_n$ we rewrite the above $\eta_n-\sigma_{i\gamma}'(y')\xi_n$ as follows:
$$
\eta_n-\sigma_{i\gamma}'(y')\xi_n
=\left((T+1)^{-1}-(T_n+1)^{-1} \right)
\bigl(\bigl(y'T\xi+\sigma_{i\gamma}'(y')\xi\bigr) +\zeta_n
\bigr).
$$
We estimate
\begin{align*}
\|\eta_n-\sigma_{i\gamma}'(y')\xi_n\|
&\leq \|\left((T+1)^{-1}-(T_n+1)^{-1} \right)
\bigl( y'T\xi+\sigma_{i\gamma}'(y')\xi\bigr)\|\\
& \hskip 3.5cm
+\|(T+1)^{-1}-(T_n+1)^{-1}\|\,\|\zeta_n\|\\
&\leq \|\left((T+1)^{-1}-(T_n+1)^{-1} \right)
\bigl( y'T\xi+\sigma_{i\gamma}'(y')\xi\bigr)\|+2\|\zeta_n\|
\end{align*}
so that the strong resolvent convergence of $T_n$ to $T$ and \eqref{F-C.4} show
\begin{equation}\label{F-C.5} %%%%%%%%%%%%%%%%%%%%%%%%%%%%%% \label{F-C.5}
\lim_{n \to \infty}
\|\eta_n-\sigma_{i\gamma}'(y')\xi_n\|
=0,
\end{equation}
that is, 
\begin{equation}\label{F-C.6} %%%%%%%%%%%%%%%%%%%%%%%%%%%%%% \label{F-C.6}
\lim_{n \to \infty}
\|\eta_n-\sigma_{i\gamma}'(y')\xi\|
=0
\end{equation}
(see \eqref{F-C.1}).

On the other hand, we have
\begin{align*}
T\eta_n-y'T\xi
&=(T+1)\eta_n-\eta_n-y'T\xi\\
&=(T_n+1)\sigma_{i\gamma}'(y')\xi_n-\eta_n-y'T\xi\\
&=\bigl( T_n\sigma_{i\gamma}'(y')\xi_n-y'T\xi\bigr)
+
\bigl(\sigma_{i\gamma}'(y')\xi_n-\eta_n \bigr),
\end{align*}
which enables us to conclude
\begin{equation}\label{F-C.7} %%%%%%%%%%%%%%%%%%%%%%%%%%%%%% \label{F-C.7}
\lim_{n \to \infty}\|T\eta_n-y'T\xi\|=0
\end{equation}
due to \eqref{F-C.5} and \eqref{F-C.3}.

Each $\eta_n$ sits in ${\mathcal D}(T)$ so that the convergences \eqref{F-C.6} and \eqref{F-C.7} mean
$$
(\sigma_{i\gamma}'(y')\xi, y'T\xi) \in \overline{\Gamma(T)}, 
\ \mbox{the closure of the graph of $T$ (in $\cH \oplus \cH$}).
$$
Therefore, from the closedness of $T$ we conclude $\sigma_{i\gamma}'(y')\xi \in {\mathcal D}(T)$ with 
$T\sigma_{i\gamma}'(y')\xi =y'T\xi$ and we are done.
\end{proof}

If all of $T_n$'s and the limiting operator $T$ are assumed to be non-singular, then 
Lemma \ref{L-C.3} is much easier to prove.
Indeed, in this situation one can just use the condition (involving unitary operators $T_n^{it}$, $T^{it}$)
in Definition \ref{D-C.1}.
Then, since $T_n^{it}=\exp(it\log T_n)$ and $T^{it}=\exp(it\log T)$, one can use Trotter's theorem (see
\cite[Theorem VIII.21]{RS}) as in the proof of \cite[Corollary 15]{C3}.

In \cite{C3} an increasing sequence $\{\omega_n\}$ of faithful normal weights was studied.
The key point here is that the point-wise supremum $\omega=\sup_n \omega_n$
(i.e., $\omega(x)=\sup_n \omega_n(x)$, $x \in \cM_+$)
is normal. 
Indeed, normality means lower semi-continuity in the $\sigma$-weak topology, and the supremum
of lower semi-continuous functions is lower semi-continuous.
When $\omega$ is semi-finite, it is known that
spatial derivatives $d\omega_n/d\chi$ (increasingly) tend to $d\omega/d\chi$ in the
strong resolvent sense and hence corresponding modular automorphisms behave 
in the expected way (see \cite[Corollary 15]{C3} for details).
We will deal with a decreasing sequence $\{\omega_n\}_n$ of normal weights instead in the
next theorem. (The theorem is used in Remark \ref{R-5.8},(iii) of \S\ref{S-5}.)  We assume
semi-finiteness of $\omega_1$ (and hence all $\omega_n$'s), but faithfulness is not assumed
(and hence Lemma \ref{L-C.3} is needed).
The point-wise infimum $\inf_n \omega_n$ is certainly a weight (in the algebraic sense), but 
one cannot expect lower-semicontinuity in the $\sigma$-weak topology (i.e., normality),
see Remark \ref{R-C.5},(ii).  

%%%%%%%%%%%%%%%%%%% \label{T-C.4} %%%%%%%%%%%%%%%%%%%%%%%%%%
\begin{theorem}\label{T-C.4}
Let $\{\omega_n\}_{n \in \bN}$ be a decreasing sequence of semi-finite normal weights on $\cM$. 
There exists a semi-finite normal weight $\omega$ such that spatial derivatives
$d\omega_n/d\chi$, $n=1,2,\cdots$,  converge $($decreasingly$)$ to $d\omega/d\chi$
in the strong resolvent sense.
Furthermore, $\omega$ is the maximum of all semi-finite normal weights majorized by
the point-wise infimum $\inf_n \omega_n$.
\end{theorem}

Our proof is based on the order preserving one-to-one correspondence.
$$
\varphi \in P_0(\cM, \bC) \quad \longleftrightarrow \quad d\varphi/d\chi \in \overline{\cM^+_{-1}}
$$
explained in \S\ref{S-5.1} (see \eqref{F-5.1}).
\begin{proof}
We have the following decreasing sequence of spatial derivatives:
$$
d\omega_1/d\chi \geq d\omega_2/d\chi \geq d\omega_3/d\chi \geq \cdots.
$$
We set $T=\Inf_n d\omega_n/d\chi$ so that the sequence $d\omega_n/d\chi$ converges to $T$
in the strong resolvent sense (see \cite{Ko6,Si1} for details).
Since each $d\omega_n/d\chi$ is $(-1)$-homogeneous, so is $T$ thanks to Lemma \ref{L-C.3}.
Thus, we have $T=d\omega/d\chi$ for some semi-finite normal weight $\omega$.
Since $d\omega_1/d\chi \geq d\omega_2/d\chi \geq d\omega_3/d\chi \geq \cdots \geq T=d\omega/\chi$,
we know $\omega_n \geq \omega$.

Conversely, we assume that a semi-finite normal weight $\omega'$ satisfies $\inf_n\omega_n \geq \omega'$.
Then, we have $\omega_n \geq \omega'$ for each for each $n$.
Hence we have $d\omega_n/d\chi \geq d\omega'/d\chi$ for each $n$ so that
the point-wise infimum $\inf_n$ satisfies
$$
\inf_nd\omega_n/d\chi \geq d\omega'/d\chi.
$$
Thus, the maximality of $\Inf_n d\omega_n/d\chi \ (=T)$ among all positive forms
majorized by $\inf_n$ (see the paragraph before Theorem \ref{T-2.3}) implies 
$$
\Inf_n d\omega_n/d\chi \ (=d\omega/d\chi) \geq d\omega'/d\chi.
$$
This means $\omega \geq \omega'$, showing the desired maximality of $\omega$.
\end{proof}

A few remarks on the theorem are in order:
%%%%%%%%%%%%%%%%%%%%%%\label{R-C.5} %%%%%%%%%%%%%%%%%%%%%%%
\begin{remark}\label{R-C.5}\rm
\mbox{}
\begin{itemize}

\item[(i)]
The weight $\omega$ mentioned in the theorem does not depend upon the choice of
a fathful semi-finite normal weight $\chi$ on the commutant $\cM'$
$($which can be seen from the latter statement of the theorem$)$. 
Let us use the notation
$$
\omega=\Inf_n \omega_n,
$$ 
and the proof of Theorem \ref{T-C.4} says 
$$
d\left(\Inf_n \omega_n\right)/d\chi=\Inf_n \left(d\omega_n/d\chi \right).
$$

\item[(ii)]
When $\{\omega_n\}_{n \in \bN}$ is a decreasing sequence in $\cM_*^+$, 
$\omega_n$ tends to $\omega \in \cM_*^+$
in norm. We have $\omega_n(x) \searrow \omega(x)$ for $x \in \cM_+$ and
$\omega(x)=\inf_n\omega_n(x)$.
However, generally the point-wise infimum $\inf_n \omega_n$ and the weight $\omega$ in the theorem
could be quite different. 
This phenomenon can be  seen by considering the special case $\cM=B(\cH)$.
In this case a semi-finite normal weight is of the form $\Tr(T\cdot)$ with a densely defined
positive self-adjoint operator $T$ and the standard trace  $\Tr$ of $B(\cH)$.
The value of this weight against a rank-one operator $\xi \otimes \xi^c$ is
$$
\Tr(T(\xi \otimes \xi^c))=\|T^{1/2}\xi\|^2 \ \bigl(=q_T(\xi)\bigr),
$$
i.e., the canonical quadratic form associated with $T$. There are many examples of decreasing sequences
$T_1 \geq T_2 \geq T_3 \geq \cdots$ of densely defined closed quadratic forms such that
the point-wise infimum $\inf_n q_{T_n}$ is not closable, i.e., $\inf_n q_{T_n} \gneqq \Inf_n q_{T_n}$
$($see \cite{Ko6, Si1}$)$. 
\end{itemize}
\end{remark}

Let us recall the canonical order reversing correspondence
$$
\varphi \in P(\cM,\bC)
\quad
\longleftrightarrow 
\quad
\varphi^{-1} 
\in
P(B(\cH), \cM')
$$
explained in \S\ref{S-B.3}. We will interpret the meaning of $\Inf_n \omega_n$ in terms of this
correspondence.

For simplicity let us start from the following situation:
$$
\omega_1 \geq  \omega_2 \geq \omega_3 \geq \cdots \geq \omega_0
$$
with some faithful normal weight $\omega_0$.
Then we obviously have $\inf_n\omega_n \geq \omega_0$ and the
maximality in the preceding theorem shows $\Inf_n \omega_n \geq \omega_0$.
Therefore, faithfulness of all $\omega_n$'s and also 
$\Inf_n \omega_n$ are guaranteed (which is the only role played by $\omega_0$).
Then, by using the above-explained order reversing one-to-one correspondence,
we get the increasing sequence  
$$
\omega_1^{-1} \leq  \omega_2^{-1} \leq \omega_3^{-1} \leq \cdots \leq \omega_0^{-1}.
$$
of faithful semi-finite normal operator valued weights from $B(\cH)$ to $\cM'$.
%%%%%%%%%%%%%%%%%%%% \label{L-C.6} %%%%%%%%%%%%%%%%%%%%%%%%%
\begin{lemma}\label{L-C.6}
Under the above-explained situation the point-wise supremum  
$
\sup_n \omega_n^{-1}
$
of the increasing sequence $\{\omega_n^{-1}\}_{n \in \bN}$ in $P(B(\cH),\cM')$
$($which is obviously an operator valued weight from $B(\cH)$ to $\cM'$$)$
is normal.
\end{lemma}
\begin{proof}
Let us assume $x_{\iota} \nearrow x_0$ in $B(\cH)_+$ and 
$$
\sup_n\omega_n^{-1}(x_{\iota}) \nearrow y_0 \, \left(\leq \sup_n \omega_n^{-1}(x_0) \right)
$$
taken in the extended positive part $\widehat{\cM'}_+$.
For a fixed $\iota$ we have $\sup_n\omega_n^{-1}(x_{\iota}) \leq y_0$ and consequently
$$
\sup_n \varphi \circ \omega_n^{-1}(x_{\iota})
=\varphi\left( \sup_n \omega_n^{-1}(x_{\iota})\right) \leq \varphi(y_0)
$$
for each normal weight $\varphi$ on $\cM'$ (which extends to $\widehat{\cM'}_+$,
see \cite[Proposition 1.10]{Haa3}).
Therefore, we have
\begin{equation}\label{F-C.8} %%%%%%%%%%%%%%%%%%%%%%%%%%%%%% \label{F-C.8}
\sup_{\iota}\left(\sup_n \varphi \circ \omega_n^{-1}(x_{\iota})\right) \leq \varphi(y_0).
\end{equation}
On the other hand, $\{ \varphi \circ \omega_n^{-1}\}_{n \in \bN}$ is an increasing sequence 
of normal weights on $B(\cH)$ so that $\sup_n \varphi \circ \omega_n^{-1}$ is normal
(see the paragraph right before Theorem \ref{T-C.4}).
Thus, the left hand side of \eqref{F-C.8} is actually $\sup_n \varphi \circ \omega_n^{-1}(x_0)$
and we have
$$
\sup_n \varphi \circ \omega_n^{-1}(x_0) \leq \varphi(y_0).
$$
The left hand side here is equal to 
$\varphi\left(\sup_n \omega_n^{-1}(x_0) \right)$ by \cite[Proposition 1.10]{Haa3}
and the above inequality is valid for each $\varphi$ so that we conclude  
$$
\sup_n \omega_n^{-1}(x_0) \leq y_0.
$$
\end{proof}

The following fact is probably worth pointing out, that is considered as a weight version of the formula
\eqref{F-2.5} in \S\ref{S-2}:

%%%%%%%%%%%%%%%%%%%%%% \label{P-C.7} %%%%%%%%%%%%%%%%%%%%%%
\begin{proposition}\label{P-C.7}
Let $\{\omega_n\}_{n \in \bN}$ be a decreasing sequence of semi-finite normal weights on $\cM$,
and we assume
$$
\omega_1 \geq  \omega_2 \geq \omega_3 \geq \cdots \geq \omega_0
$$
with some faithful normal weight $\omega_0$.
Then, the semi-finite normal weight $\omega=\Inf_n \omega_n$ in Theorem \ref{T-C.4}
$($i.e., the maximum of all semi-finite normal weights
majorized by the point-wise infimum $\inf_n\omega_n$$)$ is given by
$$
\Inf_n \omega_n = \bigl(\sup_n \omega_n^{-1}\bigr)^{-1}.
$$
\end{proposition}
\begin{proof}
It suffices to show $\left(\Inf_n \omega_n\right)^{-1} = \sup_n \omega_n^{-1}$.
For each $n$ we have
$$
\chi \circ \omega_n^{-1}=\Tr\left((d\chi/d\omega_n)\, \cdot \right)
$$
thanks to \eqref{F-B.9}. Therefore, we have
\begin{align}\label{F-C.9} %%%%%%%%%%%%%%%%%%%%%%%%%%%%%% \label{F-C.9}
\chi\circ\left(\sup_n \omega_n^{-1} \right)
&=\sup_n \chi \circ \omega_n^{-1} \quad (\mbox{since $\chi$ is normal})
\nonumber\\
&= \sup_n \Tr \left((d\chi/d\omega_n)\, \cdot\right)
= \Tr \left(\left(\sup_n d\chi/d\omega_n\right) \cdot\right). 
\end{align}
We observe that the density in the above right side is
\begin{align*}
\sup_n d\chi/d\omega_n
&= \sup_n \left( d\omega_n/d\chi \right)^{-1}
= \left(\Inf_n d\omega_n/d\chi \right)^{-1}\\
&= \left(d\left(\Inf_n \omega_n\right)/d\chi \right)^{-1}
= d\chi/d\left(\Inf_n \omega_n\right).
\end{align*}
Here, the second equality comes from \eqref{F-2.5}
while the third is just the definition of $\Inf_n \omega_n$ (see
Remark \ref{R-C.5},(i)).
Therefore, from \eqref{F-C.9} and then \eqref{F-B.9} we have
$$
\chi\circ\left(\sup_n \omega_n^{-1} \right)
=\Tr\left( d\chi/d\left(\Inf_n \omega_n\right) \cdot \right)
=\chi \circ \left(\Inf_n \omega_n\right)^{-1},
$$
showing  $\left(\Inf_n \omega_n\right)^{-1}=\sup_n \omega_n^{-1}$
(since $\chi$ can be arbitrary).
\end{proof}
The following natural interpretation is also possible: 
The second part of Theorem \ref{T-C.4} and the order reversing property of 
$
\varphi \in P(\cM,\bC)
\leftrightarrow 
\varphi^{-1} 
\in
P(B(\cH), \cM')
$
yield that $(\Inf_n \omega_n)^{-1}$ is the minimum of all operator valued weights 
(in $P(B(\cH), \cM')$) majorizing all $\omega_n^{-1}$ ($n \in \bN$). However, the point-wise
supremum $\sup_n \omega_n^{-1}$ is already a normal operator valued weight so that
this must be the minimum in question.

%%%%%%%%%%%%%%%%%%%%%%%%%%%%%%%%%%%%%%%%%%%%%%%%%%%
%%%%%%%%% References %%%%%%%%%%%%%%%%%%%%%%%%%%%%%%%%%%%
%%%%%%%%%%%%%%%%%%%%%%%%%%%%%%%%%%%%%%%%%%%%%%%%%%%

\end{document}